\documentclass[11pt]{amsart}

\usepackage{amsmath}
\usepackage{amsfonts}
\usepackage{amssymb}
\usepackage{amscd}
\usepackage{amsthm}
\usepackage{framed}
\usepackage{fullpage}
\usepackage{graphicx}
\usepackage{latexsym}
\usepackage[numbers]{natbib}

\usepackage{hyperref}
\hypersetup{colorlinks,linkcolor={red},citecolor={blue},urlcolor={blue}}

\newtheorem{theorem}{Theorem}[section]
\newtheorem{lemma}[theorem]{Lemma}
\newtheorem{proposition}[theorem]{Proposition}
\newtheorem{corollary}[theorem]{Corollary}

\theoremstyle{definition}
\newtheorem{definition}[theorem]{Definition}

\newtheorem{remark}[theorem]{Remark}

\numberwithin{equation}{section}

\newcommand{\CC}{\mathbb C}
\newcommand{\HH}{\mathbb H}

\newcommand{\NN}{\mathbb N}

\newcommand{\QQ}{\mathbb Q}
\newcommand{\RR}{\mathbb R}
\newcommand{\ZZ}{\mathbb Z}

\newcommand{\Mp}{\mathop{\mathrm {Mp}}\nolimits}
\newcommand{\SL}{\mathop{\mathrm {SL}}\nolimits}

\newcommand{\Sp}{\mathop{\mathrm {Sp}}\nolimits}
\newcommand{\Orth}{\mathop{\null\mathrm {O}}\nolimits}

\newcommand{\latt}[1]{{\langle{#1}\rangle}}

\def\dim{\operatorname{dim}}

\newenvironment{psmallmatrix}
  {\left(\begin{smallmatrix}}
{\end{smallmatrix}\right)}

\begin{document}

\title[Jacobi forms of weight one on \texorpdfstring{$\Gamma_0(N)$}{}]{Jacobi forms of weight one on \texorpdfstring{$\mathbf{\Gamma_0(N)}$}{}} 

\author{Jialin Li}

\address{School of Mathematics and Statistics, Wuhan University, Wuhan 430072, Hubei, China}

\email{jlli.math@whu.edu.cn}

\author{Haowu Wang}

\address{School of Mathematics and Statistics, Wuhan University, Wuhan 430072, Hubei, China}

\email{haowu.wangmath@whu.edu.cn}


\date{\today}


\begin{abstract} 
Let $J_{1,m}(N)$ be the vector space of Jacobi forms of weight one and index $m$ on $\Gamma_0(N)$. In 1985, Skoruppa proved that $J_{1,m}(1)=0$ for all $m$. In 2007, Ibukiyama and Skoruppa proved that $J_{1,m}(N)=0$ for all $m$ and all squarefree $N$ with $\mathrm{gcd}(m,N)=1$. This paper aims to extend their results. We determine all levels $N$ separately, such that $J_{1,m}(N)=0$ for all $m$; or $J_{1,m}(N)=0$ for all $m$ with $\mathrm{gcd}(m,N)=1$. We also establish explicit dimension formulas of $J_{1,m}(N)$ when $m$ and $N$ are relatively prime or $m$ is squarefree. These results are obtained by refining Skoruppa's method and analyzing local invariants of Weil representations. As applications, we prove the vanishing of Siegel modular forms of degree two and weight one in some cases. 
\end{abstract}

\maketitle

\section{Introduction}
The theory of Jacobi forms was first systematically studied by Eichler--Zagier \cite{EZ85} and subsequently generalized in several contexts, such as Siegel--Jacobi forms \cite{Zie89}, Jacobi forms of lattice index \cite{Gri88}, and Jacobi forms over number fields \cite{Boy15}. Jacobi forms are an elegant intermediate between different types of modular forms. For example, they occur as Fourier--Jacobi coefficients of Siegel modular forms or modular forms on orthogonal groups $\Orth(n,2)$ \cite{Gri88}, and can be identified as vector-valued modular forms for the Weil representation of $\SL_2(\ZZ)$ or certain modular forms on congruence subgroups of $\SL_2(\ZZ)$ \cite{SZ88}. In recent years, Jacobi forms have attracted increasing attention due to their wide applications in mathematics and physics. For example, the elliptic genus of Calabi–Yau manifolds \cite{Gri88}, the topological partition function in string theory \cite{MNVW98}, and the denominator function of affine Lie algebras \cite{GSZ19} can all be expressed in Jacobi forms.

In this paper, we focus on classical Jacobi forms. A \emph{holomorphic Jacobi form} $\phi$ of integral weight $k$ and index $m$ for $\Gamma_0(N)$ is a holomorphic function on $\HH\times \CC$ that is invariant under the action of the Jacobi group $\Gamma_0^J(N):=\Gamma_0(N)\ltimes \ZZ^2$ and satisfies a certain boundary condition at every cusp of $\Gamma_0(N)$. More precisely, 
\begin{align*}
\big(\phi|_{k, m} A\big)(\tau, z):=&(c \tau+d)^{-k}  e\left(-\frac{ m c z^2}{c \tau+d}\right) \phi\left(\frac{a \tau+b}{c \tau+d}, \frac{z}{c \tau+d}\right)=\phi(\tau, z), \\
\phi(\tau, z)=&e\big(m(x^2 \tau+2 x z)\big) \phi(\tau, z+x \tau+y),
\end{align*}
for any $A=\begin{psmallmatrix}
a & b \\ c & d    
\end{psmallmatrix} \in \Gamma_0(N)$ and any $x, y \in \mathbb{Z}$, and  
$$
\big(\phi|_{k, m}g\big)(\tau, z)=\sum_{\substack{n \in \frac{1}{h} \mathbb{N},\; r \in \mathbb{Z} \\ 4nm-r^2 \geq 0}} f_g(n, r) q^n \zeta^r, \quad q=e^{2 \pi i \tau}, \; \zeta=e^{2 \pi i z},
$$
for some positive integer $h$. We denote the complex vector space of such functions by $J_{k,m}(N)$. According to \cite[p. 10]{EZ85}, $\phi$ vanishes if $k$ is negative and is a constant if $k=0$. This implies that if $\phi$ is not constant, its weight $k$ is at least $1$.

Jacobi forms of weight one seem to be exceptional. In 1985, Skoruppa \cite{Sko85} demonstrated that $J_{1,m}(1)=0$ for any index $m$. In 2007, Ibukiyama and Skoruppa  proved that $J_{1,m}(N)=0$ for any index $m$ and any squarefree level $N$ with $\mathrm{gcd}(m,N)=1$ (see \cite{IS07} and its correction). As a direct consequence, they showed that every Siegel modular form of degree $2$ and weight $1$ for the subgroup $\Gamma_0^{(2)}(N)$ of $\Sp_2(\ZZ)$ is zero if $N$ is squarefree.  In 2018, Cheng, Duncan and Harvey \cite{CDH18} confirmed the existence of Jacobi forms of weight one for $\Gamma_0(N)$, and established the non-existence for many combinations of index and level. For example, $J_{1,12}(36)\neq 0$ and $J_{1,2}(2^a)=J_{1,3}(4\cdot 3^a)=J_{1,4}(2^a)=0$ for any non-negative integer $a$. They further used these vanishing results to obtain uniqueness for mock Jacobi forms of weight one in umbral moonshine \cite{CDH14a,CDH14b}, which contributed to the proof of umbral moonshine conjecture given by Duncan, Griffin and Ono \cite{DGO15}. 

In this paper, we aim to extend the previous results and provide a detailed study of Jacobi forms of weight one. Our first main theorem is the optimal generalization of Skoruppa's result. 

\begin{theorem}[Theorem \ref{th:vanish_arbitrary_m}]\label{MTH1}
Suppose that $N=2^np_1^{\alpha_1}...p_s^{\alpha_s}\cdot q_1^{\beta_1}...q_t^{\beta_t}$,
where $p_i$, $q_j$ are distinct odd primes and $p_i=1\,\bmod 4$ for any $1\leq i\leq s$, $q_j=3\,\bmod 4$ for any $1\leq j\leq t$. Then $J_{1,m}(N)=0$ for any positive integer $m$ if and only if $0\leq n\leq 4$ and $0\leq \beta_j \leq 1$ for $1\leq j\leq t$.    
\end{theorem}

Our second main theorem is the optimal generalization of the Ibukiyama--Skoruppa result. 

\begin{theorem}[Theorem \ref{th:vanish_m_coprime}]\label{MTH2}
Suppose that $N=2^np_1^{\alpha_1}...p_s^{\alpha_s}\cdot q_1^{\beta_1}...q_t^{\beta_t}$, where $p_i$, $q_j$ are distinct odd primes and $p_i=1\,\bmod 4$ for any $1\leq i\leq s$, $q_j=3\,\bmod 4$ for any $1\leq j\leq t$. Then $J_{1,m}(N)=0$ for any positive integer $m$ with $\mathrm{gcd}(m,N)=1$ if and only if $0\leq n\leq 5$ and $0\leq \beta_j\leq 2$ for $1\leq j\leq t$.
\end{theorem}

In addition, we establish a sufficient and necessary condition for $J_{1,p}(N)=0$, where $p$ is a fixed prime (see Proposition \ref{prop:index-2} and Proposition \ref{prop:vanishcondition_J1_P(N)}). As a key step in proving Theorem \ref{MTH2}, we show in Proposition \ref{prop:dimension_induction} that for any positive integer $n$ and any prime $p$ with $\mathrm{gcd}(p,mN)=1$, the equation
$$
\dim J_{1,m}(p^n N) = \Big( \Big[\frac{n}{2}\Big]+1 \Big) \dim J_{1,m}(N)
$$
holds for $p$ odd and $n\leq 2$ or $p=2$ and $n\leq 5$, where $[x]$ denotes the integer part of $x$ for $x\in\QQ$. This greatly extends \cite[Theorem 1.2]{CDH18} which is a correction to \cite[Theorem 2]{IS07}. We remark that the bound of $n$ above is sharp. (See remark after Proposition \ref{prop:dimension_induction}).

Our third main theorem provides, for the first time, non-trivial explicit dimension formulas of $J_{1,m}(N)$ in some specific cases. 

\begin{theorem}[Theorem \ref{th:m_N_COPRIME} and Theorem \ref{th:m_squarefree}]\label{MTH3}
Let $m$ and $N$ be positive odd integers.
\begin{enumerate}
\item Suppose that $\mathrm{gcd}(m,N)=1$. For any non-negative integer $k$, certain precise dimension formulas of $J_{1,m}(2^k N)$, $J_{1,2m}(2^k N)$ and $J_{1,2^k m}(N)$ are given. 
\item Suppose that $m$ is squarefree. For any non-negative integer $k$, certain precise dimension formulas of $J_{1,m}(2^k N)$ and $J_{1,2m}(2^k N)$ are given. 
\end{enumerate}    
\end{theorem}

Since these formulas are quite long, we omit the details and only present one particular example. Let $p$ be an odd prime. For any non-negative integers $k$ and $a$,  
$$
\dim J_{1,p}(2^k p^a) = \frac{\big( \frac{-1}{p} \big)+1}{2} \mathrm{max}\Big\{0, \Big[ \frac{k-4}{2} \Big]\Big\}  + \frac{\big( \frac{-2}{p} \big)+1}{2} \mathrm{max}\Big\{0, \Big[ \frac{k-7}{2} \Big]\Big\},
$$
which is independent of $a$, where $\big( \frac{x}{y} \big)$ denotes the Kronecker symbol.  

We sketch the proof of the above main results. On the one hand, Jacobi forms of weight one can be viewed as vector-valued modular forms of weight $1/2$ for the Weil representation of $\SL_2(\ZZ)$. On the other hand, Serre and Stark \cite{SS77} proved that the only modular forms of weight $1/2$ on congruence subgroups of $\SL_2(\ZZ)$ are theta series, which can also be regarded as vector-valued modular forms for the Weil representation in some sense. By the two facts, Skoruppa \cite{Sko85, Sko08} can realize Jacobi forms of weight one as certain $\Gamma_0(N)$-invariants of the Weil representation associated with the direct sum of two kinds of finite quadratic forms (see Proposition \ref{prop:main-iso} for details). Recall that every finite quadratic form can be decomposed as a direct sum of its $p$-parts for primes $p$, which induces a decomposition of the invariants of the Weil representation into local parts. All the vanishing results in \cite{Sko85, IS07, CDH18} mentioned before were proved by analyzing the local irreducible components of the corresponding Weil representation. In this paper we instead compute and construct these local invariants explicitly in a direct way (see Section \ref{sec:invariants}). Surprisingly, this natural and fundamental approach is both feasible and efficient, allowing us to achieve the above results.

We present several applications of our main results to Siegel modular forms. Let us consider  arithmetic subgroups of $\Sp_2(\RR)$ of the form
$$
K_0(t,N)=\left\{ \begin{pmatrix} * & *t & * & * \\ * & * & * & */t \\ *N & *Nt & * & * \\ *Nt & *Nt & *t & * \end{pmatrix} \in \Sp_2(\QQ): \; \text{all $*$ are integers} \right\}.
$$
They are congruence subgroups of the paramodular group of level $t$. 
Note that $K_0(1,1)=\Sp_2(\ZZ)$, $K_0(1,N)=\Gamma_0^{(2)}(N)$ is the congruence subgroup of $\Sp_2(\ZZ)$, and $K_0(t,1)=K(t)$ is the paramodular group of level $t$. The group $\Sp_2(\RR)$ acts on the Siegel upper half-space 
$$
\mathbb{H}_2 = \left\{Z = \begin{psmallmatrix} \tau & z \\ z & \omega \end{psmallmatrix} \in \mathrm{Mat}_2(\CC) : \mathrm{Im}(Z)>0\right\}
$$
by M\"obius transformations. A \emph{Siegel modular form} of integral weight $k$ for $K_0(t,N)$ is a holomorphic function $F : \mathbb{H}_2 \rightarrow \mathbb{C}$ satisfying 
$$
F\left( A \cdot Z \right) = \mathrm{det}(cZ + d)^k F(Z), \quad \text{for all} \; A=\begin{psmallmatrix} a & b \\ c & d \end{psmallmatrix} \in K_0(t,N).
$$
Let $M_k(K_0(t,N))$ denote the complex vector space of Siegel modular forms of weight $k$ for $K_0(t,N)$.  
The function $F$ can be expanded into
$$
F(Z)=\sum_{m=0}^\infty \phi_m(\tau,z)e(mt\omega),
$$
where the $m$-th Fourier--Jacobi coefficient $\phi_m$ lies in $J_{k,mt}(N)$. From Theorem \ref{MTH1} we immediately deduce the following vanishing result. 

\begin{corollary}
Suppose that $N=2^np_1^{\alpha_1}...p_s^{\alpha_s}\cdot q_1^{\beta_1}...q_t^{\beta_t}$,
where $p_i$, $q_j$ are distinct odd primes and $p_i=1\,\bmod 4$ for any $1\leq i\leq s$, $q_j=3\,\bmod 4$ for any $1\leq j\leq t$. If $0\leq n\leq 4$ and $0\leq \beta_j \leq 1$ for $1\leq j\leq t$, then $M_1(K_0(t,N))=0$ for any positive integer $t$.        
\end{corollary}

We now assume that $t=1$. By the arguments in \cite[Section 1]{IS07}, the above $F$ vanishes if its Fourier--Jacobi coefficients $\phi_m$ vanish for all $m$ that are coprime to $N$. We then conclude from Theorem \ref{MTH2} the following vanishing result. 

\begin{corollary}
Suppose that $N=2^np_1^{\alpha_1}...p_s^{\alpha_s}\cdot q_1^{\beta_1}...q_t^{\beta_t}$,
where $p_i$, $q_j$ are distinct odd primes and $p_i=1\,\bmod 4$ for any $1\leq i\leq s$, $q_j=3\,\bmod 4$ for any $1\leq j\leq t$. If $0\leq n\leq 5$ and $0\leq \beta_j\leq 2$ for $1\leq j\leq t$, then $M_1(\Gamma_0^{(2)}(N))=0$.
\end{corollary}

It is still an open question whether $M_1(\Gamma_0^{(2)}(N))$ is always trivial. The answer seems to be yes, however Theorem \ref{MTH2} shows that this question is unlikely to be fully resolved via the Fourier--Jacobi expansion argument. We remark that Siegel modular forms of weight one for $\Gamma_0^{(2)}(N)$ with character do exist. Indeed, by \cite[Section 5.1 and Section 6.1]{WW23}, Siegel modular forms of weight one for $\Gamma_0^{(2)}(3)$ and $\Gamma_0^{(2)}(4)$ exist, and they can be constructed via the Borcherds additive lifts. 

This paper is organized as follows. In Section \ref{sec:Weil-rep} we review the Weil representation and the decomposition of the discriminant forms arising from the even positive-definite lattices of rank one.  In Section \ref{sec:Skoruppa} we introduce Jacobi forms and review Skoruppa's approach to study Jacobi forms of weight one.  In Section \ref{sec:invariants}, we compute local invariants of the Weil representation corresponding to Jacobi forms of weight one. The related results are formulated by seventeen lemmas. This is the most technical part of the paper.  In Section \ref{sec:existence} we prove the existence of Jacobi forms of weight one and construct them in terms of Jacobi theta functions in some cases. Section \ref{sec:non-existence} is devoted to the proofs of Theorem \ref{MTH1} and Theorem \ref{MTH2}. In Section \ref{sec:dimension} we formulate Theorem \ref{MTH3} in a precise form and give it a proof.

\section{Discriminant forms and Weil representations}\label{sec:Weil-rep}
In this section, we review the representations of the metaplectic double cover of $\SL_2(\ZZ)$ arising, via a construction of Weil, from discriminant forms. 

A discriminant form $(D, Q)$ is a finite abelian group $D$ together with a non-degenerate quadratic form $Q : D \to \QQ/\ZZ$, that is, a function satisfying the properties
\begin{itemize}
\item[(1)] $Q(ax)=a^2Q(x)$ for all $a\in\ZZ$, $x\in D$;
\item[(2)] $B(x,y):=Q(x+y)-Q(x)-Q(y)$ is a non-degenerate bilinear form. 
\end{itemize}
The level of $(D,Q)$ is the smallest positive integer $l$ such that $lQ=0$. Two discriminant forms are called isomorphic if there exists a map that is both an ismorphism between two groups and an isometry between two quadratic forms. 

Every even lattice $L$ with bilinear form $\latt{-,-}$ and dual lattice $L'$ induces a discriminant form $(L'/L,\latt{-,-}/2)$. Conversely, every discriminant form arises from this way (see \cite[Theorem 1.3.2]{Nik79}). Via this realization, 
we define the signature of $(D,Q)$ as 
$$
\mathrm{sign}(D):=\mathrm{sign}(L)+8\ZZ,
$$
where $L$ is an even lattice such that $(D,Q) \cong (L'/L,\latt{-,-}/2)$. The residue class $\mathrm{sign}(D)$ can be calculated by Milgram’s formula
\begin{equation}\label{eq:Milgram}
\sum_{\gamma \in D} e\big( Q(\gamma) \big) = \sqrt{|D|} \cdot e\big(\mathrm{sign}(D)/8  \big),  \end{equation}
where $e(x):=e^{2\pi i x}$ for any $x\in\CC$ as defined in the introduction. 

Let $(D^1,Q_1)$ and $(D^2,Q_2)$ be two discriminant forms. The direct sum $D^1\oplus D^2$ equipped with the quadratic form $D^1\oplus D^2 \to \QQ/\ZZ, \; \gamma_1+\gamma_2 \mapsto Q_1(\gamma_1)+Q_2(\gamma_2)$ defines a new discriminant form. As a finite abelian group, $D$ is the direct sum of its prime components, which induces a decomposition of $(D,Q)$ (see e.g. \cite{SC98} or \cite[Proposition 1.6]{Boy15}). Indeed, $D\cong \bigoplus_{p \mid N} D^p$ as discriminant forms, where $N$ is the order of $D$, $D^p$ is the subgroup of $D$ which contains all elements annihilated by a power of $p$, and $D^p$ is endowed with the induced quadratic form $Q|_{D^p}$. The component $(D^p, Q|_{D^p})$ is called the $p$-part of $(D,Q)$.

We recall the metaplectic group $\Mp_2(\RR)$, which is a double cover of $\SL_2(\RR)$. Let $\sqrt{z}$ denote the principal branch of the square root i.e. arg($\sqrt{z}$)$\in (-\pi/2,\pi/2]$. For 
$A=\begin{psmallmatrix}
    a & b \\ c & d
\end{psmallmatrix} \in \SL_2(\RR),$
let $\phi$ be a holomorphic function on $\HH$ satisfying $\phi(\tau)^2=c\tau+d$, $\tau\in \HH$. The elements of $\Mp_2(\RR)$ are pairs $(A,\phi(\tau))$.
The product of two elements of $\Mp_2(\RR)$ is given by
$$
\big(A,\phi_1(\tau)\big)\big(B,\phi_2(\tau)\big)=\big(AB,\phi_1(B\tau)\phi_2(\tau)\big).
$$
The map 
$$
A\longmapsto \widetilde{A}=\left(A,\sqrt{c\tau+d}\right)
$$
defines a locally isomorphic embedding of $\SL_2(\RR)$ into $\Mp_2(\RR)$. Let $\Mp_2(\ZZ)$ be the inverse image of $\SL_2(\ZZ)$ under the covering map $\Mp_2(\RR)\to \SL_2(\RR)$. The group $\Mp_2(\ZZ)$ has two generators 
$$
T=\left( \left(\begin{array}{cc}
1 & 1 \\ 
0 & 1
\end{array}   \right) ,1 \right) \quad \text{and} \quad  S=\left( \left(\begin{array}{cc}
0 & -1 \\ 
1 & 0
\end{array}   \right) ,\sqrt{\tau} \right).
$$
The center of $\Mp_2(\ZZ)$ is generated by  
$$
Z=\left(\left(\begin{array}{cc}
-1 & 0 \\ 
0 & -1
\end{array}  \right), i\right),
$$
and the relations $S^2=(ST)^3=Z$ hold. 

Let $\CC D$ denote be complex vector space of all formal linear combinations $\sum_{\gamma \in D} c_\gamma \mathbf{e}^\gamma$, where $\mathbf{e}^\gamma$ is a symbol and $c_\gamma\in\CC$. The Weil representation of $\Mp_2(\ZZ)$ on $\CC D$ is a unitary representation defined by the action of the generators as follows:
\begin{align*}
\rho_D(T)\mathbf{e}^\gamma &= e\big(Q(\gamma)\big)\mathbf{e}^\gamma,\\
\rho_D(S)\mathbf{e}^\gamma &= \frac{e\big(-\mathrm{sign}(D)/8\big)}{\sqrt{|D|}} \sum_{\beta\in D}e\big(-B(\gamma,\beta)\big)\mathbf{e}^\beta.
\end{align*}
We note that
\begin{equation}\label{eq:Z-action}
\rho_D(Z)\mathbf{e}^\gamma=e(-\mathrm{sign}(D)/4)\mathbf{e}^{-\gamma}
\end{equation}
and $Z^2=(I,-1)$. 
The action $\rho_D$ factors through $\SL_2(\ZZ)$ if and only if $\mathrm{sign}(D)$ is even. 

In this paper, we focus on the discriminant form 
$$
D_m:=\Big(\ZZ/2m\ZZ, \gamma\mapsto \frac{\gamma^2}{4m}\Big)
$$
arising from the even lattice $\ZZ$ with the bilinear form $2x^2$.  We further define the discriminant forms
$$
D_m(a):=\Big(\ZZ /2m\ZZ, \gamma \mapsto \frac{a\gamma^2}{4m}\Big) \quad \text{and} \quad L_n(b):=\Big( \ZZ/n\ZZ, \gamma \mapsto \frac{b\gamma^2}{n} \Big),
$$ 
Since these forms are non-degenerate, we have that $\mathrm{gcd}(a,2m)=1$, $\mathrm{gcd}(b,n)=1$, and $n$ is odd. 
We decompose $D_m$ as the direct sum of its $p$-parts:
$$
D_m\cong D_{m_2}(a_2)\oplus\bigoplus_{\mathrm{odd} \; p \mid m}L_{m_p}(a_p),
$$ 
where $m_p$ is the largest power of $p$ dividing $m$ for any prime $p$, $a_2 = (m/m_2)^{-1} \, \bmod \, 4m_2$, and $a_p = (4m/m_p)^{-1} \, \bmod \, m_p$ for any odd prime $p$. By \cite[Lemma 2.50]{Boy15}, this isomorphism induces the following isomorphism
\begin{equation}\label{eq:iso-local-global}
\begin{split}
\CC D_m&\cong \CC D_{m_2}(a_2)\otimes \bigotimes_{\mathrm{odd} \; p \mid m}\CC L_{m_p}(a_p),  \\
\mathbf{e}^{\gamma\,(\bmod 2m)}&\mapsto \mathbf{e}^{\gamma\,(\bmod 2m_2)}\otimes \bigotimes_{\mathrm{odd} \,p
\mid m}\mathbf{e}^{\gamma\,(\bmod m_p)}
\end{split}
\end{equation}
between left $\Mp_2(\ZZ)$-modules. 

Let $\Gamma(4m)^\ast $ be the subgroup of $\Mp_2(\ZZ)$ consisting of all pairs $(A,j(A,\tau))$ with $A \in \Gamma(4m)$ and $j(A,\tau):=\theta(A\tau)/\theta(\tau)$, where $\theta(\tau)=\sum_{n\in\ZZ}e(n^2\tau)$. It is known that $\CC D_m$ is invariant under $\Gamma(4m)^\ast$, while $\CC L_m(a)$, which is of even signature, is invariant under $\Gamma(m)$ (see \cite[Theorem 5.4]{Bor00} and \cite[Theorem 3.1]{Boy15}). Therefore, we may regard $\CC D_m$ as an $\Mp_2(\ZZ)/\Gamma(4m)^\ast$-module that is isomorphic to the outer tensor product of $\Mp_2(\ZZ)/\Gamma(4m_2)^\ast$-module $\CC D_{m_2}(a_2)$ and $\SL_2(\ZZ)/\Gamma(m_p)$-modules $\CC L_{m_p}(a_p)$ for $p\mid m$ (see \cite[Section 3]{Sko08}).

%For any positive integer $m$, the orthogonal group $\Orth_m:=\{t\in\ZZ/2m\ZZ: \, t^2=1 \,\bmod \, 4m \}$ acts on $\CC D_m$ via $t \cdot \mathbf{e}^\gamma=\mathbf{e}^{t\gamma}$ for any $\gamma \in D_m$ and $t \in \Orth_m$. This action commutes with the action of $\Mp_2(\ZZ)$. The group $\Orth_m$ is generated by $\xi_p$, where $p$ are prime divisors of $m$ and $\xi_p$ are defined via
%$$
%\xi_p = 1 \, \bmod \, 2m/m_p , \qquad \xi_p = -1 \, \bmod \, 2m_p.  
%$$

Let $\CC D_m^{\pm}(a)$ and $\CC L_m^{\pm}(b)$ be the subspaces generated by $\mathbf{e}^\gamma\pm \mathbf{e}^{-\gamma}$ for all $\gamma \in D_m$ and $\gamma\in L_m$, respectively. We then deduce that $\CC D_m=\CC D_m^+\oplus \CC D_m^-$ and 
$$
\CC D_m^{\pm} \cong \bigoplus_{\prod\epsilon_\ast=\pm1} \left( \CC D_{m_2}^{\epsilon_2}(a_2) \otimes \, \bigotimes_{\mathrm{odd} \, p\mid m} \CC L_{m_p}^{\epsilon_p}(a_p)\right), 
$$
where $\epsilon_*\in\{ +,-\}$,  the number of $-$ among all $\epsilon_*$ is odd for $\CC D_m^-$ and even for $\CC D_m^+$.

For any positive integers $m,d$ and any integer $a$ with $\mathrm{gcd}(a,2md)=1$,  there exists an injective homomorphism 
$$
\CC D_m(a) \to \CC D_{md^2}(a), \quad \mathbf{e}^\gamma \mapsto \sum_{\substack{x = \gamma d\, \bmod \, 2md}} \mathbf{e}^x
$$
between $\Mp_2(\ZZ)$-modules.
Therefore, $\CC D_m(a)$ can be seen as a submodule of $\CC D_{md^2}(a)$. Similarly, 
for any positive odd integers $m,d$ and any integer $a$ with $\mathrm{gcd}(a,md)=1$, $\CC L_{m}(a)$ is a submodule of $\CC L_{md^2}(a)$ by the $\SL_2(\ZZ)$-homomorphism  
$$
\CC L_m(a) \to \CC L_{md^2}(a), \quad \mathbf{e}^\gamma \mapsto \sum_{\substack{x = \gamma d\, \bmod \, md}} \mathbf{e}^x.
$$

\section{Jacobi forms and Skoruppa's approach}\label{sec:Skoruppa}
In this section, we introduce Jacobi forms \cite{EZ85} on the congruence subgroups $\Gamma_0(N)$ and review Skoruppa's method \cite{Sko85} to study the space of Jacobi forms of weight one. 

\begin{definition}
Let $k,m\geq 0$ and $N\geq 1$ be integers. A holomorphic function $\phi: \mathbb{H} \times \mathbb{C} \rightarrow \mathbb{C}$ is called a holomorphic Jacobi form of weight $k$ and index $m$ on $\Gamma_0(N)$ if it satisfies
$$
\big(\phi|_{k, m} A\big)(\tau, z):=(c \tau+d)^{-k}  e\left(-\frac{ m c z^2}{c \tau+d}\right) \phi\left(\frac{a \tau+b}{c \tau+d}, \frac{z}{c \tau+d}\right)=\phi(\tau, z)
$$
for any $A=\begin{psmallmatrix}
a & b \\ c & d    
\end{psmallmatrix} \in \Gamma_0(N)$ and
$$
e\big(m(x^2 \tau+2 x z)\big) \phi(\tau, z+x \tau+y)=\phi(\tau, z)
$$
for any $x, y \in \mathbb{Z}$, and if for every $g \in \mathrm{SL}_2(\mathbb{Z})$ the function $\phi|_{k, m} g$ has the Fourier expansion
$$
\big(\phi|_{k, m}g\big)(\tau, z)=\sum_{\substack{n \in \frac{1}{h} \mathbb{N},\; r \in \mathbb{Z} \\ 4nm-r^2 \geq 0}} f_g(n, r) q^n \zeta^r, \quad q=e^{2 \pi i \tau}, \; \zeta=e^{2 \pi i z},
$$
for some positive integer $h$. We denote the complex vector space of holomorphic Jacobi forms of weight $k$ and index $m$ on $\Gamma_0(N)$ by $J_{k, m}(N)$.    
\end{definition}

In this paper, we aim to study the space $J_{1,m}(N)$, that is, the complex vector space of the holomorphic Jacobi forms of minimal positive weight on $\Gamma_0(N)$. 

Let $\phi \in J_{1,m}(N)$. Then it has the theta decomposition
$$ 
\phi(\tau,z)=\sum_{\mu \, \bmod \,2m} h_{m,\mu}(\tau)\theta_{m,\mu}(\tau,z),
$$
where 
$$
\theta_{m,\mu}(\tau,z)=\sum_{\substack{r\in\ZZ \\ r=\mu \, \bmod 2m}}q^{\frac{r^2}{4m}}\zeta^r
$$
are the theta functions attached to the even positive-definite lattice $\ZZ$ with the bilinear form $2mx^2$, which are modular forms of weight $1/2$ on $\Gamma(4m)^\ast$. 
Let $\Theta_m$ denote the complex vector space spanned by $\theta_{m,\mu}$ for all $\mu \bmod \,2m$. 
Since $\Theta_m$ is invariant under $\Gamma(4m)^\ast$, the theta coefficients  $h_{m,\mu}$ are modular forms of weight $1/2$ on $\Gamma(4M)^\ast$,
where $M$ is the smallest positive integer such that $m \mid M$ and $N \mid 4M$. It is well known that the space $M_{1/2}(\Gamma(4M)^\ast)$ of modular forms of weight $1/2$ on $\Gamma(4M)^\ast$ is generated by theta series (see \cite{SS77}). More precisely, Skoruppa \cite[pp. 101]{Sko85} decomposed $M_{1/2}(\Gamma(4M)^\ast)$ as follows 
$$ 
M_{1/2}(\Gamma(4M)^\ast) \cong \bigoplus_{\substack{m' \mid M \\ M/m' \, \text{squarefree} }} \Theta_{m'}^+,
$$
where $\Theta_{m'}^\pm $ is generated by  $ \theta_{m',\mu}\pm \theta_{m',-\mu}$ for all $\mu \bmod \, 2m'$.

When a group $G$ acts on a module $V$, the invariant submodule $V^G$ is defined as the set of elements $v \in V$ such that $gv=v$ for all $g \in G$. For any subgroup $H$ of $G$, we denote the underlying space with
respect to the restricted representation of group $H$ by  $\mathrm{Res}^G_HV$. And we write $(\mathrm{Res}^G_H V)^H$ as $V^H$ with slight ambiguity. As a right $\Mp_2(\ZZ)$-module, $\Theta_m$ satisfies the following transformation laws:
\begin{align*}
\theta_{m,\mu}|_{1/2,m}T(\tau,z)&=e\Big(\frac{\mu^2}{4m}\Big)\theta_{m,\mu}(\tau,z),\\
\theta_{m,\mu}|_{1/2,m}S(\tau,z)&=\frac{1}{\sqrt{2m}}e^{-\pi i/4}\sum_{\nu \, \bmod \, 2m} e\Big(\frac{-\mu \nu}{2m}\Big)\theta_{m,\nu}(\tau,z).
\end{align*}
Then one derives that
\begin{equation}\label{eq:theta-decomposition}
J_{1,m}(N) \cong \bigoplus_{\substack{m'\mid M \\ M/m' \, \text{squarefree}}}(\Theta_m\otimes\Theta_{m'}^{+})^{\Gamma_0(N)}.
\end{equation}

Although the group $\Mp_2(\ZZ)$ has the same action on $\Theta_m$ as it acts on $\CC D_m$, $AB \mathbf{e}^\gamma=A(B\mathbf{e}^\gamma)$  while  $\theta_{m,\mu}|_{1/2,m}AB=(\theta_{m,\mu}|_{1/2,m}A)|_{1/2,m}B$ for any $A, B \in \Mp_2(\ZZ)$ and $\gamma \in D_m$. For the left $\Mp_2(\ZZ)$-module $\CC D_m$, we define a right $\Mp_2(\ZZ)$-module $\CC D_m^\ast$ with underlying space  $\CC D_m$ and action $\mathbf{e}^\gamma A = A^{-1} \mathbf{e}^\gamma$. Since the Weil representation is unitary, it follows that
\begin{equation}\label{eq:identification}
\Theta_m\cong \CC D_m^\ast \cong \CC D_m(-1).
\end{equation}
Therefore, $(\Theta_m\otimes\Theta_{m'}^{+})^{\Gamma_0(N)}$ is isomorphic to $(\CC D_m(-1)\otimes \CC D_{m'}^+(-1))^{\Gamma_0(N)}$ as vector spaces. 
By considering the action of $Z=(-I,i)\in \Mp_2(\ZZ)$, we deduce that $(\CC D_{m}^+(-1) \otimes \CC D_{m'}^+(-1) )^{\Gamma_0(N)} =0$. Combining \eqref{eq:theta-decomposition} with the above facts we derive the following result. 

\begin{proposition}\label{prop:main-iso}
Let $m$ and $N$ be positive integers and let $M$ be the smallest positive integer such that $m \mid M$ and $N \mid 4M$. Then the following isomorphism hold:
\begin{equation}\label{iso_J_1_m_N}
J_{1,m}(N) \cong \bigoplus_{\substack{m'\mid M \\ M/m' \, \mathrm{squarefree}}} \big( \CC D_m^{-}(-1)\otimes\CC D_{m'}^{+}(-1)\big)^{\Gamma_0(N)}. 
\end{equation}
Recall that $\CC D_m(a)$ and $\CC L_m(a)$ factor through $\Mp_2(\ZZ)/\Gamma(4m)^\ast$ and $\SL_2(\ZZ)/\Gamma(m)$, respectively. Hence the vector space  $( \CC D_m^{-}(-1)\otimes\CC D_{m'}^{+}(-1))^{\Gamma_0(N)}$ is isomorphic to
\begin{equation}\label{eq:parts}
\bigoplus_{\substack{\prod \epsilon_*=- \\ \prod \epsilon'_*=+}}\left( \big(\CC D_{m_2}^{\epsilon_2}(a_2)\otimes \CC D_{m'_2}^{\epsilon_2'}(a_2') \big)^{\Gamma_0(N_2)}\otimes\bigotimes_{\mathrm{odd} \; p \mid M} \big(\CC L_{m_p}^{\epsilon_p}(a_p)\otimes \CC L_{m'_p}^{\epsilon_p'}(a_p')\big)^{\Gamma_0(N_p)}\right),
\end{equation}
where $N_p$, $m_p$ and $m'_p$ are the largest power of $p$ dividing $N$, $m$ and $m'$ for any prime $p$, respectively; $a_2=-(m/m_2)^{-1} \, \bmod \, 4m_2$, and $a_p=-(4m/m_p)^{-1}\, \bmod \, m_p$ for any odd prime $p$; $a_2'$ and $a_p'$ are defined in a similar way; $\epsilon_*$ and $\epsilon'_*$ take $+$ or $-$, $\prod \epsilon_*=-$ means that the number of $-$ among all $\epsilon_*$ is odd, and $\prod \epsilon'_*=+$ means that the number of $-$ among all $\epsilon'_*$ is even.
\end{proposition}

We remark that this proposition is essentially a special case of \cite[Theorem 8]{Sko08}. 

For any non-trivial subspace of $J_{1,m}(N)$ appeared in \eqref{eq:parts}, we call $(\CC D_{m_2}^{\epsilon_2}(a_2)\otimes \CC D_{m'_2}^{\epsilon_2'}(a_2') )^{\Gamma_0(N_2)}$ a $2$-part of $J_{1,m}(N)$ and $(\CC L_{m_p}^{\epsilon_p}(a_p)\otimes \CC L_{m'_p}^{\epsilon_p'}(a_p'))^{\Gamma_0(N_p)}$ a $p$-part of $J_{1,m}(N)$. In some sense, $2$-parts and $p$-parts can be regarded as local inviriants of the Weil representation associated with the discriminant form $D_m$.

\section{Local invariants of Weil representations}\label{sec:invariants}
In this section, we calculate the dimension of the local parts $(\CC D_{m_2}^{\epsilon_2}(a_2)\otimes \CC D_{m'_2}^{\epsilon_2'}(a_2') )^{\Gamma_0(N_2)}$ and $(\CC L_{m_p}^{\epsilon_p}(a_p)\otimes \CC L_{m'_p}^{\epsilon_p'}(a_p'))^{\Gamma_0(N_p)}$, and construct an explicit basis if the space is non-trivial. 

We first review the methods used in \cite{Sko08} and \cite[Section 3.3]{CDH18}. Let $\chi_{m}^\epsilon: \mathrm{Mp}_2(\mathbb{Z}) \to \mathbb{C}^*$ be the character of $\mathbb{C} D_m^\epsilon$, defined as the trace $\operatorname{tr}(\rho(\gamma))$ of the Weil representation $\rho$ on $\mathbb{C} D_m^\epsilon$ for $\gamma \in \mathrm{Mp}_2(\mathbb{Z})$.
The character of $\CC D_m^\epsilon(a)$ equals $\sigma(\chi_{m}^\epsilon)$, where $\sigma$ is the Galois automorphism of $\QQ\big(e(\frac{1}{4m})\big)$ mapping $e(\frac{1}{4m})$ to $e(\frac{a}{4m})$. Then for $\Mp_2(\ZZ)$-module $\CC D_{m}^\epsilon(a)\otimes \CC D_{m'}^{\epsilon'}(a')$, we have
$$ 
\mathrm{dim}\big(\CC D_{m}^\epsilon(a)\otimes \CC D_{m'}^{\epsilon'}(a')\big)^{\Gamma_0(N)}=\big\langle \sigma(\chi_m^\epsilon)\sigma'(\chi_{m'}^{\epsilon'}),\;  \mathrm{Ind}_{\Gamma_0(N)}^{\SL_2(\ZZ)}1_{\Gamma_0(N)}\big\rangle.  
$$
where $\mathrm{Ind}_{\Gamma_0(N)}^{\mathrm{SL}_2(\mathbb{Z})}1_{\Gamma_0(N)}$ is the representation induced from the trivial representation of $\Gamma_0(N)$ to $\mathrm{SL}_2(\mathbb{Z})$, whose underlying space comprises complex-valued functions on $\mathrm{SL}_2(\mathbb{Z})$ that are invariant under the left action of $\Gamma_0(N)$, and the group acts by right translation.

When $n=1$, it is simplified as
$$
\big\langle \sigma(\chi_m^\epsilon)\sigma'(\chi_{m'}^{\epsilon'}), 1\big\rangle=\big\langle \sigma(\chi_m^\epsilon), \overline{ \sigma'(\chi_{m'}^{\epsilon'}) } \big\rangle. 
$$
For a finite group $G$, the inner product of two irreducible characters, $\langle \chi_V, \chi_W \rangle$, is either 0 or 1. It equals 1 if and only if $V$ and $W$ are isomorphic representations; otherwise, it is 0.

Thus the dimension is equal to the multiplicity of the trivial representation in the tensor product $\sigma(\chi_m^\epsilon) \otimes \sigma'(\overline{\chi_{m'}^{\epsilon'}})$, or equivalently, the number of isomorphic irreducible components (counting multiplicity) in the decompositions of $\mathbb{C} D_m(a)$ and $\mathbb{C} D_{m'}(-a')$. We denote by $\NN$ the set of non-negative integers. 

\begin{lemma}\label{lem:irreducible_D2^N_LP^N}
Let $k_1,k_2\in \NN$, $p$ be a prime, and $a_1,a_2$ be integers with $\mathrm{gcd}(a_1a_2,p)=1$.  

\vspace{2mm}

\noindent
$\mathbf{(I)}$ If $p=2$, then $( \CC D_{2^{k_1}}^{\epsilon}(a_1)\otimes \CC D_{2^{k_2}}^{\epsilon'} (a_2) )^{\SL_2(\ZZ)} \neq 0$ if and only if it satisfies:
\begin{enumerate}
\item If $k_1, k_2>0$, then $\epsilon=\epsilon'$. If one of $k_1$ and $k_2$ is $0$, then $\epsilon=\epsilon'=+$.
\item $k_1+k_2$ is even.
\item If $k_1$ and $k_2$ are odd, then $a_1a_2=-1 \, \bmod \, 8$. If $k_1$ and $k_2$ are even and $\epsilon=\epsilon'=+$, then $a_1a_2=-1 \, \bmod \,4$. If $k_1$ and $k_2$ are even and $\epsilon=\epsilon'=-$, then $a_1a_2=-1 \, \bmod \, 8$.
\end{enumerate}
Moreover, for any integer $k\geq 0$, we have
\begin{align*}
\dim \big( \CC D_{2^{2k+1}}^{\pm}(a_1)\otimes \CC D_{2^{2k+1}}^{\pm} (a_2) \big)^{\SL_2(\ZZ)}&=\left\{ \begin{array}{ll}
0, & \text{if} \; a_1a_2\neq 7\bmod 8 \\
k+1, & \text{if} \; a_1a_2= 7\bmod 8
\end{array}\right. \\
\dim \big( \CC D_{2^{2k}}^{-}(a_1)\otimes \CC D_{2^{2k}}^{-} (a_2) \big)^{\SL_2(\ZZ)}&=\left\{ \begin{array}{ll}
0, & \text{if} \; a_1a_2\neq 7\bmod 8 \\
k,\quad \enspace & \text{if} \; a_1a_2= 7\bmod 8
\end{array}\right.\\
\dim \big( \CC D_{2^{2k}}^{+}(a_1)\otimes \CC D_{2^{2k}}^{+} (a_2) \big)^{\SL_2(\ZZ)}&=\left\{ \begin{array}{ll}
1, & \text{if} \;  a_1a_2=3\bmod 8 \\
k+1, & \text{if} \;  a_1a_2= 7\bmod 8.
\end{array}\right.
\end{align*} 

\vspace{3mm}

\noindent
$\mathbf{(II)}$ If $p$ is odd, then $( \CC L_{p^{k_1}}^{\epsilon}(a_1)\otimes \CC L_{p^{k_2}}^{\epsilon'} (a_2) )^{\SL_2(\ZZ)} \neq 0$ if and only if it satisfies: 
\begin{enumerate}
\item If $k_1,k_2>0$, then $\epsilon=\epsilon'$.  If one of $k_1$ and $k_2$ is $0$, then $\epsilon=\epsilon'=+$.
\item $k_1+k_2$ is even.
\item If either $k_1, k_2$ are odd, or $k_1, k_2$ are even and $\epsilon=\epsilon'=-$, then $-a_1a_2$ is a square modulo $p$.
\end{enumerate}
Moreover, for any integer $k\geq 0$, we have
\begin{align*}
\dim \big( \CC L_{p^{2k+1}}^{\pm}(a_1)\otimes \CC L_{p^{2k+1}}^{\pm} (a_2) \big)^{\SL_2(\ZZ)}&=\left\{ \begin{array}{ll}
0, & \text{if} \; \big(\frac{-a_1a_2}{p}\big)=-1 \\
k+1, & \text{if} \;  \big(\frac{-a_1a_2}{p}\big)=1
\end{array}\right. \\
\dim \big( \CC L_{p^{2k}}^{-}(a_1)\otimes \CC L_{p^{2k}}^{-} (a_2) \big)^{\SL_2(\ZZ)}&=\left\{ \begin{array}{ll}
0, & \text{if} \; \big(\frac{-a_1a_2}{p}\big)=-1 \\
k, \quad \enspace & \text{if} \;  \big(\frac{-a_1a_2}{p}\big)=1
\end{array}\right. \\
\dim \big( \CC L_{p^{2k}}^{+}(a_1)\otimes \CC L_{p^{2k}}^{+} (a_2) \big)^{\SL_2(\ZZ)}&=\left\{ \begin{array}{ll}
1, & \text{if} \; \big(\frac{-a_1a_2}{p}\big)=-1 \\
k+1, & \text{if} \;  \big(\frac{-a_1a_2}{p}\big)=1,
\end{array}\right. 
\end{align*}
where $(\frac{x}{y})$ denotes the Kronecker symbol. 
\end{lemma}

\begin{proof}
We first prove the even case. Let $\CC D_{2^n}^{\epsilon,\mathrm{new}}(a)$ denote the orthogonal complement of $\CC D_{2^{n-2}}^{\epsilon}(a)$ with respect to the $\Mp_2(\ZZ)$-invariant scalar product given by $\latt{\mathbf{e}^\alpha,\mathbf{e}^\beta}=\delta_{\alpha\beta}$. Let $[x]$ denote the integer part of $x$. Then $\CC D_{2^n}^\epsilon(a)$ has the decomposition
$$ 
\CC D_{2^n}^\epsilon(a)=\bigoplus_{k=0}^{[n/2]} \CC D_{2^{n-2k}}^{\epsilon,\mathrm{new}}(a). 
$$
By \cite[Lemma 3 ]{Sko08}, $\CC D_{2^n}^{\epsilon,\mathrm{new}}(a)$ is irreducible, and $\CC D_{2^n}^{\epsilon,\mathrm{new}}(a)$ is isomorphic to $\CC  D_{2^{n'}}^{\epsilon',\mathrm{new}}(a') $ if and only if $n=n'$, $\epsilon=\epsilon'$ and $aa'$ is a square modulo $2^{n+2}$.
Hence $\CC D_{2^n}^{\epsilon,\mathrm{new}}(a)\cong \CC  D_{2^{n'}}^{\epsilon',\mathrm{new}}(a')   $ implies $\CC D_{2^{n-2k}}^{\epsilon,\mathrm{new}}(a)\cong \CC  D_{2^{n'-2k}}^{\epsilon',\mathrm{new}}(a')$ for any positive integer $k$ with $n,n'\geq 2k$.
Therefore, 	$\CC D_{2^{k_1}}^{\epsilon_1}(a_1) $ and $\CC D_{2^{k_2}}^{\epsilon_2} (-a_2)$ have isomorphic irreducible components if and only if $k_1+k_2$ is even, $\epsilon=\epsilon'$ and $\CC D_{2^{m}}^{\epsilon}(a_1)\cong \CC D_{2^{m}}^{\epsilon'} (-a_2)$, where $m$ is the smallest integer such that $\CC D_{2^{m}}^{\epsilon}(a_1)$ and $\CC D_{2^{m}}^{\epsilon'}(-a_2)$ are non-empty, and that $k_1-m$ and $k_2-m$ are non-negative and even. More precisely, 
$$ 
m=\left\{ \begin{array}{ll}
1, & \text{if} \; 2\nmid k_1 \\ 0, & \text{if} \; 2\mid k_1 \; \text{and} \; \epsilon=+ \\ 2, & \text{if} \; 2\mid k_1 \; \mathrm{and} \; \epsilon=-.
\end{array}\right. 
$$
Moreover, $\CC D_{2^{m}}^{\epsilon}(a_1)\cong \CC D_{2^{m}}^{\epsilon} (-a_2)$ as irreducible components if and only if 
$$ 
-a_1a_2 =  \left\{ \begin{array}{ll}
 1 \; (\bmod 8), & \quad \text{if} \; 2\nmid k_1 \\ 1\; (\bmod 4), &\quad \text{if} \;  2\mid k_1 \; \text{and} \; \epsilon=+ \\ 1,9\; (\bmod 16), &\quad \text{if} \; 2\mid k_1 \; \text{and} \; \epsilon=-.
\end{array}\right.
$$ 
This proves the even case. 

We then prove the odd case.  
Let $\CC L_{p^n}^{\epsilon,\mathrm{new}}(a)$ denote the orthogonal complement of $\CC L_{p^{n-2}}^{\epsilon}(a)$ with respect to the $\SL_2(\ZZ)$-invariant scalar product given by $\latt{\mathbf{e}^\alpha,\mathbf{e}^\beta}=\delta_{\alpha\beta}$. Similarly, $\CC L_{p^n}^\epsilon(a)$ has the decomposition
$$ 
\CC L_{p^n}^\epsilon(a)=\bigoplus_{k=0}^{[n/2]} \CC L_{p^{n-2k}}^{\epsilon,\mathrm{new}}(a). 
$$
According to \cite[Lemma 2]{Sko08}, $\CC L_{p^n}^{\epsilon,\mathrm{new}}(a)$ is irreducible, and $\CC L_{p^n}^{\epsilon,\mathrm{new}}(a)$ is isomorphic to $\CC  L_{p^{n'}}^{\epsilon',\mathrm{new}}(a') $ if and only if $n=n'$, $\epsilon=\epsilon'$ and $aa'$ is a square modulo $p^{n}$. Therefore, $\CC L_{p^{k_1}}^{\epsilon}(a_1) $ and $\CC L_{p^{k_2}}^{\epsilon'} (-a_2)$ have isomorphic irreducible components if and only if $k_1+k_2$ is even, $\epsilon=\epsilon'$ and $\CC L_{p^{m}}^{\epsilon}(a_1)\cong \CC L_{p^{m}}^{\epsilon'} (-a_2)$, where $m$ is the smallest integer such that $\CC L_{p^{m}}^{\epsilon}(a_1)$ and $\CC L_{p^{m}}^{\epsilon'}(-a_2)$ are non-empty, and $k_1-m$ and $k_2-m$ are non-negative and even. In fact, 
$$ 
m=\left\{ \begin{array}{ll}
 1, & \text{if} \; 2\nmid k_1 \\ 0, & \text{if} \; 2\mid k_1 \; \text{and} \; \epsilon=+ \\ 2, & \text{if} \; 2\mid k_1 \; \text{and} \; \epsilon=- 
\end{array}\right. 
$$
This proves the odd case in a similar way. 
\end{proof}

It is not difficult to derive the following lemma from \cite[Lemma 3.9]{CDH18}. 
\begin{lemma}\label{lem:CD16-+}
Suppose that $k_1,k_2\in\{0,1,2,3,4\}$ and $a_1,a_2$ are odd integers with $a_1a_2$ being a square modulo $4$. Then
$$
\big(\CC D_{2^{k_1}}^-(a_1)\otimes \CC D_{2^{k_2}}^+(a_2)\big)^{\Gamma_0(16)}=0.
$$
\end{lemma}

When the level $N$ is greater than one, finding isomorphic irreducible components or determining the inner product of characters requires a lot of computational work. To determine the space $J_{1,m}(N)$, we instead compute their $p$-parts for all primes $p$. 

Let us first review two basic facts that we will use later. For a vector space $V$ acted on by a finite group $G$, we have $V^G=e\cdot V$, where $e=\frac{1}{|G|}\sum_{g\in G} g \in \CC[G]$. We also need the following coset decomposition:
\begin{equation}\label{coset_decomposition}
\Gamma_0(p^n) =\bigsqcup_{\substack{ l \, \bmod p^n\\ \mathrm{gcd}(l,p)=1}}\bigsqcup_{j\, \bmod p^n}ST^l ST^{l^{-1}}ST^j \, \Gamma(p^n),
\end{equation} 
where $p$ is a prime and $l^{-1}$ is an inverse of $l$ modulo $p^n$. 

Recall that $\mathrm{sign}(D_m)=1 \, \bmod 8$. The next lemma confirms the signatures of the discriminant forms $D_{2^n}(a)$ and $L_{p^n}(a)$. 

\begin{lemma}\label{lem:sign}
Let $p$ be an odd prime, $n,m\in\NN$, and $a, b\in \ZZ$ with $\mathrm{gcd}(2,a)=1$ and $\mathrm{gcd}(p,b)=1$. Then we have
$$ 
e\left(\frac{\mathrm{sign}(D_{2^n}(a))}{8}\right)=\left\{  \begin{array}{ll} e(\frac{a}{8}), & \text{if}\; 2\nmid n \\ \big(1+e(\frac{a}{4})\big)/\sqrt{2}, & \text{if}\; 2\mid n
\end{array} \right. 
$$
and
$$ 
e\left(\frac{\mathrm{sign}(L_{p^m}(b))}{8}\right)=\left\{\begin{array}{cl}\big(\frac{b}{p}\big) i,  & \text{if}\; p = 3\,(\bmod 4)\; \text{and}\; 2\nmid m \\ \big(\frac{b}{p}\big),  & \text{if}\; p = 1\,(\bmod 4)\; \text{and}\; 2\nmid m \\ 1, & \text{if}\; 2 \mid m\end{array}\right. 
$$
where $\left(\frac{x}{y}\right)$ denotes the Kronecker symbol as before.
\end{lemma}
\begin{proof}
The proof follows from Milgram's formula \eqref{eq:Milgram} and the following identities of Gauss sums (see e.g. \cite[pp. 85--90]{Lan94})
\begin{align*}
\sum_{k=1}^{2^l} e\left(\frac{a k^2}{2^l}\right)&=\left\{\begin{array}{cl}2^{\frac{l+1}{2}} e\big(\frac{a}{8}\big), & \text{if}\; 2\nmid l \\ 2^{\frac{l}{2}}\big(1+e\left(\frac{a}{4}\right)\big), & \text{if}\; 2\mid l\end{array}\right.\\
\sum_{k=1}^{p^m} e\left(\frac{b k^2}{p^m}\right)&=\left\{\begin{array}{cl}\big(\frac{b}{p}\big) ip^{\frac{m}{2}},  & \text{if}\; p = 3\,(\bmod 4)\; \text{ and }\; 2\nmid m \\ \big(\frac{b}{p}\big)p^{\frac{m}{2}},   & \text{if}\; p = 1\,(\bmod 4)\; \text{and}\; 2\nmid m \\ p^{\frac{m}{2}},  & \text{if}\; 2 \mid m\end{array}\right.	
\end{align*}
where $l>1$ and $m\geq 0$. 
\end{proof}

The following two lemmas describe the action of certain specific elements of $\Mp_2(\ZZ)$ on $\CC L_{p^n}(a)$ and $\CC D_{2^n}(a)$ respectively. We refer to \cite[Theorem 4.7]{Scheithauer09} and \cite[Theorem 3.1]{Boy15} for more general formulas of this kind.

\begin{lemma}\label{lem:oddformula}
Let $p$ be an odd prime and $a$ be an integer that is coprime to $p$. Fix $\gamma \in L_{p^n}(a)$. For any integer $m$, we have $ST^mS=\begin{psmallmatrix}
-1 & 0 \\ m & -1    
\end{psmallmatrix}$ and 
\begin{equation*}
ST^mS \mathbf{e}^\gamma = \left\{ \begin{aligned}
-\frac{1}{p^n}  \sum_{\alpha,\, \beta \in L_{p^n}(a)}e\left( \frac{a(-2\gamma \alpha-2\alpha \beta+m\alpha^2)}{p^n}\right)\mathbf{e}^\beta, & \quad \text{if}\; p=3\; \bmod\,4 \; \text{and} \; 2\nmid n, \\ \frac{1}{p^n}  \sum_{\alpha,\, \beta \in L_{p^n}(a)}e\left( \frac{a(-2\gamma \alpha-2\alpha \beta+m\alpha^2)}{p^n}\right)\mathbf{e}^\beta, &\quad   \text{otherwise}.
\end{aligned}\right.
\end{equation*}
If there exists $m'\in\ZZ$ such that $mm'=1 \; \bmod \; p^n$, then $ST^{m'}ST^mS=\begin{psmallmatrix}
-m & 1 \\ 0 & -m'    
\end{psmallmatrix} \; \bmod \; p^n$ and 
\begin{equation}\label{eq:oddformulaSTSTS}
ST^{m'}ST^{m}S \mathbf{e}^\gamma= \left\{ \begin{array}{ll}
e\left(\frac{-a m' \gamma^2}{p^ n}\right) \mathbf{e}^{-m'\gamma}, & \quad \text{if}\; 2 \mid n,\\  \left(\frac{-m}{p}\right)e\left(\frac{- am' \gamma^2}{p^ n}\right) \mathbf{e}^{-m'\gamma}, & \quad \text{if}\; 2 \nmid n. 
\end{array} \right. 		
\end{equation}	
\end{lemma}

\begin{proof}
By definition, we have
$$
ST^mS \mathbf{e}^\gamma=\frac{e(-\mathrm{sign}(L_{p^n}(a))/4)}{p^n} \sum_{\alpha,\, \beta \in L_{p^n}(a)}e\left( \frac{a(-2\gamma \alpha-2\alpha \beta+m\alpha^2)}{p^n}\right)\mathbf{e}^\beta. 
$$
By Lemma \ref{lem:sign}, we find   
$$
e\Big(\frac{-\mathrm{sign}(L_{p^n}(a))}{4}\Big)=\left(\frac{-1}{p^n}\right). 
$$
The first formula then follows from the two identities above. 

If $\mathrm{gcd}(m,p)=1$ and $mm' =1 \; \bmod \; p^n$, then
\begin{align*}
ST^mS \mathbf{e}^\gamma&=\frac{1}{p^n}\left(\frac{-1}{p^n}\right) \sum_{\beta\in L_{p^n}(a)}\sum_{\alpha\in L_{p^n}(a)}e\left(\frac{ma(\alpha-m'(\gamma+\beta))^2}{p^n}\right)e\left(\frac{-am'(\gamma+\beta)^2}{p^n}\right) \mathbf{e}^\beta \\ & =\left(\frac{m}{p^n}\right)\frac{e(-\mathrm{sign}(L_{p^n}(a))/8)}{p^{\frac{n}{2}}} \sum_{\beta \in L_{p^n}(a)}e\left(\frac{-am'(\gamma+\beta)^2}{p^n}\right) \mathbf{e}^\beta. 
\end{align*}

It follows that 
$$
T^{m'}ST^{m}S \mathbf{e}^\gamma=e\left(\frac{-am'\gamma^2}{p^n}\right)\left(\frac{m}{p^n}\right)\frac{e(-\mathrm{sign}(L_{p^n}(a))/8)}{p^{\frac{n}{2}}}\sum_{\beta \in L_{p^n}(a)}e\left(\frac{-2am'\gamma\beta}{p^n}\right) \mathbf{e}^\beta,
$$
and thus
\begin{align*}
ST^{m'}ST^{m}S \mathbf{e}^\gamma&=e\left(\frac{-am'\gamma^2}{p^n}\right)\left(\frac{-m}{p^n}\right)\frac{1}{p^n}\sum_{\alpha,\, \beta \in L_{p^n}(a)}e\left(\frac{-2am'\gamma\alpha-2a\alpha\beta}{p^n}\right) \mathbf{e}^\beta\\
&=e\left(\frac{-am'\gamma^2}{p^n}\right)\left(\frac{-m}{p^n}\right)\mathbf{e}^{ -m'\gamma}.
\end{align*}
This proves the second formula. 
\end{proof}

The second formula above can also be deduced from \cite[Proposition 4.2]{Scheithauer09}.

\begin{lemma}\label{lem:evenformula}
Let $\gamma \in D_{2^n}(a)$ with $a$ being odd. For any $m\in \ZZ$, we have
\begin{equation*}
ST^m S \mathbf{e}^\gamma=\frac{(-i)^a}{2^{n+1}}\sum_{\alpha,\, \beta\in D_{2^n}(a)}e\left(\frac{a(-2\gamma\alpha-2\alpha\beta+m\alpha^2)}{2^{n+2}} \right)\mathbf{e}^\beta. 
\end{equation*}
If $m$ is odd and $mm' =1 \, \bmod \,2^{n+2}$, then
\begin{equation}\label{eq:evenformulaSTSTS}
ST^{m'}ST^mS \mathbf{e}^\gamma=\left\{ \begin{array}{ll} \frac{\left(1+e\left(\frac{-a}{4}\right)\right)^3\left(1+e\left(\frac{am}{4}\right)\right)}{4}  e\left(\frac{-am'\gamma^2}{2^{n+2}}\right) \mathbf{e}^{-m'\gamma}, & \text{if}\; 2 \mid n, \\
e\left(-\frac{3a}{8}\right) e\left(\frac{am}{8}\right) e\left(\frac{-am'\gamma^2 }{2^{n+2}}\right) \mathbf{e}^{-m'\gamma}, & \text{if}\; 2\nmid n. \end{array}\right.
\end{equation}
\end{lemma}
\begin{proof}
Similarly, we have
$$
ST^mS \mathbf{e}^\gamma=\frac{e(-\mathrm{sign}(D_{2^n}(a))/4)}{2^{n+1}} \sum_{\alpha,\,\beta \in D_{2^n}(a)}e\left( \frac{a(-2\gamma \alpha-2\alpha \beta+m\alpha^2)}{2^{n+2}}\right)\mathbf{e}^\beta
$$
and
$$
e\big(-\mathrm{sign}(D_{2^n}(a))/4\big)=(-i)^a,
$$
which yield the first formula. By Lemma \ref{lem:sign}, we have
$$
A_n:=\sum_{k=1}^{2^{n+1}}e\left(\frac{mak^2}{2^{n+2}}\right)
=\left\{\begin{array}{cl}2^{\frac{n+1}{2}} e\left(\frac{ma}{8}\right), & \text{if}\; 2\nmid n, \\ 2^{\frac{n}{2}}\left(1+e\left(\frac{ma}{4}\right)\right), & \text{if}\; 2\mid n.\end{array}\right.
$$
If $m$ is odd and $mm' =1 \, \bmod \,2^{n+2}$, then
\begin{align*}
ST^mS \mathbf{e}^\gamma&=\frac{(-i)^a}{2^{n+1}}\sum_{\beta\in D_{2^n}(a)}\sum_{\alpha\in D_{2^n}(a)}e\left(\frac{ma(\alpha-m'(\gamma+\beta))^2}{2^{n+2}}\right)e\left(\frac{-am'(\gamma+\beta)^2}{2^{n+2}}\right) \mathbf{e}^\beta \\ & =\frac{(-i)^a}{2^{n+1}} A_n \sum_{\beta \in L_{p^n}(a)}e\left(\frac{-am'(\gamma+\beta)^2}{2^{n+2}}\right) \mathbf{e}^\beta.
\end{align*}
It follows that
$$
T^{m'}ST^mS \mathbf{e}^\gamma=e\left(\frac{-am'\gamma^2}{2^{n+2}}\right)\frac{(-i)^a}{2^{n+1}} A_n \sum_{\beta \in D_{2^n}(a)}e\left(\frac{-2am'\gamma\beta}{2^{n+2}}\right)\mathbf{e}^\beta, 
$$
and thus 
$$
ST^{m'}ST^mS \mathbf{e}^\gamma=e\left(\frac{-am'\gamma^2}{2^{n+2}}\right)\frac{A_ne(-3\mathrm{sign}(D_{2^n}(a))/8)}{2^{\frac{n+1}{2}}}\mathbf{e}^{ -m'\gamma}.
$$
This proves the second formula by Lemma \ref{lem:sign}. 
\end{proof}

We now determine the sign of non-trivial $p$-parts for any odd prime $p$. 

\begin{lemma}\label{lem:vanish-sign}
Let $k_1,k_2$ be positive integers and $a_1,a_2$ be integers that are coprime to an odd prime $p$. If either $ k_1+k_2$ is even or $p=1 \, \bmod 4$, then  $(\CC L_{p^{k_1}}^{\epsilon}(a_1)\otimes\CC L_{p^{k_2}}^{-\epsilon}(a_2))^{\Gamma_0(p^n)}=0$ for $\epsilon\in\{+,-\}$ and any integer $n\geq 0$. 
\end{lemma}

\begin{proof}
Recall that $Z=(-I,i)$ and $-I\in\Gamma_0(N)$ for any positive integer $N$.
By Lemma \ref{lem:sign} and \eqref{eq:Z-action}, $\rho(Z)$ maps $v\in \CC L_{p^{k_1}}^{\epsilon}(a_1)\otimes\CC L_{p^{k_2}}^{-\epsilon}(a_2)$ to $-v$. This proves the desired lemma. 
\end{proof}

In a similar way, we confirm the sign of non-trivial $2$-parts. 

\begin{lemma}\label{lem:vanish-sign-2}
Let $k_1,k_2$ be non-negative integers and $a_1,a_2$ be odd integers. Let $\epsilon\in\{+,-\}$ and $n\geq 0$ be an integer. If $a_1a_2$ is a square modulo $4$, then  $(\CC D_{2^{k_1}}^{\epsilon}(a_1)\otimes\CC D_{2^{k_2}}^{\epsilon}(a_2))^{\Gamma_0(2^n)}=0$. If $a_1a_2$ is not a square modulo $4$, then  $(\CC D_{2^{k_1}}^{\epsilon}(a_1)\otimes\CC D_{2^{k_2}}^{-\epsilon}(a_2))^{\Gamma_0(2^n)}=0$. 
\end{lemma}

The following reduction formulas are useful.

\begin{lemma}\label{lem:reduce_level_from_p^n+1_to_p^n} 
Let  $p$ be a prime, and $a_1,a_2$ be integers that are coprime to $p$.
Let $k_1,k_2,k_3$ be integers with $0\leq k_1\leq k_2\leq k_3$. If $p=2$, then 
$$ 
\big(\CC D_{2^{k_1}}(a_1)\otimes\CC D_{2^{k_2}}(a_2)\big)^{\Gamma_0(2^{2+k_3})}= \big(\CC D_{2^{k_1}}(a_1)\otimes\CC D_{2^{k_2}}(a_2)\big)^{\Gamma_0(2^{2+k_2})}.
$$
If $p$ is odd, then 
$$ 
\big(\CC L_{p^{k_1}}(a_1)\otimes\CC L_{p^{k_2}}(a_2)\big)^{\Gamma_0(p^{k_3})}= \big(\CC L_{p^{k_1}}(a_1)\otimes\CC L_{p^{k_2}}(a_2)\big)^{\Gamma_0(p^{k_2})}.
$$
\end{lemma}

\begin{proof}
We only prove the odd case, since the even case is similar. Assume that $p$ is odd. 
Let  $m,m_1,m_2$ be integers that are coprime to $p$ and satisfy $mm_1 = 1 \bmod \, p^{k_2}$ and $mm_2 =1 \, \bmod \, p^{k_3}$.  According to \eqref{coset_decomposition}, it suffices to prove that $ST^mST^{m_1}S$ and $ST^mST^{m_2}S$ have the same action on both $\CC L_{p^{k_1}}(a_1)$ and $\CC L_{p^{k_2}}(a_2)$, which can be derived from \eqref{eq:oddformulaSTSTS}. 
\end{proof}

In the next lemmas, we compute the $2$-parts and $p$-parts of $J_{1,m}(N)$ in some particular cases. 

\begin{lemma}\label{lem:CD1_a1a2_square}
Let $a_1,a_2$ be odd integers such that $a_1a_2$ is a square modulo $4$. For $n\geq2$, we have
\begin{align*}
&\big(\CC D_{1}(a_1)\otimes\CC D_{2^{2n}}(a_2)\big)^{\Gamma_0(2^{2n+2})}\big/ \big(\CC D_{1}(a_1)\otimes\CC D_{2^{2n-2}}(a_2)\big)^{\Gamma_0(2^{2n})}\\=& \CC\big(\mathbf{e}^0\otimes(\mathbf{e}^{2^{2n-1}}-\mathbf{e}^{3\cdot2^{2n-1}})\big).
\end{align*}
For $n\geq 3$, we have
\begin{align*}
&\big(\CC D_{1}(a_1)\otimes\CC D_{2^{2n+1}}(a_2)\big)^{\Gamma_0(2^{2n+3})}\big/\big(\CC D_{1}(a_1)\otimes\CC D_{2^{2n-1}}(a_2)\big)^{\Gamma_0(2^{2n+1})} \\ =  & \CC \big(\mathbf{e}^0\otimes(\mathbf{e}^{2^{2n-1}}+\mathbf{e}^{3\cdot2^{2n-1}}-\mathbf{e}^{5\cdot2^{2n-1}}-\mathbf{e}^{7\cdot2^{2n-1}})\big).
\end{align*}
For $t=0,1,2,3,5$, we have
\begin{equation}\label{eq:CD1_CD16-_vanish}
\big(\CC D_{1}(a_1)\otimes\CC D_{2^t}(a_2)\big)^{\Gamma_0(2^{t+2})} = 0.
\end{equation}
\end{lemma}

\begin{proof}
We first consider the even case. By Section \ref{sec:Weil-rep} and the discussion before \eqref{coset_decomposition}, for a finite group $G$ acting on a vector space $V$, the invariant subspace $V^G$ is the image of the projection operator $e = \frac{1}{|G|} \sum_{g \in G} g$.

Let us review a basic fact. Since for any integers $a$ and $m>0$, $\CC D_{m}(a)$ is invariant under $\Gamma(4m)$, $\big(\CC D_{1}(a_1)\otimes\CC D_{2^{2n}}(a_2)\big)$ is then invariant under $\Gamma(2^{2n+2}).$  In our context, the group $G$ is the finite quotient group $\Gamma_0({2^{2n+2}} )/ \Gamma(2^{2n+2})$. We determine the invariant subspace by directly computing the action of the projection operator $$ I_0:=\sum_{g\in \Gamma_0(2^{2n+2})/\Gamma(2^{2n+2})}\rho(g)= \sum_{\substack{ l \, \bmod 2^{2n+2}\\ \mathrm{gcd}(l,2)=1}}\rho(ST^l ST^{l^{-1}}S)\sum_{j\, \bmod 2^{2n+2}}\rho(T^j)$$
on a basis $\textbf{e}^{\gamma_1}\otimes\textbf{e}^{\gamma_2}$ of $\CC D_{1}(a_1)\otimes\CC D_{2^{2n}}(a_2).$

For $\gamma_1\in D_1(a_1)$ and $\gamma_2\in D_{2^{2n}}(a_2)$, the sum
$$
\sum_{l=1}^{2^{2n+2}}T^l (\mathbf{e}^{\gamma_1}\otimes \mathbf{e}^{\gamma_2})= \sum_{l=1}^{2^{2n+2}}e\Big(\frac{l(a_1 2^{2n}\gamma_1^2+a_2\gamma_2^2)}{2^{2n+2}}\Big)\mathbf{e}^{\gamma_1}\otimes \mathbf{e}^{\gamma_2}
$$
does not vanish if and only if $2^{2n+2}\mid (a_1 2^{2n}\gamma_1^2+a_2\gamma_2^2)$, or equivalently, $\gamma_1=0$ and $2^{n+1}\mid \gamma_2$. Assume that $\gamma_2=2^{s}k$ with $k$ odd and $n+1\leq s\leq 2n+1$. If $s\in\{2n,2n+1\}$, then $\gamma_2=2^s$ in $D_{2^{2n}}(a_2)$, and thus by \eqref{eq:evenformulaSTSTS} and considering $l$ modulo $4$ we verify
$$
\sum_{\mathrm{odd} \, l=1}^{2^{2n+2}-1} ST^{l^{-1}}ST^l S (\mathbf{e}^0\otimes \mathbf{e}^{\gamma_2})=0. 
$$
In particular, for $t=0,2$ we have
$$ 
\big(\CC D_{1}(a_1)\otimes\CC D_{2^t}(a_2)\big)^{\Gamma_0(2^{t+2})} = 0.  
$$
If $s=2n-1$, then $\gamma_2=2^{2n-1}$ or $3\cdot 2^{2n-1}$ and 
$$
I_0(\mathbf{e}^0\otimes \mathbf{e}^{\gamma_2})=c_0 \big(\mathbf{e}^0\otimes(\mathbf{e}^{2^{2n-1}}-\mathbf{e}^{3\cdot2^{2n-1}})\big) 
$$
for some constant $c_0\neq0$. If $n+1\leq s<2n-1$, then by writing $l=2^{2n-s}l_2+l_1$ we compute 
$$ 
I_0(\mathbf{e}^0\otimes \mathbf{e}^{\gamma_2})=\sum_{\mathrm{odd}\, l_1=1}^{2^{2n-s}}\sum_{l_2=1}^{2^{s+2}} c_{l_1} \mathbf{e}^0\otimes \mathbf{e}^{l_1 k 2^s+l_2 k 2^{2n}},
$$
where $c_{l_1}$ are non-zero constants depending only on $l_1$. By the embedding introduced at the end of Section \ref{sec:Weil-rep}, the above inner summand lies in $\CC D_{1}(a_1)\otimes \CC D_{2^{2n-2}}(a_2)$, and thus $I_0(\mathbf{e}^0\otimes \mathbf{e}^{\gamma_2})$ is invariant under $\Gamma_0(2^{2n})$ by Lemma \ref{lem:reduce_level_from_p^n+1_to_p^n}. This proves the formula for the even case. 
	
We then consider the odd case. Let $I_1:=\sum_{g\in \Gamma_0(2^{2n+3})/\Gamma(2^{2n+3})}\rho(g)$ be the projection operator. Let us consider the action of $T$ again. For $\gamma_1\in D_1(a_1)$ and $\gamma_2\in D_{2^{2n+1}}(a_2)$, we show that
$$
\sum_{l=1}^{2^{2n+2}}T^l (\mathbf{e}^{\gamma_1}\otimes \mathbf{e}^{\gamma_2})= \sum_{l=1}^{2^{2n+2}}e\left(\frac{l(2^{2n+1}a_1\gamma_1^2+a_2\gamma_2^2)}{2^{2n+3}}\right)\mathbf{e}^{\gamma_1}\otimes \mathbf{e}^{\gamma_2}
$$
does not vanish if and only if $\gamma_1=0$ and $2^{n+2}\mid \gamma_2$. Assume that $\gamma_2=2^{s}k$ with $k$ odd and $n+2\leq s\leq 2n+2$. If $s\in\{2n,2n+1,2n+2\}$, then by \eqref{eq:evenformulaSTSTS} and considering $l$ modulo $8$ we confirm
$$
\sum_{\mathrm{odd} \, l=1}^{2^{2n+3}-1} ST^{l^{-1}}ST^l S (\mathbf{e}^0\otimes \mathbf{e}^{\gamma_2})=0. 
$$
In particular, for $t=1,3,5$ we have
$$  
\big(\CC D_{1}(a_1)\otimes\CC D_{2^t}(a_2)\big)^{\Gamma_0(2^{t+2})} = 0.
$$
If $s=2n-1$, by \eqref{eq:evenformulaSTSTS} and considering $l$ modulo $8$ we obtain
$$ 
I_1(\mathbf{e}^0\otimes \mathbf{e}^{\gamma_2})=c_1\big(\mathbf{e}^0\otimes(\mathbf{e}^{2^{2n-1}}+\mathbf{e}^{3\cdot2^{2n-1}}-\mathbf{e}^{5\cdot2^{2n-1}}-\mathbf{e}^{7\cdot2^{2n-1}})\big) $$
for some non-zero constant $c_1$. If $n+2\leq s<2n-1$, then
$$ 
I_1(\mathbf{e}^0\otimes \mathbf{e}^{\gamma_2})=\sum_{\mathrm{odd}\, l_1=1}^{2^{2n+1-s}}\sum_{l_2=1}^{2^{s+2}} d_{l_1} \mathbf{e}^0\otimes \mathbf{e}^{l_1k2^s+l_2k2^{2n+1}} 
$$
where $d_{l_1}$ are certain non-zero constants depending only on $l_1$. Similarly to before, the inner summand lies in $\CC D_{1}(a_1)\otimes \CC D_{2^{2n-1}}(a_2)$, and $I_1(\mathbf{e}^0\otimes \mathbf{e}^{\gamma_2})$ is invariant under $\Gamma_0(2^{2n+1})$ by Lemma \ref{lem:reduce_level_from_p^n+1_to_p^n}.	 We thus prove the formula for the odd case.
\end{proof}

\begin{lemma}\label{lem:CD1_-a1a2_square}
Let $a_1,a_2$ be odd integers such that $-a_1a_2$ is a square modulo $4$. For $n\geq 1$, we have
$$
\big(\CC D_{1}(a_1)\otimes\CC D_{2^{2n}}(a_2)\big)^{\Gamma_0(2^{2n+2})}\big/ \big(\CC D_{1}(a_1)\otimes\CC D_{2^{2n-2}}(a_2)\big)^{\Gamma_0(2^{2n})}= \CC \big(\mathbf{e}^0\otimes \mathbf{e}^{2^{2n}}\big).
$$
Moreover, 
\begin{align}
\label{eq:CD1CD4_Gamma16}\big(\CC D_{1}(a_1)\otimes \CC D_{4}(a_2)\big)^{\Gamma_0{(16)}} &=\CC \big( \mathbf{e}^0\otimes \mathbf{e}^0) \oplus \CC \big(\mathbf{e}^0\otimes \mathbf{e}^4\big) \oplus \CC \big(\mathbf{e}^1\otimes(\mathbf{e}^2+\mathbf{e}^6)\big),\\
\label{eq:CD1_Gamma4} \big(\CC D_{1}(a_1)\otimes \CC D_{1}(a_2) \big)^{\Gamma_0{(4)}} &= \CC \big( \mathbf{e}^0\otimes \mathbf{e}^0 \pm \mathbf{e}^1 \otimes \mathbf{e}^1 \big),\\
\label{eq:CD1_Gamma2} \big(\CC D_{1}(a_1)\otimes \CC D_{1}(a_2)\big)^{\Gamma_0{(2)}} &= \CC \big( \mathbf{e}^0\otimes \mathbf{e}^0 + \mathbf{e}^1 \otimes \mathbf{e}^1 \big),\\
\label{eq:CD1_Gamma1} \big(\CC D_{1}(a_1)\otimes \CC D_{1}(a_2)\big)^{\SL_2(\ZZ)} &= \CC \big( \mathbf{e}^0\otimes \mathbf{e}^0 + \mathbf{e}^1 \otimes \mathbf{e}^1 \big).
\end{align}
For $n\geq 3$, we have
\begin{align*}
&\big(\CC D_{1}(a_1)\otimes\CC D_{2^{2n+1}}(a_2)\big)^{\Gamma_0(2^{2n+3})}\big/ \big(\CC D_{1}(a_1)\otimes\CC D_{2^{2n-1}}(a_2)\big)^{\Gamma_0(2^{2n+1})} \\ =  & \CC \big(\mathbf{e}^0\otimes(\mathbf{e}^{2^{2n-1}}-\mathbf{e}^{3\cdot2^{2n-1}}-\mathbf{e}^{5\cdot2^{2n-1}}+\mathbf{e}^{7\cdot2^{2n-1}})\big). 
\end{align*}
For $n=0,1,2$, we have
\begin{equation}\label{eq:CD1_CD16+_vanish}
\big(\CC D_{1}(a_1)\otimes\CC D_{2^{2n+1}}(a_2)\big)^{\Gamma_0(2^{2n+3})} = 0.
\end{equation}
\end{lemma}

\begin{proof}
The proof is similar to that of Lemma \ref{lem:CD1_a1a2_square}. We first consider the even case. Again, we set $I_0:=\sum_{g\in \Gamma_0(2^{2n+2})/\Gamma(2^{2n+2})} \rho(g)$. For $\gamma_1\in D_1(a_1)$ and $\gamma_2\in D_{2^{2n}}(a_2)$, we find that 
$$
\sum_{l=1}^{2^{2n+2}}T^l (\mathbf{e}^{\gamma_1}\otimes \mathbf{e}^{\gamma_2})= \sum_{l=1}^{2^{2n+2}}e\Big(\frac{l(2^{2n}a_1\gamma_1^2+a_2\gamma_2^2)}{2^{2n+2}}\Big)\mathbf{e}^{\gamma_1}\otimes \mathbf{e}^{\gamma_2}
$$
does not vanish if and only if either $\gamma_1=0$ and $2^{n+1}\mid \gamma_2$, or $\gamma_1=1$ and $\gamma_2=2^{n}k$ with $k$ odd. 

Assume that $\gamma_2=2^{s}k$ with $k$ odd and $n+1\leq s\leq 2n+1$. We notice that
\begin{align*}
I_0(\mathbf{e}^0\otimes \mathbf{e}^{0})&=c_0 \big(\mathbf{e}^0\otimes \mathbf{e}^{0}\big),\\
I_0(\mathbf{e}^0\otimes \mathbf{e}^{2^{2n}})&=c'_0 \big(\mathbf{e}^0\otimes \mathbf{e}^{2^{2n}}\big),
\end{align*}
where $c_0$ and $c'_0$ are certain non-zero constants. 
Moreover, 
$$
\mathbf{e}^0\otimes \mathbf{e}^{0}+\mathbf{e}^0\otimes \mathbf{e}^{2^{2n}}\in \big(\CC D_{1}(a_1)\otimes \CC D_{2^{2n-2}}(a_2)\big)^{\Gamma_0(2^n)}.
$$
Similar to before, if $n+1\leq s<2n-1$, then
$$ 
I_0(\mathbf{e}^0\otimes \mathbf{e}^{\gamma_2})=\sum_{\mathrm{odd}\, l_1=1}^{2^{2n-s}}\sum_{l_2=1}^{2^{s+2}} c_{l_1} \mathbf{e}^0\otimes \mathbf{e}^{l_1 k 2^s+l_2 k 2^{2n}},
$$
where $c_{l_1}$ are certain non-zero constants depending only on $l_1$. Therefore,
the inner summand also lies in  $(\CC D_{1}(a_1)\otimes \CC D_{2^{2n-2}}(a_2))^{\Gamma_0(2^n)}$. 
	
In addition, $I_0(\mathbf{e}^1\otimes \mathbf{e}^{2^{n}k})$ is equal to
$
\mathbf{e}^1\otimes \big(\sum_{\mathrm{odd}\, l=1}^{2^{2n+1}-1} \mathbf{e}^{l 2^{n}}\big) 
$
up to non-zero constant factor, and thus lies in $(\CC D_{1}(a_1)\otimes \CC D_{1}(a_2))^{\Gamma_0(4)}$. 
By taking $n=0,1$, we obtain \eqref{eq:CD1CD4_Gamma16} and \eqref{eq:CD1_Gamma4}.
In the space $\CC D_1(a_1)\otimes \CC D_1(a_2)$, we check that
\begin{align*}
ST^2S(\mathbf{e}^0\otimes \mathbf{e}^0)&=\mathbf{e}^1\otimes \mathbf{e}^1,\\
ST^2S(\mathbf{e}^1\otimes \mathbf{e}^1)&=\mathbf{e}^0\otimes \mathbf{e}^0,\\
S(\mathbf{e}^0\otimes \mathbf{e}^0+ \mathbf{e}^1\otimes \mathbf{e}^1)&=\mathbf{e}^0\otimes \mathbf{e}^0+\mathbf{e}^1\otimes \mathbf{e}^1,
\end{align*}
which yield \eqref{eq:CD1_Gamma2} and \eqref{eq:CD1_Gamma1}.
	
We then consider the odd case. Let $I_1:=\sum_{g\in \Gamma_0(2^{2n+3})/\Gamma(2^{2n+3})}\rho(g)$. For $\gamma_1\in D_1(a_1)$ and $\gamma_2\in D_{2^{2n+1}}(a_2)$, the sum
$$
\sum_{l=1}^{2^{2n+2}}T^l (\mathbf{e}^{\gamma_1}\otimes \mathbf{e}^{\gamma_2})= \sum_{i=1}^{2^{2n+2}}e\Big( \frac{l(2^{2n+1}a_1\gamma_1^2+a_2\gamma_2^2)}{2^{2n+3}}\Big) \mathbf{e}^{\gamma_1}\otimes \mathbf{e}^{\gamma_2}
$$
does not vanish if and only if $\gamma_1=0$ and $2^{n+2}\mid \gamma_2$. Assume that $\gamma_2=2^{s}k$ with $k$ odd and $n+2\leq s\leq 2n+2$. If $s\in\{2n,2n+1,2n+2\}$, then
$$
\sum_{\mathrm{odd} \, l=1}^{2^{2n+3}-1} ST^{l^{-1}}ST^l S (\mathbf{e}^0\otimes \mathbf{e}^{\gamma_2})=0,
$$
which particularly yields \eqref{eq:CD1_CD16+_vanish}.
If $s=2n-1$, then by considering $l$ modulo $8$ we verify that 
$$ 
I_0(\mathbf{e}^0\otimes \mathbf{e}^{\gamma_2})=c_1 \big(\mathbf{e}^0\otimes(\mathbf{e}^{2^{2n-1}}-\mathbf{e}^{3\cdot2^{2n-1}}-\mathbf{e}^{5\cdot2^{2n-1}}+\mathbf{e}^{7\cdot2^{2n-1}})\big) 
$$
for some non-zero constant $c_1$. If $n+2\leq s<2n-1$, then
$$ 
I_1(\mathbf{e}^0\otimes \mathbf{e}^{\gamma_2})=\sum_{\mathrm{odd}\, l_1=1}^{2^{2n+1-s}}\sum_{l_2=1}^{2^{s+2}} c_{l_1} \mathbf{e}^0\otimes \mathbf{e}^{l_1k2^s+l_2k2^{2n+1}}, 
$$
where $c_{l_1}$ are certain non-zero constants depending only on $l_1$. As before, 
the inner summand lies in $(\CC D_{1}(a_1)\otimes \CC D_{2^{2n-1}}(a_2))^{\Gamma_0(2^{2n+1})}$. We thus prove the formula for the odd case. 
\end{proof}

\begin{lemma}\label{lem:CD2CD2N}
Let $n$ be a positive integer and $a_1,a_2$ be odd integers.
\begin{enumerate}
\item If $n\geq 3$, then the quotient space
$$
\big(\CC D_{2}(a_1)\otimes\CC D_{2^{2n}}(a_2))^{\Gamma_0(2^{2n+2})}\big/ \big(\CC D_{2}(a_1)\otimes\CC D_{2^{2n-2}}(a_2)\big)^{\Gamma_0(2^{2n})}
$$
is generated by 
$$
\mathbf{e}^0\otimes\big(\mathbf{e}^{2^{2n-2}}+\mathbf{e}^{3\cdot2^{2n-2}}-\mathbf{e}^{5\cdot2^{2n-2}}-\mathbf{e}^{7\cdot2^{2n-2}}\big)
$$
if $a_1a_2=1\, \bmod 4$, and by 
$$
\mathbf{e}^0\otimes\big(\mathbf{e}^{2^{2n-2}}-\mathbf{e}^{3\cdot2^{2n-2}}-\mathbf{e}^{5\cdot2^{2n-2}}+\mathbf{e}^{7\cdot2^{2n-2}}\big) 
$$ 
if $a_1a_2=3\, \bmod 4$. Besides, for $n=0,1,2$ we have
$$ 
\big(\CC D_{2}(a_1)\otimes\CC D_{2^{2n}}(a_2))^{\Gamma_0(2^{2n+2})}=0.  
$$
		
\item If $n\geq 2$ and $a_1a_2=1\,\bmod 4$, then 
\begin{align*}
&\big(\CC D_{2}(a_1)\otimes\CC D_{2^{2n+1}}(a_2)\big)^{\Gamma_0(2^{2n+3})} \big/ \big(\CC D_{2}(a_1)\otimes\CC D_{2^{2n-1}}(a_2)\big)^{\Gamma_0(2^{2n+1})} \\ =&\CC \big(\mathbf{e}^0\otimes( \mathbf{e}^{2^{2n}}-\mathbf{e}^{3\cdot 2^{2n}})\big).    
\end{align*} 
Besides, for $n<2$ we have
\begin{align*}
\big(\CC D_{2}(a_1)\otimes\CC D_{8}(a_2)\big)^{\Gamma_0(32)} &=  \CC \big(\mathbf{e}^2\otimes (\mathbf{e}^4-\mathbf{e}^{12})\big),\\
\big(\CC D_{2}(a_1)\otimes\CC D_{2}(a_2)\big)^{\Gamma_0(8)}&=0.
\end{align*}
		
\item If $n\geq1$ and $a_1a_2=3\,\bmod 4$, then
\begin{align*}
&\big(\CC D_{2}(a_1)\otimes\CC D_{2^{2n+1}}(a_2)\big)^{\Gamma_0(2^{2n+3})}\big/ \big(\CC D_{2}(a_1)\otimes\CC D_{2^{2n-1}}(a_2)\big)^{\Gamma_0(2^{2n+1})} \\ =&\CC \big(\mathbf{e}^0\otimes \mathbf{e}^{0}\big).    
\end{align*}
Besides, $(\CC D_2(a_1)\otimes\CC D_2(a_2))^{\Gamma_0(8)}$ is generated by $\mathbf{e}^0\otimes \mathbf{e}^0$ and $\mathbf{e}^2\otimes \mathbf{e}^2$ if $a_1a_2=3\,\bmod 8$, and by $\mathbf{e}^0\otimes \mathbf{e}^0$, $\mathbf{e}^2\otimes \mathbf{e}^2$, $(\mathbf{e}^1+ \mathbf{e}^3)\otimes (\mathbf{e}^1+ \mathbf{e}^3)$ and $(\mathbf{e}^1- \mathbf{e}^3)\otimes (\mathbf{e}^1- \mathbf{e}^3)$ if $a_1a_2=7\,\bmod 8$.
		
\item If $a_1a_2=3\,\bmod 8$, then
\begin{align*}
&\big(\CC D_2^+(a_1)\otimes\CC D_2^+(a_2)\big)^{\Gamma_0(4)}=\big(\CC D_2^+(a_1)\otimes\CC D_2^+(a_2)\big)^{\Gamma_0(2)} \\
=& \CC\big(\mathbf{e}^0\otimes \mathbf{e}^0+\mathbf{e}^2\otimes \mathbf{e}^2\big),
\end{align*}
and 
$$
\big(\CC D_2^+(a_1)\otimes\CC D_2^+(a_2)\big)^{\SL_2(\ZZ)}=0.
$$
If  $a_1a_2=7\,\bmod 8$, then
\begin{align*}
&\big(\CC D_2^+(a_1)\otimes\CC D_2^+(a_2)\big)^{\Gamma_0(4)}=\big(\CC D_2^+(a_1)\otimes\CC D_2^+(a_2)\big)^{\Gamma_0(2)} \\
=&\CC \big(\mathbf{e}^0\otimes \mathbf{e}^0+\mathbf{e}^2\otimes \mathbf{e}^2\big) + \CC \big((\mathbf{e}^1+ \mathbf{e}^3)\otimes (\mathbf{e}^1+ \mathbf{e}^3)\big),  
\end{align*}
and
$$ 
\big(\CC D_2^+(a_1)\otimes\CC D_2^+(a_2))^{\SL_2(\ZZ)}=\CC \big(2\mathbf{e}^0\otimes \mathbf{e}^0+2\mathbf{e}^2\otimes \mathbf{e}^2+ (\mathbf{e}^1+ \mathbf{e}^3)\otimes (\mathbf{e}^1+ \mathbf{e}^3) \big),
$$
and 
\begin{align*}
&\big(\CC D_2^-(a_1)\otimes\CC D_2^-(a_2)\big)^{\Gamma_0(8)}= \big(\CC D_2^-(a_1)\otimes\CC D_2^-(a_2)\big)^{\SL_2(\ZZ)} \\
=& \CC \big( (\mathbf{e}^1- \mathbf{e}^3)\otimes (\mathbf{e}^1- \mathbf{e}^3) \big).
\end{align*}
\end{enumerate}
\end{lemma}

\begin{proof}
We first prove assertion (1). We set $I_0:=\sum_{g\in \Gamma_0(2^{2n+2})/\Gamma(2^{2n+2})}\rho(g)$. Let $\gamma_1\in \CC D_2(a_1)$ and $\gamma_2\in \CC D_{2^{2n}}(a_2)$. We deduce from 
$$
\sum_{l=1}^{2^{2n+2}}T^l(\mathbf{e}^{\gamma_1}\otimes \mathbf{e}^{\gamma_2})=\sum_{l=1}^{2^{2n+2}} e\Big(\frac{l(a_1 2^{2n-1}\gamma_1^2+a_2\gamma_2^2)}{2^{2n+2}}\Big)\mathbf{e}^{\gamma_1}\otimes \mathbf{e}^{\gamma_2}
$$
that $I_0(\mathbf{e}^{\gamma_1}\otimes \mathbf{e}^{\gamma_2})=0$ if $\gamma_1=1,2,3$. 
When $\gamma_1=0$, $I_0(\mathbf{e}^{\gamma_1}\otimes \mathbf{e}^{\gamma_2})\neq 0$ only if $2^{n+1}|\gamma_2$. 
Assume that $\gamma_2=2^{s}k$ with $n+1\leq s\leq 2n+1$ and $k$ odd.  
Then $I_0(\mathbf{e}^{0}\otimes \mathbf{e}^{\gamma_2})=0$ when $s=2n+1,2n,2n-1$. 
In particular, for $n=0,1,2$ we have
$$ 
\big(\CC D_{2}(a_1)\otimes\CC D_{2^{2n}}(a_2)\big)^{\Gamma_0(2^{2n+2})}=0.  
$$
If $s=2n-2$, then $I_0(\mathbf{e}^{0}\otimes \mathbf{e}^{\gamma_2})$ is of the form
$$
\mathbf{e}^0\otimes(\mathbf{e}^{2^{2n-2}}+\mathbf{e}^{3\cdot2^{2n-2}}-\mathbf{e}^{5\cdot2^{2n-2}}-\mathbf{e}^{7\cdot2^{2n-2}})
$$
if $a_1a_2=1\, \bmod 4$, and of the form
$$
\mathbf{e}^0\otimes(\mathbf{e}^{2^{2n-2}}-\mathbf{e}^{3\cdot2^{2n-2}}-\mathbf{e}^{5\cdot2^{2n-2}}+\mathbf{e}^{7\cdot2^{2n-2}})
$$ 
if $a_1a_2=3\, \bmod 4$. As before, if $n+1\leq s<2n-2$, then
$$ 
I_0(\mathbf{e}^0\otimes \mathbf{e}^{\gamma_2}) \in \big(\CC D_{2}(a_1)\otimes\CC D_{2^{2n-2}}(a_2)\big)^{\Gamma_0(2^{2n})}.
$$
We thus prove the desired assertion. 

\vspace{3mm}

We then prove assertions (2) and (3). We set $I_1:=\sum_{g\in \Gamma_0(2^{2n+3})/\Gamma(2^{2n+3})}\rho(g)$. Let $\gamma_1\in \CC D_2(a_1)$ and $\gamma_2\in \CC D_{2^{2n+1}}(a_2)$. Similarly, we conclude from 
$$
\sum_{l=1}^{2^{2n+3}}T^l(\mathbf{e}^{\gamma_1}\otimes \mathbf{e}^{\gamma_2})=\sum_{l=1}^{2^{2n+3}} e\Big(\frac{l(a_1 2^{2n}\gamma_1^2+a_2\gamma_2^2)}{2^{2n+3}}\Big)\mathbf{e}^{\gamma_1}\otimes \mathbf{e}^{\gamma_2}
$$
that $I_1(\mathbf{e}^{\gamma_1}\otimes \mathbf{e}^{\gamma_2})\neq 0$ only if one of the following conditions holds: 

\vspace{2mm}

\textbf{(i)} $\gamma_1=1,3$,  $a_1a_2=7\,\bmod 8$ and $\gamma_2=2^nk$ with $k$ odd. In this case, $ I_1(\mathbf{e}^{\gamma_1}\otimes \mathbf{e}^{2^nk}) $ has the form 
$$ 
\sum_{\substack{ 1\leq l \leq 2^{2n+3} \\ l=1\bmod 4}} \mathbf{e}^{\gamma_1}\otimes \mathbf{e}^{2^nkl}+\sum_{\substack{ 1\leq l \leq 2^{2n+3} \\ l=3\bmod 4}}\mathbf{e}^{-\gamma_1}\otimes \mathbf{e}^{2^nkl},   
$$
which belongs to $(\CC D_2(a_1)\otimes \CC D_2(a_2))^{\Gamma_0(8)}$ and is identical to $\mathbf{e}^{\gamma_1}\otimes \mathbf{e}^{k\bmod 4}+\mathbf{e}^{-\gamma_1}\otimes \mathbf{e}^{-k\bmod 4}$.

\vspace{2mm}
	
\textbf{(ii)} $\gamma_1=2$ and $\gamma_2=2^{n+1}k$ with $k$ odd. In this case, $I_1(\mathbf{e}^{2}\otimes \mathbf{e}^{2^{n+1}k})$ has the form
$$ 
\sum_{\substack{ 1\leq l \leq 2^{2n+3} \\ l=1\bmod 4}}\mathbf{e}^2\otimes \mathbf{e}^{2^{n+1}kl}+e\Big(\frac{a_1+a_2}{4}\Big)\sum_{\substack{ 1\leq l \leq 2^{2n+3} \\ l=3\bmod 4}}\mathbf{e}^{2}\otimes \mathbf{e}^{2^{n+1}kl}.
$$ 
If $n=0$ and $a_1a_2=1\bmod 4$, then $I_1(\mathbf{e}^{2}\otimes \mathbf{e}^{2^{n+1}k})=0$;
if $n=1$ and $a_1a_2=1\bmod 4$, then $I_1(\mathbf{e}^{2}\otimes \mathbf{e}^{2^{n+1}k})$ has the form $\mathbf{e}^2\otimes (\mathbf{e}^4-\mathbf{e}^{12})$; if $a_1a_2=1\bmod 4$, then
$$
I_1(\mathbf{e}^{2}\otimes \mathbf{e}^{2^{n+1}k}) \in \big(\CC D_2(a_1)\otimes \CC D_8(a_2)\big)^{\Gamma_0(32)};
$$
if $a_1a_2=3\bmod 4$, then 
$$
I_1(\mathbf{e}^{2}\otimes \mathbf{e}^{2^{n+1}k}) \in \big(\CC D_2(a_1)\otimes \CC D_2(a_2)\big)^{\Gamma_0(8)}. 
$$

\vspace{2mm}
	
\textbf{(iii)} $\gamma_1=0$ and $\gamma_2=2^sk$ with $n+2\leq s\leq 2n+2$ and $k$ odd. There are two cases.

\vspace{2mm}

\textbf{Case (I)} $a_1a_2=1\bmod 4$. As before, we show that $I_1(\mathbf{e}^0\otimes \mathbf{e}^{2^s k})=0$ if $s\in\{2n+1,2n+2\}$. Combining this fact and \textbf{(i)}-\textbf{(ii)} above, we prove that  $(\CC D_{2}(a_1)\otimes\CC D_{2}(a_2))^{\Gamma_0(8)}=0$ and the space $(\CC D_{2}(a_1)\otimes\CC D_{8}(a_2))^{\Gamma_0(32)}$ is of dimension one and generated by $\mathbf{e}^2\otimes (\mathbf{e}^4-\mathbf{e}^{12})$. 
By direct calculation, we find that
$$
I_1(\mathbf{e}^0\otimes \mathbf{e}^{2^{2n}k})=c\big(\mathbf{e}^0\otimes( \mathbf{e}^{2^{2n}}-\mathbf{e}^{3\cdot 2^{2n}})\big)
$$ 
for some non-zero constant $c$. If $s<2n$, then as before we assert that 	
$$ 
I_1(\mathbf{e}^0\otimes \mathbf{e}^{\gamma_2})\in \big(\CC D_2(a_1)\otimes \CC D_{2^{2n-1}}(a_2)\big)^{\Gamma_0(2^{2n+1})}. 
$$ 

\vspace{2mm}

\textbf{Case (II)} $a_1a_2=3\bmod 4$. We confirm that
\begin{align*}
I_1(\mathbf{e}^0\otimes \mathbf{e}^{2^{2n+1}})&=c_1\big(\mathbf{e}^0\otimes \mathbf{e}^{2^{2n+1}}\big),\\   I_1(\mathbf{e}^0\otimes \mathbf{e}^{0})&=c_2\big(\mathbf{e}^0\otimes \mathbf{e}^{0}\big), 
\end{align*}
for some non-zero constants $c_1$ and $c_2$. We notice that
$$
\mathbf{e}^0\otimes \mathbf{e}^{0}+\mathbf{e}^0\otimes \mathbf{e}^{2^{2n+1}} \in \big(\CC D_2(a_1)\otimes \CC D_{2^{2n-1}}(a_2)\big)^{\Gamma_0(2^{2n+1})}. 
$$
If $s\leq 2n$, then as before we conclude that 	
$$ 
I_1(\mathbf{e}^0\otimes \mathbf{e}^{\gamma_2})\in \big(\CC D_2(a_1)\otimes \CC D_{2^{2n-1}}(a_2)\big)^{\Gamma_0(2^{2n+1})}. 
$$ 
Combining \textbf{(i)}, \textbf{(ii)} and \textbf{(iii)} above, we complete the proofs of assertions (2) and (3). 

\vspace{3mm}

We finally prove assertion (4). For $\gamma\in D_2(a)$ with $a$ odd, we have
\begin{align*}
ST^4S \mathbf{e}^\gamma &=(-i)^a\mathbf{e}^{-\gamma +2},\\
ST^2S \mathbf{e}^\gamma &=\frac{(-i)^a+1}{2}\mathbf{e}^{-\gamma}+\frac{(-i)^a-1}{2}\mathbf{e}^{-\gamma +2}.
\end{align*}
This is enough to determine the spaces $(\CC D_2(a_1)\otimes \CC D_2(a_2))^{\Gamma_0(4)}$ and $(\CC D_2(a_1)\otimes \CC D_2(a_2))^{\Gamma_0(2)}$. 
By Lemma \ref{lem:irreducible_D2^N_LP^N} and direct calculation, we find the generator of $(\CC D_2(a_1)\otimes \CC D_2(a_2))^{\SL_2(\ZZ)}$.  We then finish the proof. 
\end{proof}

\begin{lemma}\label{lem:CLP^N}
Let $p$ be an odd prime and $a$ be an integer that is coprime to $p$. For any positive integer $n$, we have
\begin{equation}\label{eq:CLP^N_even_induction}
\big(\CC L_{p^{2n}}(a)\big)^{\Gamma_0(p^{2n})} = \big(\CC L_{p^{2n-2}}(a)\big)^{\Gamma_0(p^{2n-2})}   \bigoplus \CC \left(\sum_{l=1}^{p-1}\mathbf{e}^{lp^{2n-1}}\right)
\end{equation}
and 
\begin{equation}\label{eq:CLP^N_odd_induction}
\big(\CC L_{p^{2n+1}}(a)\big)^{\Gamma_0(p^{2n+1})} = \big(\CC L_{p^{2n-1}}(a)\big)^{\Gamma_0(p^{2n-1})} \bigoplus \CC \left(\sum_{l=1}^{p-1}\left(\frac{l}{p}\right)\mathbf{e}^{lp^{2n}}\right).	
\end{equation}
In particular, 
\begin{equation}\label{eq:CL_P2}
\big(\CC L_{p^2}(a)\big)^{\Gamma_0(p^2)} = \CC \left(\sum_{l=1}^{p}\mathbf{e}^{lp}\right) \oplus \CC \left(\sum_{l=1}^{p-1}\mathbf{e}^{lp}\right),
\end{equation}
where $\sum_{l=1}^{p}\mathbf{e}^{lp} $ lies in $(\CC L_{1})^{\SL_2(\ZZ)}$, and
\begin{equation}\label{eq:CL_P3}
\big(\CC L_{p^3}(a)\big)^{\Gamma_0(p^3)}=\CC \left(\sum_{l=1}^{p-1}\left(\frac{l}{p}\right)\mathbf{e}^{lp^2}\right),
\end{equation}
and for any non-negative integer $n$,
\begin{equation}\label{eq:CL_P1}
\big(\CC L_{p}(a)\big)^{\Gamma_0(p^n)}=0. 
\end{equation}
\end{lemma}

\begin{proof}
We first prove \eqref{eq:CLP^N_even_induction}.
Let $I_0:=\sum_{g\in \Gamma_0(p^{2n}) / \Gamma(p^{2n})} \rho(g)$ and $\gamma \in  L_{p^{2n}}(a)$. Then the sum
$$
\sum_{l=1}^{p^{2n}}T^l \mathbf{e}^\gamma=\sum_{l=1}^{p^{2n}}e\Big(\frac{la\gamma^2}{p^{2n}}\Big)\mathbf{e}^\gamma
$$
does not vanish if and only if $p^n \mid \gamma$. Assume $\gamma=p^sk$ with $n\leq s\leq 2n$ and $\mathrm{gcd}(p,k)=1$. The identities
$$
I_0(\mathbf{e}^0)=c_0 \mathbf{e}^0 \quad \text{and} \quad  
I_0(\mathbf{e}^{kp^{2n-1}})=c'_0\sum_{l=1}^{p-1}\mathbf{e}^{l p^{2n-1}}
$$
hold for some non-zero constants $c_0$ and $c_0'$. We note that 
$$
I_0 (c'_0\mathbf{e}^0+c_0\mathbf{e}^{kp^{2n-1}} ) \in \big(\CC L_{p^{2n-2}}(a)\big)^{\Gamma_0(p^{2n})}.
$$
By taking  $n=1$ in this case, we obtain \eqref{eq:CL_P2}.
If $s<2n-1$, then 
$$
I_0(\mathbf{e}^{\gamma})=p^{2n}\sum_{\substack{1\leq l_1\leq p^{2n-1-s}
\\ \mathrm{gcd}(l_1,p)=1}}\sum_{1\leq l_2\leq p^{s+1}} \mathbf{e}^{l_1kp^{s}+l_2kp^{2n-1}}. 
$$  
By Lemma \ref{lem:reduce_level_from_p^n+1_to_p^n} and the embedding introduced at the end of Section \ref{sec:Weil-rep}, the inner summand above lies in $(\CC L_{p^{2n-2}}(a))^{\Gamma_0(p^{2n-2})}$. We thus prove \eqref{eq:CLP^N_even_induction}. 
	
The proof of \eqref{eq:CLP^N_odd_induction} is similar. Let $I_1:=\sum_{g\in \Gamma_0(p^{2n+1})/ \Gamma(p^{2n+1})} \rho(g)$ and $\gamma\in L_{p^{2n+1}}(a)$. Similarly, $\sum_{l=1}^{p^{2n+1}}T^l \mathbf{e}^\gamma$ does not vanish if and only if $p^{n+1}\mid \gamma$. Assume that $\gamma=p^sk$ with $n+1\leq s\leq 2n+1$ and $\mathrm{gcd}(p,k)=1$. We verify that $I_1(\mathbf{e}^0)=0$ and 
$$
I_1(\mathbf{e}^{kp^{2n}})=c_1\sum_{l=1}^{p-1} \left(\frac{l}{p}\right)\mathbf{e}^{lkp^{2n}}=c_1\left(\frac{k}{p}\right) \sum_{l=1}^{p-1} \left(\frac{l}{p}\right)\mathbf{e}^{lp^{2n}} 
$$ 
for some non-zero constant $c_1$. By taking $n=0$ and $n=1$, we prove \eqref{eq:CL_P1} and \eqref{eq:CL_P3}, respectively. 
If $s<2n$, then 
$$ 
I_1(\mathbf{e}^{\gamma})=p^{2n+1}\sum_{\substack{1\leq l_1\leq p^{2n-s}
\\ \mathrm{gcd}(l_1,p)=1}}\sum_{1\leq l_2\leq p^{s}}\left(\frac{l_1}{p}\right) \mathbf{e}^{l_1kp^{s}+l_2kp^{2n}}.
$$
The inner summand lies in $(\CC L_{p^{2n-1}}(a))^{\Gamma_0(p^{2n-1})}$. This completes the proof of \eqref{eq:CLP^N_odd_induction}. 
\end{proof}

\begin{lemma}\label{lem:CLPCLP^2N}
Let $p$ be an odd prime and $a_1, a_2$ be integers that are coprime to $p$. For any positive integer $n$, we have 
\begin{equation}\label{eq:CLPCLP^N_even_induction}
\begin{split}
&\big(\CC L_p(a_1)\otimes \CC L_{p^{2n}}(a_2)\big)^{\Gamma_0(p^{2n})}\big/ \big(\CC L_p(a_1)\otimes \CC L_{p^{2n-2}}(a_2)\big)^{\Gamma_0(p^{2n-2})} \\ 
= & \CC \left(\mathbf{e}^0\otimes \sum_{l=1}^{p-1}\left(\frac{l}{p}\right)\mathbf{e}^{lp^{2n-1}}\right).
\end{split}
\end{equation}
In particular, 
\begin{equation}\label{eq:CLPCLP2}
\big(\CC L_p(a_1)\otimes \CC L_{p^{2}}(a_2)\big)^{\Gamma_0(p^{2n})}=\CC \left(\mathbf{e}^0\otimes\sum_{l=1}^{p-1}\left(\frac{l}{p}\right)\mathbf{e}^{lp}\right).  
\end{equation}
For any positive integer $n$, we have
\begin{equation}\label{eq:CLPCLP^N_odd_induction}
\big(\CC L_p(a_1)\otimes \CC L_{p^{2n+1}}(a_2)\big)^{\Gamma_0(p^{2n+1})}= \big(\CC L_p(a_1)\otimes \CC L_{p^{2n-1}}(a_2)\big)^{\Gamma_0(p^{2n-1})} \oplus \CC \big(\mathbf{e}^0\otimes \mathbf{e}^0\big).         
\end{equation}
If $\big(\frac{-a_1a_2}{p}\big)=-1$, then
$$ 
\big(\CC L_p(a_1)\otimes \CC L_{p}(a_2)\big)^{\Gamma_0(p)}=\CC \big(\mathbf{e}^0\otimes \mathbf{e}^0\big). 
$$
If $\big(\frac{-a_1a_2}{p}\big)=1$, that is, $a_1=-a_2x^2 \bmod p$ for some integer $x$, then
$$ 
\big(\CC L_p(a_1)\otimes \CC L_{p}(a_2)\big)^{\Gamma_0(p)}=\CC \big(\mathbf{e}^0\otimes \mathbf{e}^0\big)\oplus \CC\left(\sum_{l=1}^{p-1}(\mathbf{e}^l\otimes \mathbf{e}^{lx})\right) \oplus \CC\left(\sum_{l=1}^{p-1}(\mathbf{e}^l\otimes \mathbf{e}^{-lx})\right).  
$$
In particular, for any non-negative integer $n$,
\begin{equation}\label{eq:CLP-CLP-vanish}
\big(\CC L_p^-(a_1)\otimes \CC L_{p^{2n+1}}^-(a_2)\big)^{\Gamma_0(p^{2n+1})}=0
\end{equation}
if and only if $\big(\frac{-a_1a_2}{p}\big)=-1$.
\end{lemma}

\begin{proof}
We first consider the even case. We set $I_0:=\sum_{g\in \Gamma_0(p^{2n}) / \Gamma(p^{2n})} \rho(g)$. Let $\gamma_1\in L_p(a_1)$ and $\gamma_2 \in L_{p^{2n}}(a_2)$. Then the sum 
$$ 
\sum_{l=1}^{p^{2n}}T^l (\mathbf{e}^{\gamma_1}\otimes \mathbf{e}^{\gamma_2})=\sum_{l=1}^{p^{2n}}e\Big(\frac{l(a_1p^{2n-1}\gamma_1^2+a_2\gamma_2^2)}{p^{2n}}\Big)\mathbf{e}^{\gamma_1}\otimes \mathbf{e}^{\gamma_2}
$$
does not vanish if and only if $\gamma_1=0$ and $p^n\mid\gamma_2$. Assume that $\gamma_2=p^sk$ with $n\leq s\leq 2n$ and $\mathrm{gcd}(p,k)=1$. We show that $I_0(\mathbf{e}^0\otimes \mathbf{e}^0)=0$ and
$$
I_0(\mathbf{e}^0\otimes \mathbf{e}^{kp^{2n-1}})= c_0\left(\mathbf{e}^0\otimes \sum_{l=1}^{p-1}\left(\frac{l}{p}\right)\mathbf{e}^{lp^{2n-1}}\right) 
$$
for some non-zero constant $c_0$. By taking $n=1$, we get \eqref{eq:CLPCLP2}. If $s<2n-1$, then as in Lemma \ref{lem:CLP^N} we show that 
$$
I_0(\mathbf{e}^0\otimes \mathbf{e}^{\gamma_2})=\big(\CC L_p(a_1)\otimes \CC L_{p^{2n-2}}(a_2)\big)^{\Gamma_0(p^{2n-2})}.
$$ 
We thus prove \eqref{eq:CLPCLP^N_even_induction}. 

We then consider the odd case. We set $I_1:=\sum_{g\in \Gamma_0(p^{2n+1}) / \Gamma(p^{2n+1})} \rho(g)$. Let $\gamma_1\in L_p(a_1)$ and $\gamma_2 \in L_{p^{2n+1}}(a_2)$. Then the sum
$$ 
\sum_{l=1}^{p^{2n+1}}T^l (\mathbf{e}^{\gamma_1}\otimes \mathbf{e}^{\gamma_2})=\sum_{l=1}^{p^{2n+1}}e\Big(\frac{l(a_1p^{2n}\gamma_1^2+a_2\gamma_2^2)}{p^{2n+1}}\Big)
$$
does not vanish if and only if either $\gamma_1=0$ and $p^{n+1}\mid \gamma_2$, or $\gamma_1\neq 0$, $a_1=-a_2x^2\bmod p$ for some integer $x$, and $\gamma_2=\pm p^nx\gamma_1$.   

In the first case, we write $\gamma_2=p^sk$ with $n+1\leq s \leq 2n+1$ and $\mathrm{gcd}(p,k)=1$. If $n+1\leq s\leq 2n-1$, then as before we conclude 
$$
I_1(\mathbf{e}^0\otimes \mathbf{e}^{\gamma_2})\in \big(\CC L_p(a_1)\otimes \CC L_{p^{2n-1}}(a_2)\big)^{\Gamma_0(p^{2n-1})}.
$$
For the other $s$, we verify that
$$ 
I_1(\mathbf{e}^0\otimes \mathbf{e}^0)=c_1\big(\mathbf{e}^0\otimes \mathbf{e}^0\big) \quad \text{and} \quad I_1(\mathbf{e}^0\otimes \mathbf{e}^{p^{2n}k})=c'_1\left(\sum_{l=1}^{p-1} \mathbf{e}^0\otimes \mathbf{e}^{lp^{2n}}\right)
$$
for some non-zero constants $c_1$ and $c_1'$. We note that
$$
\sum_{l=1}^p \mathbf{e}^0\otimes \mathbf{e}^{lp^{2n}} \in \big(\CC L_p(a_1)\otimes \CC L_{p^{2n-1}}(a_2)\big)^{\Gamma_0(p^{2n-1})}.
$$

In the last case, we confirm that
$$
I_1(\mathbf{e}^{\gamma_1}\otimes \mathbf{e}^{\epsilon p^n x \gamma_1}) = c \sum_{l=1}^{p-1} \Big( \mathbf{e}^l\otimes \sum_{j=1}^{p^{2n}} \mathbf{e}^{\epsilon p^nx(l+jp)} \Big), \quad \epsilon\in\{+,-\}
$$
for some non-zero constant $c$, and therefore it is identical to the element $\sum_{j=1}^{p-1}\mathbf{e}^l\otimes \mathbf{e}^{\epsilon lx}$ in the space $(\CC  L_p(a_1)\otimes \CC L_p(a_2))^{\Gamma_0(p)}$.  We thus prove \eqref{eq:CLPCLP^N_odd_induction} and determine $(\CC L_p(a_1)\otimes \CC L_{p}(a_2))^{\Gamma_0(p)}$ by combining the above two cases together. By \eqref{eq:CLPCLP^N_odd_induction} and the fact
$$ 
\sum_{l=1}^{(p-1)/2}\big((\mathbf{e}^l-\mathbf{e}^{-l})\otimes (\mathbf{e}^{lx}-\mathbf{e}^{-lx})\big)\in \big(\CC L_p^-(a_1)\otimes \CC L_{p}^-(a_2)\big)^{\Gamma_0(p)}, \quad  a_1=-a_2x^2 \, \bmod p,
$$
we prove the last claim. 
\end{proof}

\begin{lemma}\label{lem:oddinduction}
Let $k_1$, $k_2$, $k_3$ be positive integers such that $k_1+k_2$ is odd, $k_1 > k_2$ and $k_1 >k_3$. Let $a_1$, $a_2$ be integers that are coprime to a prime $p$. If $p$ is odd, then
$$
\big(\CC L_{p^{k_1}}(a_1 )\otimes\CC L_{p^{k_2}}(a_2)\big)^{\Gamma_0(p^{k_3})}=
\big(\CC L_{p^{k_1-2}}(a_1 )\otimes\CC L_{p^{k_2}}(a_2)\big)^{\Gamma_0(p^{k_3})}.
$$  
If $p=2$, then for $r\in \{ 0,1,2\}$,
$$
\big(\CC D_{2^{k_1}}(a_1)\otimes\CC D_{2^{k_2}}(a_2)\big)^{\Gamma_0(2^{k_3+r})}= \\
\big(\CC D_{2^{k_1-2}}(a_1)\otimes\CC D_{2^{k_2}}(a_2)\big)^{\Gamma_0(2^{k_3+r})}.
$$
\end{lemma} 

\begin{proof}
We first consider the odd case. Suppose $\omega=\sum_{\gamma_1\in L_{p^{k_1}}(a_1)}\sum_{\gamma_2 \in L_{p^{k_2}}(a_2)} c_{\gamma_1,\gamma_2} \mathbf{e}^{\gamma_1}\otimes \mathbf{e}^{\gamma_2}$ is invariant under $\Gamma_0(p^{k_3})$. As explained at the end of Section \ref{sec:Weil-rep}, $\omega\in (\CC L_{p^{k_1-2}}(a_1)\otimes\CC L_{p^{k_2}}(a_2))^{\Gamma_0(p^{k_3})}$ if $c_{\gamma_1,\gamma_2}=0$  for $p\nmid \gamma_1$ and  $c_{\gamma_1,\gamma_2}=c_{\gamma_1+lp^{k_1-1},\gamma_2}$ for any $1 \leq l \leq p$. This can be deduced by acting $\sum_{l=1}^{p^{k_1}}T^l$ and $\sum_{l=1}^{p}ST^{p^{k_1-1}l}S$ on $\omega$ respectively. The details are as follows.
	
It is clear that $\omega=\frac{1}{p^{k_1}} \sum_{l=1}^{p^{k_1}}T^l \omega$ and the sum 
$$
\sum_{l=1}^{p^{k_1}}T^l (\mathbf{e}^{\gamma_1}\otimes \mathbf{e}^{\gamma_2})=\sum_{l=1}^{p^{k_1}}e\left(\frac{l(a_1\gamma_1^2+p^{k_1-k_2}a_2\gamma_2^2)}{p^{k_1}}\right) \mathbf{e}^{\gamma_1}\otimes \mathbf{e}^{\gamma_2}
$$
does not vanish if and only if $p^{k_1}\mid a_1\gamma_1^2+p^{(k_1-k_2)}a_2\gamma_2^2$, or equivalently, $p^{k_1}\mid \gamma_1^2$ and $p^{k_2} \mid \gamma_2^2$, since $k_1-k_2$ is odd. It follows that $c_{\gamma_1,\gamma_2}=0$ if $p \nmid \gamma_1$.
	
By Lemma \ref{lem:oddformula}, we confirm that $\omega=\frac{1}{p}\sum_{l=1}^{p}ST^{p^{k_1-1}l}S\omega$, and 
$$
ST^{p^{k_1-1}l}S \mathbf{e}^{\gamma_2}=\frac{\epsilon}{p^{k_2}}\sum_{\alpha, \, \beta \in L_{p^{k_2}}(a_2)}e\left( \frac{a_2(-2\gamma_2 \alpha-2\alpha \beta)}{p^{k_2}}\right)\mathbf{e}^\beta 
=\epsilon \mathbf{e}^{-\gamma_2}
$$
for some $\epsilon\in \{\pm 1\}$, and
\begin{align*}
\frac{1}{p}\sum_{l=1}^{p}ST^{p^{k_1-1}l}S \mathbf{e}^{\gamma_1} &= \frac{\epsilon'}{p^{k_1+1}} \sum_{\alpha, \, \beta \in L_{p^{k_1}}(a_1) }\sum_{l=1}^{p}e\left( \frac{a_1(-2\gamma_1 \alpha-2\alpha \beta+ lp^{k_1-1}\alpha^2)}{p^{k_1}}\right)\mathbf{e}^\beta\\
&=\frac{\epsilon'}{p^{k_1}}\sum_{\substack{\alpha, \, \beta \in L_{p^{k_1}}(a_1)\\  p \mid \alpha}}e\left( \frac{a_1(-2\gamma_1 \alpha-2\alpha \beta)}{p^{k_1}}\right)\mathbf{e}^\beta\\
&=\frac{\epsilon'}{p}\sum_{l=1}^{p}\mathbf{e}^{-\gamma_1+lp^{k_1-1}}
\end{align*} 
for some $\epsilon'\in \{\pm 1\}$. 
Therefore,
$$
\frac{1}{p}\sum_{l=1}^{p}ST^{p^{k_1-1}l}S (\mathbf{e}^{\gamma_1}\otimes \mathbf{e}^{\gamma_2})=\frac{\epsilon\epsilon'}{p}\left(\sum_{l=1}^{p}\mathbf{e}^{-\gamma_1+p^{k_1-1}l}\otimes \mathbf{e}^{-\gamma_2}\right),
$$
which yields that $c_{\gamma_1,\gamma_2}=c_{\gamma_1+lp^{k_1-1},\gamma_2}$ for any $1 \leq l \leq p$. This proves the odd case.

We then consider the even case. Suppose that $\omega=\sum_{\gamma_1\in D_{2^{k_1}}(a_1)}\sum_{\gamma_2 \in D_{2^{k_2}}(a_2)} c_{\gamma_1,\gamma_2} \mathbf{e}^{\gamma_1}\otimes \mathbf{e}^{\gamma_2}$ is invariant under $\Gamma_0(2^{k_3+2})$. Then $\omega=\frac{1}{2^{k_1+2}} \sum_{l=1}^{2^{k_1+2}}T^l \omega$ and the sum
$$
\sum_{l=1}^{2^{k_1+2}}T^l (\mathbf{e}^{\gamma_1}\otimes \mathbf{e}^{\gamma_2})=\sum_{l=1}^{2^{k_1+2}}e\left(\frac{l(a_1\gamma_1^2+2^{k_1-k_2}a_2\gamma_2^2)}{2^{k_1+2}}\right) \mathbf{e}^{\gamma_1}\otimes \mathbf{e}^{\gamma_2}
$$ 
does not vanish if and only if $2^{k_1+2}\mid a_1\gamma_1^2+2^{(k_1-k_2)}a_2\gamma_2^2$. It follows that $c_{\gamma_1,\gamma_2}=0$ if $2\nmid \gamma_1$. Besides, $\frac{1}{2}(-I+ST^{2^{k_1+1}}S)\omega=\omega$ and we verify by Lemma \ref{lem:evenformula} that
\begin{align*}
\frac{1}{2}\big(-I+ST^{2^{k_1+1}}S\big)\mathbf{e}^{\gamma_2}&=(-i)^{a_2}\mathbf{e}^{-\gamma_2},\\    
\frac{1}{2}\big(-I+ST^{2^{(k_1+1)}}S\big)\mathbf{e}^{\gamma_1}&=\frac{(-i)^{a_1}}{2}\big(\mathbf{e}^{-\gamma_1}+\mathbf{e}^{-\gamma_1+2^{k_1}}\big).
\end{align*}
Therefore, 
$$
\frac{1}{2}\big(-I+ST^{2^{k_1+1}}S\big)(\mathbf{e}^{\gamma_1}\otimes\mathbf{e}^{\gamma_2})= \frac{(-i)^{a_1+a_2}}{2}\big(\mathbf{e}^{-\gamma_1}+\mathbf{e}^{-\gamma_1+2^{k_1}}\big)\otimes \mathbf{e}^{-\gamma_2},
$$
which yields that $c_{\gamma_1,\gamma_2}=c_{\gamma_1+2^{k_1},\gamma_2}$. We then prove that $\omega \in \CC D_{2^{k_1-2}}(a_1)\otimes\CC D_{2^{k_2}}(a_2)$. This proves the even case for $r=2$. The proof for $r=0,1$ is similar and we omit it. 
\end{proof}

Combining \eqref{eq:CL_P1} and the above lemma, we prove the following result. 

\begin{lemma}\label{lem:oddresult}
Let $k_1,k_2$ be positive integers with $k_1+k_2$ being odd. Let $a_1,a_2$ be integers that are coprime to an odd prime $p$. Then
$$ 
\big(\CC L_{p^{k_1}}(a_1)\otimes\CC L_{p^{k_2}}(a_2)\big)^{\Gamma_0(p)}=0.  
$$
\end{lemma}

\begin{lemma}\label{lem:eveninduction}
Let $k_1$, $k_2$, $k_3$ be positive integers such that $k_1+k_2$ is even, $k_1 \geq k_2$ and $k_1 > k_3$. Let $a_1$, $a_2$ be odd integers such that $a_1a_2$ is a square modulo $4$. Then 
$$
\big(\CC D_{2^{k_1}}(a_1)\otimes\CC D_{2^{k_2}}(a_2)\big)^{\Gamma_0(2^{k_3})}=\big(\CC D_{2^{k_1-2}}(a_1)\otimes\CC D_{2^{k_2}}(a_2)\big)^{\Gamma_0{(2^{k_3})}}.
$$
\end{lemma}

\begin{proof}
Let us first consider the particular case where $k_1=k_2$. For any $\Gamma_0(2^{k_3})$-invariant element $\omega$, it has the form of $\omega=\sum_{\gamma_1\in D_{2^{k_1}}(a_1)}\sum_{\gamma_2 \in D_{2^{k_2}}(a_2)} c_{\gamma_1,\gamma_2} \mathbf{e}^{\gamma_1}\otimes \mathbf{e}^{\gamma_2}$. 

We then have  $\omega=\frac{1}{2^{k_1+2}} \sum_{l=1}^{2^{k_1+2}}T^l \omega$. The sum
$$ 
\sum_{l=1}^{2^{k_1+2}}T^l (\mathbf{e}^{\gamma_1}\otimes \mathbf{e}^{\gamma_2})=\sum_{l=1}^{2^{k_1+2}}e\left(\frac{l(a_1\gamma_1^2+a_2\gamma_2^2)}{2^{k_1+2}}\right)\mathbf{e}^{\gamma_1}\otimes \mathbf{e}^{\gamma_2} 
$$
does not vanish only if $2\mid \gamma_1$ and $2\mid \gamma_2$, since $a_1a_2$ is a square modulo $4$ by assumption. Hence $c_{\gamma_1,\gamma_2}=0$ if $2\nmid \gamma_1$ or $2\nmid \gamma_2$. We now assume that $2|\gamma_1$ and $2|\gamma_2$. By Lemma \ref{lem:evenformula}, for $t=1,2$ and any integer $m$ we derive
\begin{align*}
ST^{m2^{k_t-1}}S \mathbf{e}^{\gamma_t}=&\frac{(-i)^{a_t}}{2^{k_t+1}}\sum_{\alpha, \, \beta \in D_{2^{k_t}}(a_t)} e\left(\frac{a_t(-2\gamma_t\alpha-2\alpha\beta+m2^{k_t-1}\alpha^2)}{2^{k_t+2}}\right) \mathbf{e}^\beta\\
=&\frac{(-i)^{a_t}}{2^{k_t+1}}\sum_{\substack{\alpha, \, \beta \in D_{2^{k_t}}(a_t)\\ \alpha \, \mathrm{odd}}} e\left(\frac{ma_t}{8}\right) e\left(\frac{a_t(-2\gamma_t\alpha-2\alpha\beta)}{2^{k_t+2}}\right) \mathbf{e}^\beta \\
&+ \frac{(-i)^{a_t}}{2^{k_t+1}}\sum_{\substack{\alpha, \, \beta \in D_{2^{k_t}}(a_t)\\ \alpha =2,6 \, \bmod 8}} e\left(\frac{ma_t}{2}\right) e\left(\frac{a_t(-2\gamma_t\alpha-2\alpha\beta)}{2^{k_t+2}}\right) \mathbf{e}^\beta \\
&+\frac{(-i)^{a_t}}{2^{k_t+1}}\sum_{\substack{\alpha, \, \beta \in D_{2^{k_t}}(a_t)\\ \alpha =0,4 \, \bmod 8}} e\left(\frac{a_t(-2\gamma_t\alpha-2\alpha\beta)}{2^{k_t+2}}\right) \mathbf{e}^\beta \\
=& \frac{(-i)^{a_t}}{2}e\left(\frac{ma_t}{8}\right) \mathbf{e}^{-\gamma_t}+\frac{(-i)^{a_t}}{4} e\left(\frac{ma_t}{2}\right)\left( \mathbf{e}^{-\gamma_t}+\mathbf{e}^{-\gamma_t+2^{k_t}} \right)\\
&+\frac{(-i)^{a_t}}{4}\left(\sum_{l=0}^3 \mathbf{e}^{-\gamma_t+l2^{k_t-1}}\right) \\
=:& e\left(\frac{ma_t}{8}\right) E_0^t + e\left(\frac{ma_t}{2}\right) E_1^t + E_2^t.
\end{align*}
We note that
$$ 
E_1^t=\mathbf{e}^{-\gamma_t}+\mathbf{e}^{-\gamma_t+2^{k_t}}\in \CC D_{2^{k_t-2}}(a_t), \quad E_2^t=\sum_{l=0}^3 \mathbf{e}^{-\gamma_t+l2^{k_t-1}} \in \CC D_{2^{k_t-2}}(a_t). 
$$ 
The coefficient of $E_0^1\otimes E_0^2$ in $ST^{m2^{k_t-1}}S (\mathbf{e}^{\gamma_1}\otimes \mathbf{e}^{\gamma_2})$ is $e\big(\frac{ma_1+ma_2}{8}\big)$ with $a_1+a_2=2\bmod 4$. Hence the coefficient of  $E_0^1\otimes E_0^2 $ in  $ \sum_{m=0}^7 ST^{m2^{k_t-1}}S (\mathbf{e}^{\gamma_1}\otimes \mathbf{e}^{\gamma_2})$ is $\sum_{m=0}^7 e\big(\frac{ma_1+ma_2}{8}\big)$ and thus vanishes. Similarly, we show that the coefficients $E_0^1\otimes E_1^2$, $E_0^1\otimes E_2^2$, $E_1^1\otimes E_0^2$ and $E_2^1\otimes E_0^2$ also vanish in $ST^{m2^{k_t-1}}S (\mathbf{e}^{\gamma_1}\otimes \mathbf{e}^{\gamma_2})$. From $\omega=\frac{1}{8}\sum_{m=0}^7 ST^{m2^{k_t-1}}S \omega$ and the above discussion we then deduce that $\omega\in (\CC D_{2^{k_1-2}}(a_1)\otimes\CC D_{2^{k_2}}(a_2))^{\Gamma_0{(2^{k_3})}}$. This proves the desired lemma when $k_1=k_2$. 
 
We now assume that $k_1>k_2$. It is clear that
$$
\big(\CC D_{2^{k_1}}(a_1)\otimes\CC D_{2^{k_2}}(a_2)\big)^{\Gamma_0(2^{k_3})} \subset  \big(\CC D_{2^{k_1}}(a_1)\otimes\CC D_{2^{k_1}}(a_2)\big)^{\Gamma_0(2^{k_3})}.
$$
However, we have thus proved that
$$ 
\big(\CC D_{2^{k_1}}(a_1)\otimes\CC D_{2^{k_1}}(a_2)\big)^{\Gamma_0(2^{k_3})} = \big(\CC D_{2^{(k_1-2)}}(a_1)\otimes\CC D_{2^{k_1}}(a_2)\big)^{\Gamma_0(2^{k_3})}.
$$
It follows that
$$
\big(\CC D_{2^{k_1}}(a_1)\otimes\CC D_{2^{k_2}}(a_2)\big)^{\Gamma_0(2^{k_3})}\subset \big(\CC D_{2^{(k_1-2)}}(a_1)\otimes\CC D_{2^{k_2}}(a_2)\big)^{\Gamma_0{(2^{k_3})}}, 
$$ 
which yields the desired formula. 
 \end{proof}

By virtue of Lemma \ref{lem:oddinduction} and Lemma \ref{lem:eveninduction}, we can improve Lemma \ref{lem:CD16-+}. 

\begin{lemma}\label{lem:evenresult}
Let $k_1$, $k_2$ be non-negative integers and $a_1$, $a_2$ be odd integers. If $a_1a_2$ is a square modulo $4$, then 
$$ 
\big(\CC D_{2^{k_1}}(a_1)\otimes\CC D_{2^{k_2}}(a_2)\big)^{\Gamma_0(16)}=0. 
$$
\end{lemma}

\begin{proof}
If $k_1+k_2$ is odd, then it suffices to consider the case where $k_1,k_2\leq 2$ by Lemma \ref{lem:oddinduction}. If $k_1+k_2$ is even, then we only need to consider the case $k_1,k_2\leq 4$ by Lemma \ref{lem:eveninduction}. The result then follows from Lemma \ref{lem:CD16-+} and Lemma \ref{lem:vanish-sign-2}.  
\end{proof}

\section{Existence of Jacobi Forms of weight one}\label{sec:existence}
It is known that there exist Jacobi forms of weight one on $\Gamma_0(N)$. For example, the following Jacobi forms of weight one were constructed in \cite[Section 1]{CDH18}:
\begin{align*}
\eta(6\tau)\vartheta(6\tau,12z) &\in  J_{1,12}(36), \\
\frac{\eta(8\tau)^2}{\eta(4\tau)}\vartheta(4\tau,8z) &\in J_{1,8}(32),\\
\frac{\vartheta(3\tau,3az)\vartheta(3\tau,3bz)\vartheta(3\tau,3(a+b)z)}{\eta(3\tau)} & \in J_{1,3(a^2+ab+b^2)}(9),\\
\theta_{3,3}(\tau,0)\theta_{9,3}^-(\tau,z)+\theta_{3,0}(\tau,0)\theta_{9,6}^-(\tau,z) &\in J_{1,9}(36),
\end{align*}
(the first form above is cuspidal, whereas the others are not cuspidal), where $a,b$ are positive integers and these basic functions are defined as
\begin{align*}
\eta(\tau)&=q^{\frac{1}{24}}\prod_{n=1}^\infty(1-q^n),\\
\vartheta(\tau,z)&=q^{\frac{1}{8}}(\zeta^{\frac{1}{2}}-\zeta^{-\frac{1}{2}})\prod_{n=1}^\infty(1-q^n\zeta)(1-q^n\zeta^{-1})(1-q^n),\\
\theta_{m,r}(\tau,z)&=\sum_{k\in\ZZ}q^{\frac{(2km+r)^2}{4m}}\zeta^{2km+r},\\
\theta_{m,r}^+(\tau,z)&= \theta_{m,-r}(\tau,z)+\theta_{m,r}(\tau,z),\\
\theta_{m,r}^-(\tau,z)&= \theta_{m,-r}(\tau,z)-\theta_{m,r}(\tau,z),
\end{align*}
and we note that for any positive integer $h$,
$$
\theta_{m,r}(h\tau,hz)=\theta_{mh,rh}(\tau,z). 
$$

In this section, we construct more Jacobi forms of weight one by means of Lemmas \ref{lem:CD1_a1a2_square}--\ref{lem:CLPCLP^2N}. 

\begin{lemma}\label{lem:J_1,P2(P2)}
Let $p$ be an odd prime with $p=3\bmod 4$. Then  $J_{1,p^2}(p^2)$ is of dimension one and  generated by 
$$ 
\theta_{p,p}(\tau,0)\sum_{\substack{l=1\\ l \, \mathrm{odd}}}^{p-2}   \left(\frac{l}{p}\right) \theta_{p^2, lp}^{-}(\tau,z) + \theta_{p,0}(\tau,0)\sum_{\substack{l=2 \\ l \, \mathrm{even}}}^{p-1} \left(\frac{l}{p}\right) \theta_{p^2, lp}^{-}(\tau,z). 
$$
Moreover, the space $J_{1,p^2}(4p^2)$ is of dimension two and generated by
$$ 
\theta_{p,p}(\tau,0)\sum_{\substack{l=1\\ l \, \mathrm{odd}}}^{p-2}   \left(\frac{l}{p}\right) \theta_{p^2, lp}^{-}(\tau,z) \pm \theta_{p,0}(\tau,0)\sum_{\substack{l=2 \\ l \, \mathrm{even}}}^{p-1} \left(\frac{l}{p}\right) \theta_{p^2, lp}^{-}(\tau,z).
$$
\end{lemma}

\begin{proof}
By Proposition \ref{prop:main-iso}, the space $J_{1,p^2}(p^2)$ is a direct sum of
$$ 
\big(\CC D_1^+(-1)\otimes\CC D_{1}^+(-1)\big)^{\SL_2(\ZZ)} \bigotimes \big( \CC L_{p^2}^-(a)\otimes  \CC L_{p^2}^+(a) \big)^{\Gamma_0(p^2)} 
$$ 
and 
$$ 
\big(\CC D_1^+(-1)\otimes\CC D_{1}^+(1)\big)^{\SL_2(\ZZ)} \bigotimes \big( \CC L_{p^2}^-(a)\otimes  \CC L_{p}^+(a') \big)^{\Gamma_0(p^2)},
$$ 
where $a$ is an inverse of $-4$ modulo $p^2$ and $a'$ is an inverse of $-4$ modulo $p$.
The first summand is trivial by Lemma \ref{lem:irreducible_D2^N_LP^N}, while the latter one is of dimension one by \eqref{eq:CD1_Gamma1} and \eqref{eq:CLPCLP2}.
	
Similarly, by Proposition \ref{prop:main-iso} and \eqref{eq:CD1_CD16-_vanish}, the space $J_{1,p^2}(4p^2)$ is isomorphic to 
$$ 
\big(\CC D_1^+(-1)\otimes\CC D_{1}^+(1)\big)^{\Gamma_0(4)} \bigotimes \big( \CC L_{p^2}^-(a)\otimes  \CC L_{p}^+(a') \big)^{\Gamma_0(p^2)},
$$
which is of dimension two by \eqref{eq:CD1_Gamma4} and \eqref{eq:CLPCLP2}. Combining \eqref{eq:iso-local-global}, \eqref{eq:CD1_Gamma1} and \eqref{eq:CLPCLP2}, we verify that $J_{1,p^2}(p^2)$ is generated by
$$
\sum_{\substack{l=1\\ l \, \mathrm{odd}}}^{p-2}\left(\left( \frac{l}{p}\right) \mathbf{e}^{lp}\otimes \mathbf{e}^p + \left( \frac{l}{p}\right) \mathbf{e}^{lp+p^2}\otimes \mathbf{e}^0\right) + \sum_{\substack{l=2 \\ l \, \mathrm{even}}}^{p-1} \left( \left( \frac{l}{p}\right) \mathbf{e}^{lp+p^2}\otimes \mathbf{e}^p + \left( \frac{l}{p}\right) \mathbf{e}^{lp}\otimes \mathbf{e}^0\right), 
$$ 
and $J_{1,p^2}(4p^2)$ is generated by the above element and by 
$$
\sum_{\substack{l=1\\ l \, \mathrm{odd}}}^{p-2}\left(\left( \frac{l}{p}\right) \mathbf{e}^{lp}\otimes \mathbf{e}^p - \left( \frac{l}{p}\right) \mathbf{e}^{lp+p^2}\otimes \mathbf{e}^0\right) + \sum_{\substack{l=2 \\ l \, \mathrm{even}}}^{p-1} \left( \left( \frac{l}{p}\right) \mathbf{e}^{lp+p^2}\otimes \mathbf{e}^p - \left( \frac{l}{p}\right) \mathbf{e}^{lp}\otimes \mathbf{e}^0\right). 
$$ 
The result then follows from identification \eqref{eq:identification}.
\end{proof}
By taking $p=3$, the above result recovers the other construction of $J_{1,9}(9)$ in \cite{CDH18}, that is,
$$
\theta_{3,3}(\tau,0)\theta_{9,3}^-(\tau,z)-\theta_{3,0}(\tau,0)\theta_{9,6}^-(\tau,z) \in J_{1,9}(9).
$$

We are interested in whether $J_{1,m}(N)=0$ holds when $(m,N)=1$. The following two results show that it  does not hold.

\begin{lemma}\label{lem:J_1,q(p^3)}
Let $p$ be an odd prime with $p=3 \bmod 4$. For any prime $q$ with $\left(\frac{-p}{q}\right)=1$, the space $J_{1,q}(p^3)$ is non-trivial.
\end{lemma}

\begin{proof}
We first assume that $q$ is odd. By Proposition \ref{prop:main-iso}, the space
$$ 
V:=\big(\CC D_1^+(-q)\otimes \CC D_1^+(q)\big)^{\SL_2(\ZZ)}\otimes \big(\CC L_q^-(a_q)\otimes\CC L_q^-(a_q')\big)^{\SL_2(\ZZ)}\otimes \big(\CC L_{p^3}^-(a_p)\big)^{\Gamma_0(p^3)} 
$$
lies in $J_{1,q}(p^3)$, where $a_q$ and $a_q'$ is an inverse of $-4 $ and $-4p^3$ modulo $q$ respectively, and $a_p$ is an inverse of $-4q$ modulo $p^3$. By Lemma \ref{lem:irreducible_D2^N_LP^N}, the $2$-part and $q$-part of the above space is non-trivial. By \eqref{eq:CL_P3}, we have
$$
\big(\CC L_{p^3}(a_p)\big)^{\Gamma_0(p^3)}=\big(\CC L_{p^3}^-(a_p)\big)^{\Gamma_0(p^3)}=\CC \left(\sum_{l=1}^{p-1}\left(\frac{l}{p}\right)\mathbf{e}^{p^2l}\right). 
$$ 
Therefore, the above space $V$ is non-trivial. This proves the lemma in the odd case.

We now assume that $q=2$. Then $p=3\bmod 4$ and $\big(\frac{-p}{2}\big)=1$ imply that $p=-1\bmod 8$. By Proposition \ref{prop:main-iso}, the space
$$ 
V_1:=\big(\CC D_2^-(a_2)\otimes \CC D_2^-(a_2')\big)^{\SL_2(\ZZ)} \otimes \big(\CC L_{p^3}^-(a_p)\big)^{\Gamma_0(p^3)} 
$$
lies in $J_{1,2}(p^3)$, where $-a_2=1\bmod 8$, $a_2'=1\bmod 8$ and $a_p$ is an inverse of $-8$ modulo $p^3$. By Lemma \ref{lem:irreducible_D2^N_LP^N} and \eqref{eq:CL_P3}, the space $V_1$ is non-trivial. We then prove the even case. 
\end{proof}

\begin{lemma}\label{lem:J_1P64}
For any odd prime $p$ with $p=1\bmod 4$, the space $J_{1,p}(64)$ is non-trivial.
\end{lemma}

\begin{proof}
By Proposition \ref{prop:main-iso}, the space
$$ 
V:=\big(\CC D_1^+(-p)\otimes \CC D_{16}^-(a_2)\big)^{\Gamma_0(64)}\otimes \big(\CC L_p^-(a_p)\otimes\CC L_p^-(a_p')\big)^{\SL_2(\ZZ)} 
$$
lies in $J_{1,p}(64)$, where $a_p$ and $a_p'$ is an inverse of $-4 $ and $-64$ modulo $p$ respectively, and $a_2$ is an inverse of $-p$ modulo $64$. Hence $-pa_2$ is a square modulo $4$ and thus the $2$-part of $V$ is non-trivial by Lemma \ref{lem:CD1_a1a2_square}. The $p$-part of $V$ is also non-trivial by Lemma \ref{lem:irreducible_D2^N_LP^N}. We then prove the Lemma. 
\end{proof}

We now show that there are Jacobi forms of weight one and arbitrary index $m>1$. 

\begin{proposition}
For any positive integer $N$, $J_{1,1}(N)=0$. For $m>1$, there exists a positive integer $N_0$ such that $J_{1,m}(N_0)\neq 0$.
\end{proposition}

\begin{proof}
The 2-part of $J_{1,1}(N)$ is of the form
$$ 
\big(\CC D_{1}^{-}(a_2)\otimes \CC D_{m'_2}^{\epsilon_2'}(a_2') \big)^{\Gamma_0(b_2)}.
$$
Since $\CC D_{1}^{-}(a_2)=0$, $J_{1,1}(N)=0$.
	
By the Dirichlet theorem, we can choose a prime $q$ such that $q=7 \bmod 8$ and $\left(\frac{-q}{p}\right)=1$ for any odd prime $p\mid m$. We claim that $J_{1,m}(q^3)$ is non-trivial. 

If $m$ is even, then the space
\begin{align*}
V_0=& \big(\CC D_{m_2}^{-}(a_2)\otimes \CC D_{m_2}^{-}(q^{-3}a_2)\big)^{\SL_2(\ZZ)}\otimes \big(\CC L_{q^3}^-(a_q)\big)^{\Gamma_0(q^3)}  \otimes \\ 
& \bigotimes_{\substack{p\; \mathrm{odd} \\ p \mid m} } \big(\CC L_{m_p}^{+}(a_p)\otimes \CC L_{m_p}^{+}(q^{-3}a_p)\big)^{\SL_2(\ZZ)}
\end{align*}
lies in $J_{1,m}(q^3)$, where $m_p$ is the largest power of $p$ dividing $m$, $a_2 = (-m/m_2)^{-1} \bmod \, 4m_2$, $a_q = (-4m)^{-1} \, \bmod \, q^3$, and $a_p = (-4m/m_p)^{-1} \, \bmod \, m_p$ for odd $p\mid m$. We conclude from Lemma \ref{lem:irreducible_D2^N_LP^N} and \eqref{eq:CL_P3} that $V_0$ is non-trivial.

If $m$ is odd, then for some $p'\mid m$ the space 
\begin{align*}
V_1=& \big(\CC D_{1}^{+}(a_2)\otimes \CC D_{1}^{+}(q^{-3}a_2)\big)^{\SL_2(\ZZ)} \otimes \big(\CC L_{q^3}^-(a_q)\big)^{\Gamma_0(q^3)} \otimes \\ 
& \big(\CC L_{m_{p'}}^{-}(a_{p'})\otimes \CC L_{m_{p'}}^{-} (q^{-3}a_{p'})\big)^{\SL_2(\ZZ)} \otimes \\ 
& \bigotimes_{\substack{p\neq p' \; \mathrm{odd} \\ p \mid m} } \big(\CC L_{m_p}^{+}(a_p)\otimes \CC L_{m_p}^{+}(q^{-3}a_p)\big)^{\SL_2(\ZZ)}
\end{align*}
lies in $J_{1,m}(q^3)$ with the same notation as above. We conclude from Lemma \ref{lem:irreducible_D2^N_LP^N} and \eqref{eq:CL_P3} that $V_1$ is also non-trivial.

We then prove the claim, and thus prove the desired proposition. 
\end{proof}

\section{Vanishing of Jacobi forms of weight one}\label{sec:non-existence}

In this section, we establish the non-existence of Jacobi forms of weight one in some cases. For a positive integer $N$ and a prime $p$, we refer to the exponent of $p$ in the prime factorization of $N$ as the exponent of $p$ in $N$ for convenience. We first extend the Skoruppa theorem. 
 
\begin{theorem}\label{th:vanish_arbitrary_m}
Suppose that $N=2^np_1^{\alpha_1}...p_s^{\alpha_s}\cdot q_1^{\beta_1}...q_t^{\beta_t}$,
where $p_i$, $q_j$ are distinct odd primes and $p_i=1\,\bmod 4$ for any $1\leq i\leq s$, $q_j=3\,\bmod 4$ for any $1\leq j\leq t$. Then $J_{1,m}(N)=0$ for all positive integers $m$ if and only if $0\leq n\leq 4$ and $0\leq \beta_j \leq 1$ for any $1\leq j\leq t$. 
\end{theorem}

\begin{proof} 
If $32\mid N$, then $J_{1,8}(N)$ is non-trivial, since its subspace $J_{1,8}(32)$ is non-trivial as mentioned at the beginning of Section \ref{sec:existence}. If there exists $1\leq j\leq t$ with $\beta_j>1$, then $J_{1,q_j^2}(N)\supset J_{1,q_j^2}(q_j^2)$ is non-trivial by Lemma \ref{lem:J_1,P2(P2)}. 
	
We now assume that $0\leq n\leq 4$ and $\beta_j=1$ for $1\leq j\leq t$. Then $J_{1,m}(N)$ is the direct sum of the tensor products of type
$$ 
\big(\CC D_{m_2}^{\epsilon_2}(a_2)\otimes \CC D_{m'_2}^{\epsilon_2'}(a_2') \big)^{\Gamma_0(b_2)}\otimes\bigotimes_{\substack{p\; \mathrm{odd} \\ p \mid M} } \big(\CC L_{m_p}^{\epsilon_p}(a_p)\otimes \CC L_{m'_p}^{\epsilon_p'}(a_p')\big)^{\Gamma_0(b_p)},
$$
where $b_{q_j}=1$ for $1\leq j\leq t$ and $b_{2}\in \{1,2,4,8,16\}$. By Lemma \ref{lem:oddresult}, the exponent of $q_j$ in the product $m_{q_j}m_{q_j'}$ has to be even for all odd primes $q_j$ with $1\leq j\leq t$. For any prime $p\nmid N$, $b_p=1$ and thus the exponent of $p$ in $m_pm_p'$ is even by Lemma \ref{lem:irreducible_D2^N_LP^N}. In the above product, $-a_2(m/m_2)=1 \bmod 4m_2$ and $-a_2'(M/m_2')=1\bmod 4m_2'$. Hence
$$
a_2a_2'=\prod_{\mathrm{odd}\; p}(m_pm_p') \; \bmod 4
$$
is a square modulo 4, which implies  $(\CC D_{m_2}(a_2)\otimes\CC D_{m_2'}(a_2'))^{\Gamma_0(b_2)}=0$ by Lemma \ref{lem:evenresult}. It follows that $J_{1,m}(N)=0$ in this case. We thus prove the theorem. 
\end{proof}

By taking $N=1$ in the above theorem, we find that $J_{1,m}(1)=0$ for any positive integer $m$, which recovers Skoruppa's result in his thesis \cite{Sko85}.

We then establish some inequalities and equalities between the dimensions of spaces of Jacobi forms, which extend \cite[Theorem 1.2]{CDH18}. 

\begin{proposition}\label{prop:dimension_induction}
Let $m$, $N$ be positive integers and $q$ be a prime such that $\mathrm{gcd}(q,mN)=1$. For any non-negative integer $n$, we have
\begin{equation}\label{eq:main-inequality}
\dim J_{1,m}(q^n N) \geq \Big(\left[\frac{n}{2}\right]+1\Big) \dim J_{1,m}(N),   
\end{equation}
where $\big[n/2]$ denotes the integer part of $n/2$ as before. Moreover, if  $\mathrm{gcd}(q,mN)=1$, then equality holds for odd $q$ with $n\leq 2$ and for $q=2$ with $n\leq 5$ 
\end{proposition}

\begin{proof}
We first assume that $q$ is odd. From Proposition \ref{prop:main-iso} we conclude that 
\begin{align*}
J_{1,m}(N) \cong& \bigoplus_{\substack{m'\mid M \\ M/m' \, \text{squarefree}}} \big(\CC D_m^{-}(-1) \otimes\CC D_{m'}^{+}(-1)\big)^{\Gamma_0(N)},\\
J_{1,m}(q^n N) \cong& \bigoplus_{\substack{m'\mid M \\ M/m' \, \text{squarefree}}} \Big( \big(\CC D_m^{-}(-1)\otimes\CC D_{q^{n-1}m'}^{+}(-1)\big)^{\Gamma_0(q^{n}N)}\\
&\qquad \qquad \qquad  \oplus\big(\CC D_m^{-}(-1)\otimes\CC D_{q^nm'}^{+}(-1)\big)^{\Gamma_0(q^nN)}\Big),
\end{align*}
where $M$ is the smallest positive integer such that $m \mid M$ and $N \mid 4M$. 
With the same notation as in Proposition \ref{prop:main-iso}, the subspace $(\CC D_m^{-}(-1) \otimes\CC D_{m'}^{+}(-1))^{\Gamma_0(N)}$ has the decomposition
$$
\bigoplus_{\substack{\prod \epsilon_*=- \\ \prod \epsilon'_*=+}}\left( \big(\CC D_{m_2}^{\epsilon_2}(a_2)\otimes \CC D_{m'_2}^{\epsilon_2'}(a_2') \big)^{\Gamma_0(N_2)}\otimes\bigotimes_{\mathrm{odd} \; p \mid M} \big(\CC L_{m_p}^{\epsilon_p}(a_p)\otimes \CC L_{m'_p}^{\epsilon_p'}(a_p')\big)^{\Gamma_0(N_p)}\right). 
$$
Let $d$ denote $[n/2]$ and $q^{-1}$ denote an inverse of $q$ modulo $4M$. It follows that 
\begin{align*}
&\bigoplus_{\substack{\prod \epsilon_*=- \\ \prod \epsilon'_*=+}}\left( \big(\CC D_{m_2}^{\epsilon_2}(a_2)\otimes \CC D_{m'_2}^{\epsilon_2'}(q^{-2d}a_2') \big)^{\Gamma_0(N_2)}\otimes\bigotimes_{\mathrm{odd} \; p \mid M} \big(\CC L_{m_p}^{\epsilon_p}(a_p)\otimes \CC L_{m'_p}^{\epsilon_p'}(q^{-2d}a_p')\big)^{\Gamma_0(N_p)}\right) \\
&\quad \otimes \big( \CC L^+_{q^{2d}}\big)^{\Gamma_0(q^n)}    
\end{align*}
is a subspace of $(\CC D_m^{-}(-1)\otimes\CC D_{q^{2d}m'}^{+}(-1))^{\Gamma_0(q^{n}N)}$. 
By definition and Lemma \ref{lem:sign}, the map $\mathbf{e}^\gamma \mapsto \mathbf{e}^{b\gamma}$ defines an isomorphism  
$$
\CC D_{2^m}(a b^2) \cong \CC D_{2^m}(a) \qquad (\text{resp.} \; \CC L_{p^m}(a b^2) \cong \CC L_{p^m}(a)\,)
$$
between $\Mp_2(\ZZ)$-modules for $\mathrm{gcd}(2,ab)=1$ (resp. $\mathrm{gcd}(p,ab)=1$). In addition, we know from Lemma \ref{lem:CLP^N} that $\dim (\CC L^+_{q^{2d}})^{\Gamma_0(q^n)}=d+1$. Therefore, 
$$
\dim \big(\CC D_m^{-}(-1)\otimes\CC D_{q^{2d}m'}^{+}(-1)\big)^{\Gamma_0(q^{n}N)} \geq (d+1)\cdot \dim \big(\CC D_m^{-}(-1)\otimes\CC D_{m'}^{+}(-1)\big)^{\Gamma_0(N)},
$$
which yields the desired inequality. 

It is clear that 
$$\big(\CC D_m^{-}(-1)\otimes\CC D_{m'}^{+}(-1)\big)^{\Gamma_0(N)}=\big(\CC D_m^{-}(-1)\otimes\CC D_{m'}^{+}(-1)\big)^{\Gamma_0(qN)},
$$ 
since they have the same $p$-part for $p\neq q$ and the $q$-part of the latter one is trivial. The $q$-part of $(\CC D_m^{-}(-1)\otimes\CC D_{qm'}^{+}(-1))^{\Gamma_0(qN)}$ and $(\CC D_m^{-}(-1)\otimes\CC D_{qm'}^{+}(-1))^{\Gamma_0(q^2N)}$ are $(\CC L_q(a_q))^{\Gamma_0(q)}$ and $(\CC L_q(a_q))^{\Gamma_0(q^2)}$, respectively. We conclude from \eqref{eq:CL_P1} that
$$ 
\big(\CC D_m^{-}(-1)\otimes\CC D_{qm'}^{+}(-1)\big)^{\Gamma_0(qN)}=\big(\CC D_m^{-}(-1)\otimes\CC D_{qm'}^{+}(-1))^{\Gamma_0(q^2N)}=0.
$$
This proves 
$$
\dim J_{1,m}(q N) = \dim J_{1,m}(N).
$$
The $q$-part of the $(\CC D_m^{-}(-1)\otimes\CC D_{q^2m'}^{+}(-1))^{\Gamma_0(q^2N)}$ is $(\CC L_{q^2}(a_q))^{\Gamma_0(q^2)}$.  By Lemma \ref{lem:CLP^N}, the only non-trivial $q$-part is $(\CC L_{q^2}^+(a_q))^{\Gamma_0(q^2)}$, since $(\CC L_{q^2}^-(a_q))^{\Gamma_0(q^2)}=0$. From the discussion above, we then deduce that
$$
\dim J_{1,m}(q^2 N) = 2\dim J_{1,m}(N).
$$

We now assume that $q=2$. In this case, both $m$ and $N$ are odd. We fix the same decomposition of $J_{1,m}(N)$ as above. From Proposition \ref{prop:main-iso} we derive that 
\begin{align*}
J_{1,m}(2^n N) \cong& \bigoplus_{\substack{m'\mid M \\ M/m' \, \text{squarefree}}} \Big( \big(\CC D_m^{-}(-1)\otimes\CC D_{2^{n-3}m'}^{+}(-1)\big)^{\Gamma_0(2^{n}N)}\\
&\qquad \qquad \qquad  \oplus\big(\CC D_m^{-}(-1)\otimes\CC D_{2^{n-2} m'}^{+}(-1)\big)^{\Gamma_0(2^nN)}\Big)
\end{align*}
for $n\geq 3$ and 
$$
J_{1,m}(2^n N)\cong \bigoplus_{\substack{m'\mid M \\ M/m' \, \text{squarefree}}} \big(\CC D_m^{-}(-1) \otimes\CC D_{m'}^{+}(-1)\big)^{\Gamma_0(2^nN)} 
$$
for $n\leq 2$, where $M$ is the smallest positive integer such that $m \mid M$ and $N \mid 4M$. We fix the same decomposition of $(\CC D_m^{-}(-1) \otimes\CC D_{m'}^{+}(-1))^{\Gamma_0(N)}$ as above.  By  Lemma \ref{lem:CD1_a1a2_square} and Lemma \ref{lem:CD1_-a1a2_square}, the only possible non-trivial $2$-part of $(\CC D_m^{-}(-1)\otimes\CC D_{m'}^{+}(-1))^{\Gamma_0(2^n N)}$ is
$ (\CC D_{1}^+(a_2)\otimes \CC D_{1}^+(a_2') )^{\Gamma_0{(2^n)}} $ for $n=0,1,2$, where $-a_2a_2'$ is a square modulo $4$. 
By Lemma \ref{lem:CD1_-a1a2_square}, for $n=0,1,2$ we have
$$
\dim \big(\CC D_{1}^+(a_2)\otimes \CC D_{1}^+(a_2') \big)^{\Gamma_0{(2^n)}}=[n/2]+1. 
$$  
It then follows that for $n\leq 2$, 
$$
\dim J_{1,m}(2^n N) = ([n/2]+1)\dim J_{1,m}(N). 
$$
Let $n\geq 3$ and $d$ denote $[n/2]$ and $2^{-1}$ denote an inverse of $2$ modulo $M$. It follows that 
\begin{align*}
\big(\CC D_{1}^{+}(a_2)\otimes \CC D_{2^{2d-2}}^+(a_2') \big)^{\Gamma_0(2^n)}  \otimes \bigoplus_{\substack{\prod \epsilon_*=- \\ \prod \epsilon'_*=+}} \bigotimes_{\mathrm{odd} \; p \mid M} \big(\CC L_{m_p}^{\epsilon_p}(2^{2-2d}a_p)\otimes \CC L_{m'_p}^{\epsilon_p'}(2^{2-2d}a_p')\big)^{\Gamma_0(N_p)} 
\end{align*}
is a subspace of $(\CC D_m^{-}(-1)\otimes\CC D_{2^{2d-2}m'}^{+}(-1))^{\Gamma_0(2^{n}N)}$. Lemma \ref{lem:CD1_-a1a2_square} yields that
$$
\dim \big(\CC D_{1}^{+}(a_2)\otimes \CC D_{2^{2d-2}}^+(a_2')\big)^{\Gamma_0(2^n)}=d+1.
$$
We then prove the desired inequality as in the odd case. 

The $2$-part of $(\CC D_m^{-}(-1)\otimes\CC D_{2m'}^{+}(-1))^{\Gamma_0(8N)}$ and $(\CC D_m^{-}(-1)\otimes\CC D_{8m'}^{+}(-1))^{\Gamma_0(32N)}$ are respectively $(\CC D_1(a_2)\otimes \CC  D_2(a_2'))^{\Gamma_0(8)}$ and  $(\CC D_1(a_2)\otimes \CC D_8(a_2'))^{\Gamma_0(32)}$. Combining \eqref{eq:CD1_CD16+_vanish}, \eqref{eq:CD1_CD16-_vanish} and Lemma \ref{lem:reduce_level_from_p^n+1_to_p^n} together, we show that $(\CC D_m^{-}(-1)\otimes\CC D_{2m'}^{+}(-1))^{\Gamma_0(8N)}$, $(\CC D_m^{-}(-1)\otimes\CC D_{2m'}^{+}(-1))^{\Gamma_0(16N)}$ and $(\CC D_m^{-}(-1)\otimes\CC D_{8m'}^{+}(-1))^{\Gamma_0(32N)}$ vanish, since their $2$-parts vanish. 

By \eqref{eq:CD1_CD16-_vanish} and \eqref{eq:CD1_Gamma4}, the only possible non-trivial $2$-part of  $(\CC D_m^{-}(-1)\otimes\CC D_{m'}^{+}(-1))^{\Gamma_0(8N)}$ is $(\CC D_{1}^+(a_2)\otimes \CC D_{1}^+(a_2') )^{\Gamma_0(4)}$ with $-a_2a_2'$ being a square modulo $4$. 
Similarly, by \eqref{eq:CD1_CD16-_vanish} and \eqref{eq:CD1CD4_Gamma16}, the possible non-trivial $2$-part of $(\CC D_m^{-}(-1)\otimes\CC D_{4m'}^{+}(-1))^{\Gamma_0(16N)}$ is $ (\CC D_{1}^+(a_2)\otimes \CC D_{4}^+(a_2') )^{\Gamma_0(16)}$ with $-a_2a_2'$ being a square modulo $4$. From the discussion above, Lemma \ref{lem:reduce_level_from_p^n+1_to_p^n} and Lemma \ref{lem:CD1_-a1a2_square}, we deduce that
\begin{align*}
\dim J_{1,m}(2^n N) &= \dim \big( \CC D_1^+(a_2) \otimes \CC D_{2^{2d-2}}^+(a_2') \big)^{\Gamma_0(n)} \cdot \dim J_{1,m}(N) \\
&=(d+1)\dim J_{1,m}(N)
\end{align*}
for $3\leq n \leq 5$ as in the odd case. We then prove the proposition for $q=2$. 
\end{proof}

Recall that $J_{1,m}(1)=0$ for any positive integer $m$. By Lemma \ref{lem:J_1,q(p^3)} and Lemma \ref{lem:J_1P64}, the equality in \eqref{eq:main-inequality} does not hold for $q$ odd and $n>2$ or for $q=2$ and $n>5$ in general. Thus the bound of $n$ is sharp, when the equality holds.

We also extend the Ibukiyama--Skoruppa theorem \cite{IS07}. 

\begin{theorem}\label{th:vanish_m_coprime}
Suppose that $N=2^np_1^{\alpha_1}...p_s^{\alpha_s}\cdot q_1^{\beta_1}...q_t^{\beta_t}$, where $p_i$, $q_j$ are distinct odd primes and $p_i=1\,\bmod 4$ for any $1\leq i\leq s$, $q_j=3\,\bmod 4$ for any $1\leq j\leq t$. Then $J_{1,m}(N)=0$ for all positive integers $m$ with $\mathrm{gcd}(m,N)=1$ if and only if $0\leq n\leq 5$ and $0\leq \beta_j\leq 2$ for $1\leq j\leq t$.
\end{theorem}

\begin{proof}
If $\beta_j>2$ for some $1\leq j\leq t$, then there exists a prime $p$ such that $p\nmid N$ and $\left( \frac{-q_j}{p}\right)=1$, and thus $J_{1,p}(N)\supset J_{1,p}(q_j^3)$ is non-trivial by Lemma \ref{lem:J_1,q(p^3)}. If $64\mid N$, then we choose a prime $p$ with $p\nmid N$ and $p=1\, \bmod 4$, and therefore $J_{1,p}(N)\supset J_{1,p}(64)$ is non-trivial by Lemma \ref{lem:J_1P64}.
	
We now assume that $0\leq n\leq 5$ and $\beta_j\leq 2$ for $1\leq j\leq t$. In order to prove $J_{1,m}(N)=0$ for any $m$ with $\mathrm{gcd}(m,N)=1$, it suffices to consider the particular case where $n=0$ and $\beta_j=0$ for any $1\leq j\leq t$ by Proposition \ref{prop:dimension_induction}. Then the $2$-part of $J_{1,m}(N)$ is $(\CC D_{m_2}(a_2)\otimes \CC D_{m_2'}(a_2'))^{\SL_2(\ZZ)}$ with $a_2a_2'$ being a square modulo $4$. Indeed, we know from Proposition \ref{prop:main-iso} that $-a_2 (m/m_2)=1 \bmod 4m_2$ and $-a_2'(M/m_2')=1\bmod 4m_2'$.  It follows that $a_2a_2'(m/m_2)(M/m_2')=1 \mod 4$. By assumption, $(m/m_2)(M/m_2')$ is a product of powers of primes of type $4x+1$ and even powers of primes of type $4x+3$, so it is congruent to $1$ modulo $4$. This proves that $a_2a_2'=1\bmod 4$. By Lemma \ref{lem:irreducible_D2^N_LP^N}, this $2$-part vanishes and thus $J_{1,m}(N)=0$.
\end{proof}

We deduce a necessary and sufficient condition for a positive integer $N$ satisfying $J_{1,2}(N)\neq 0$.

\begin{proposition}\label{prop:index-2}
Suppose that $N=2^np_1^{\alpha_1}...p_s^{\alpha_s}\cdot q_1^{\beta_1}...q_t^{\beta_t}$, where $p_i$, $q_j$ are distinct odd primes, $\alpha_i\leq 2$ for any $1\leq i\leq s$, and $\beta_j>2$ for any $1\leq j\leq t$.  Then $J_{1,2}(N)$ is non-trivial if and only if there exists a non-empty subset $A \subset \{ 1,..,t\}$ such that $\prod_{j \in A}q_j=7 \, \bmod \,8$, that is, either there exists $q_j=7\bmod \,8$, or there exist distinct $q_{j_1}$ and $q_{j_2}$ such that $q_{j_1}=3\bmod \,8$ and $q_{j_2}=5\bmod \,8$.
Moreover, $\dim J_{1,2}(N)=1$ if and only if  $N=2^n kq^{3+r}$ or $N=2^n kq_1^{3+r_1}q_2^{3+r_2}$, where $n\geq 0$ is an integer, $k$ is odd and squarefree, $r,r_1,r_2\in\{0,1\}$, and $q,q_1,q_2$ are odd primes such that $q=7\,\bmod 8$, $q_1=3\,\bmod 8$, $q_2=5\, \bmod 8$.
\end{proposition}

\begin{proof}
Again, we use the decomposition in Proposition \ref{prop:main-iso}. 
By \eqref{eq:CL_P1}, i.e. $(\CC L_p(a_p))^{\Gamma_0(p^m)}$=0 for any integer $m>0$, the non-trivial $p_i$-part of $J_{1,2}(N)$ is $(\CC L^+_{m'_{p_i}}(a_{p_i}'))^{\Gamma_0(p_i^{\alpha_i})}$ with $m_{p_i}\in \{ 1,p_i^2\}$.
The non-trivial $2$-part of $J_{1,2}(N)$ can only be 
$$
\big(\CC D_{2}^{-}(a_2)\otimes \CC D_{m_2'}^-(a_2') \big)^{\Gamma_0(2^{n})},
$$
where $2|m_2'$, $-a_2=1\bmod 8$ and $-a_2'\prod m_{p_i}\cdot \prod m_{q_j}=1 \bmod 4m_2'$. By Lemma \ref{lem:CD2CD2N}, the $2$-part is non-trivial only if 
\begin{equation}\label{eq:constrain}
a_2a_2'=\prod m_{p_i}\cdot \prod m_{q_j} = \prod m_{q_j}=7\, \bmod \,8.    
\end{equation}
Thus $J_{1,2}(N)\neq0$ only if there exists a subset $A$ of $\{ 1,..,t\}$ such that $\prod_{j \in A}q_j=7 \, \bmod \,8$. Note that the $2$-part is of dimension one if it is non-trivial. 

If there exists $q_j$ with $q_j=7\bmod 8$, then the space 
$$ 
\left(\CC D_2^-(-1)\otimes\CC D_2^-(1)\right)^{\SL_2(\ZZ)} \bigotimes \left(   \CC L_{q_j^3}^-(a'_{q_j}) \right)^{\Gamma_0(q_j^3)} 
$$
lies in $J_{1,2}(q_j^3)$ and is non-trivial by \eqref{eq:CL_P3}. It follows that $J_{1,2}(N)\neq 0$ in this case. 

If there exists $q_{j_1}$ and $q_{j_2}$ such that $q_{j_1}=3\bmod \,8$ and $q_{j_2}=5\bmod \,8$, then the space
$$ 
\left(\CC D_2^-(-1)\otimes\CC D_{2}^-(1)\right)^{\SL_2(\ZZ)} \bigotimes \left(   \CC L_{q_{j_1}^3}^-(a'_{q_{j_1}}) \right)^{\Gamma_0(q_{j_1}^3)}  \bigotimes \left(   \CC L_{q_{j_2}^3}^+(a'_{q_{j_2}}) \right)^{\Gamma_0(q_{j_2}^3)}
$$
lies in $J_{1,2}(q_{j_1}^3q_{j_2}^3)$ and is non-trivial by \eqref{eq:CL_P3} again. Therefore, $J_{1,2}(N)\neq 0$ in this case. 
	
Let $q,q_1,q_2$ be primes with $q=7\,\bmod 8$, $q_1=3\,\bmod 8$, $q_2=5\, \bmod 8$. From \eqref{eq:constrain}, Lemma \ref{lem:CLP^N} and Lemma \ref{lem:CD2CD2N}, we see that the dimension of $J_{1,2}(2^nq^{3+r})$ and $J_{1,2}(2^nq_1^{3+r_1}q_2^{3+r_2})$ is one if and only if $r,r_1,r_2\in\{0,1\}$. The last result then follows from Proposition \ref{prop:dimension_induction}. 
\end{proof}

The above result actually indicates that $J_{1,2}(N)=0$ for $N<7^3$ and $\dim J_{1,2}(7^3)=1$. The generator can be constructed by the following general result. 

\begin{proposition}
Let $p$ be a prime with $p = 7 \, \bmod \, 8$. Then $J_{1,2}(p^3)$ is generated by
$$ 
\vartheta(\tau,2z)\cdot \sum_{l=1}^{p-1} \left(\frac{l}{p}\right)\theta_{2p,l+pl-p}(p^2\tau,0).
$$
\end{proposition}

\begin{proof}
First, by the above proof of Proposition \ref{prop:index-2},
$$ 
J_{1,2}(p^3) \cong \left(\CC D_2^-(-1)\otimes\CC D_{2}^-(1)\right)^{\SL_2(\ZZ)} \bigotimes \left( \CC L_{p^3}^-(a) \right)^{\Gamma_0(p^3)}.  
$$
Second, by Lemma \ref{lem:CD2CD2N} and Lemma \ref{lem:CLP^N} we have 
$$
\big(\CC D_2^-(-1)\otimes\CC D_{2}^-(1)\big)^{\SL_2(\ZZ)}=\CC \big((\mathbf{e}^1-\mathbf{e}^3)\otimes (\mathbf{e}^1-\mathbf{e}^3)\big), \quad  \big(\CC L_{p^3}^-(a)\big)^{\Gamma_0(p^3)}=\CC \left(\sum_{l=1}^{p-1}\left(\frac{l}{p}\right)\mathbf{e}^{lp^2}\right).
$$ 
The desired result then follows from  isomorphism \eqref{eq:iso-local-global}, isomorphism \eqref{eq:identification} and the basic facts that $\theta^-_{2,1}(\tau,\frac{1}{2}z)=-\vartheta(\tau,z)$ and $\theta^{\pm}_{m,r}(h\tau,hz)=\theta^{\pm}_{hm,hr}(\tau,z)$.  
\end{proof}

In what follows we consider Jacobi forms of weight one and odd prime index.

\begin{proposition}\label{prop:vanishcondition_J1_P(N)}
Let $N=2^ap^bq_1^{\alpha_1}...q_s^{\alpha_s}$, where $p$ and $q_i$ are distinct odd primes. Then $J_{1, p}(N)$ is non-trivial if and only if $N$ satisfies one of the following conditions:
\begin{enumerate}
\item  there exists a prime $q_i=3 \bmod 4$ with $\alpha_i\geq 3$, $\left(\frac{-q_i}{p}\right)=1$;
\item  there exist primes $q_i=3 \bmod 4$ and $q_j=1 \bmod 4$ with $\alpha_i\geq 3$, $\alpha_j\geq 3$, $\left(\frac{-q_i}{p}\right)=-1$ and  $\left(\frac{q j}{p}\right)=-1$;
\item  $a \geq 6$ and $\left(\frac{-1}{p}\right)=1$;
\item  $a \geq 6$, $\left(\frac{-1}{p}\right)=-1$, and there exists a prime $q_j=1 \bmod 4$ with $\alpha_j\geq 3$ and  $\left(\frac{q_j}{p}\right)=-1$;
\item  $a \geq 9$ and $\left(\frac{-2}{p}\right)=1$.
\end{enumerate}
\end{proposition}

\begin{proof}
By Proposition \ref{prop:dimension_induction}, we can assume that $\alpha_i\geq 3$ for $1\leq i\leq s$. We set
\begin{align*}
V_1= &\big(\CC D_1^+(a_2)\otimes \CC D_{2^{k_2}}^+(a_2')\big)^{\Gamma_0(2^a)}\otimes \big(\CC L_p^-(a_p)\otimes \CC L_{p^{k_p}}^-(a_p')\big)^{\Gamma_0(p^b)}\\
&\otimes \bigoplus_{\epsilon_*=-} \bigotimes_{q_i}\big(\CC L_{q_i^{\beta_i}}^{\epsilon_{i}}(a_{q_i}')\big)^{\Gamma_0(q_i^{\alpha_i})}, \end{align*}
where the number of primes of type $q_i=3\, \bmod 4$ with $\beta_i $ being odd is odd; $k_2=0$ if $a\leq 2$ and $k_2\in \{a-3,a-2\}$ if $a>2$; $k_p=1$ if $b=0$ and $k_p\in \{b-1,b\}$ if $b\geq 1$; $\beta_i\in \{\alpha_i-1,\alpha_i\}$; and the  other notation is the same as in Proposition \ref{prop:main-iso}. If $\beta_i$ is odd and $q_i=3 \bmod 4$, then $\epsilon_i=-$ by Lemma \ref{lem:CLP^N}; if $\beta_i$ is even or $q_i=1 \bmod 4$, then $\epsilon_i=+$ by Lemma \ref{lem:CLP^N} again; moreover, the $q_i$-part is non-trivial.  By Lemma \ref{lem:CLPCLP^2N}, the $p$-part of $V_1$ is non-trivial if and only if $k_p$ is odd and $\big(\frac{-a_pa_p'}{p}\big)=1$. Note that $-2^2a_p=1\bmod p$ and $-a_p'2^{k_2+2}\prod_{i=1}^s q_i^{\beta_i}=1 \bmod p^{k_p}$. Therefore, 
$$
a_pa_p'2^{k_2+4}\prod_{i=1}^s q_i^{\beta_i} = 1 \mod p
$$
and thus
\begin{equation}\label{eq:a_pa_p'}
1=\left( \frac{-a_pa_p'}{p} \right) \quad \Longleftrightarrow \quad \left(\frac{-2^\delta \prod_{\beta_i \; \mathrm{odd}}q_i}{p} \right)=1,    
\end{equation}
where $\delta=0$ if $k_2$ is even, and $\delta=1$ if $k_2$ is odd. If $V_1\neq 0$, then $-a_2a_2'$ is a square modulo $4$. Assume that the $p$-part of $V_1$ is non-trivial. From Lemma \ref{lem:CD1_-a1a2_square}, we conclude that the $2$-part of $V_1$ never vanishes when $k_2$ is even, and is non-trivial if and only if $a\geq 9$ when $k_2$ is odd.

We further set
\begin{align*}
V_2 =& \big(\CC D_1^+(a_2)\otimes \CC D_{2^{k_2}}^-(a_2')\big)^{\Gamma_0(2^a)}\otimes \big(\CC L_p^-(a_p)\otimes \CC L_{p^{k_p}}^-(a_p')\big)^{\Gamma_0(p^b)} \\
&\otimes  \bigoplus_{\epsilon_*=+} \bigotimes_{q_i}\big(\CC L_{q_i^{\beta_i}}^{\epsilon_{i}}(a_{q_i}')\big)^{\Gamma_0(q_i^{\alpha_i})},
\end{align*}
where the number of primes of type $q_i=3\, \bmod 4$ with $\beta_i $ being odd is even and the other is the same as above. Similarly, the $p$-part of $V_2$ is non-trivial if and only if $k_p$ is odd and $\big(\frac{-a_pa_p'}{p}\big)=1$. Note that \eqref{eq:a_pa_p'} also holds in this case. In this case, if $V_2\neq 0$, then $a_2a_2'$ is a square modulo $4$. Assume that the $p$-part of $V_2$ is non-trivial. By Lemma \ref{lem:CD1_a1a2_square}, the $2$-part of $V_2$ is non-trivial if and only if $a\geq 6$ when $k_2$ is even, and if and only if $a\geq 9$ when $k_2$ is odd.  

By Proposition \ref{prop:main-iso} and Lemma \ref{lem:CLPCLP^2N}, $J_{1,p}(N)$ is the direct sum of components of types $V_1$ and $V_2$. By the discussion above and especially \eqref{eq:a_pa_p'}, it is not difficult to verify that $V_1\neq 0$ or $V_2\neq 0$ is equivalent to one of the five conditions. Indeed, $V_1\neq 0$ is related to Condition (1) or (2) when $k_2$ is even, and is related to Condition (5) if $k_2$ is odd; $V_2\neq 0$ is related to Condition (3) or (4) if $k_2$ is even, and is related to Condition (5) if $k_2$ is odd. We then prove the desired result.
\end{proof}

By applying Proposition \ref{prop:vanishcondition_J1_P(N)} to $p=3$, we obtain the following criterion. 

\begin{corollary}
The space $J_{1,3}(N)\neq 0$ if and only if $N$ satisfies one of the following conditions:
\begin{enumerate}
\item  There exists a prime $q=11 \bmod 12$ such that $q^3\mid N$;
\item  There exist prime $q_1=7 \bmod 12$ and $q_2=5\bmod 12$ such that $q_1^3q_2^3\mid N$;
\item  $64\mid N$ and there exists a prime $q=5 \bmod 12$ such that $q^3\mid N$; 
\item  $512\mid N$. 
\end{enumerate}
\end{corollary}

This result yields that $J_{1,3}(N)=0$ if $N< 2^9$ and $J_{1,3}(2^9)\neq 0$. By \eqref{eq:dim-index-p} below, $\dim J_{1,3}(2^9)=1$. 

\begin{remark}
Suppose that $p^3\nmid mN$ for any prime $p=3\bmod 4$. For any non-trivial $p$-part of $J_{1,m}(N)$ with $p=3\bmod 4$, $m_pm_p'$ is a square. Thus the non-trivial $2$-part of $J_{1,m}(N)$ satisfies that $a_2a_2'=1\bmod 4$. By Lemma \ref{lem:evenresult}, if $32\nmid N$ then $J_{1,m}(N)=0$. By \cite[Lemma 3.10]{CDH18}, if $64\nmid N$ and either $m$ is odd or $m=4\bmod 8$ then $J_{1,m}(N)=0$. This recovers and extends \cite[Lemma 3.11]{CDH18}. \end{remark}

\section{Dimension formula}\label{sec:dimension}
In this section we establish several dimension formulas for Jacobi forms of weight one on $\Gamma_0(N)$.  

For convenience we introduce some symbols. Let $p_1,...,p_s$  be primes and $m$ be an integer with $\mathrm{gcd}(m, \prod_{i=1}^s p_i)=1$. We set
\begin{equation}
f\big(m; \{p_1,...,p_s\}\big)=\left\{\begin{array}{ll}
	1, & \quad \text{if}\; \big(\frac{m}{p_i}\big)=1, \quad \text {for all} \; 1 \leq i \leq s, \\
	0, & \quad \text{if}\; \big(\frac{m}{p_j}\big)=-1, \quad \text {for some} \; 1\leq j \leq s,
 \end{array}\right.
\end{equation}
and define
$$
f(m ; \emptyset)=1. 
$$
We denote the exponent of a prime $p$ in $x$ by $n_p^x$, that is,  $p^{n_p^x}\mid x$ and $p^{n_p^x+1}\nmid x$.

Our first theorem gives a precise formula of $\dim J_{1,m}(N)$ whenever $\mathrm{gcd}(m,N)=1$. 

\begin{theorem}\label{th:m_N_COPRIME}
Let $m$ and $N$ be positive odd integers with $\mathrm{gcd}(m,N)=1$. We set
\begin{align*}
\mathrm{P}_1^m&:=\left\{\text{prime}\; p: \; n_p^m \; \text{is odd}\right\}, \\
\mathrm{P}_0^m&:=\left\{\text{prime}\; p: \; n_p^m \; \text{is positive and even}\right\}, \\
\mathrm{P}_2^N&:=\left\{\text{odd prime}\; p: \; n_p^N=2\right\}, \\
\mathrm{P}_{>2}^{N,1}&:=\left\{\text{odd prime}\; p: \; n_p^N \geq 3,\,p=1\bmod 4\right\}, \\
\mathrm{P}_{>2}^{N,3}&:=\left\{\text{odd prime}\; p: \; n_p^N \geq 3,\,p=3\bmod 4\right\}.  
\end{align*} 
For $j\in\{0,1\}$ and any non-negative integer $k$, 
\begin{equation}
\dim J_{1,2^jm}(2^kN) = 2^{|\mathrm{P}_2^N|}\left( I_1+I_2+I_3+I_4\right).
\end{equation}
For any non-negative integer $j$, 
\begin{equation}
\dim J_{1,2^jm}(N) = 2^{|\mathrm{P}_2^N|}\left( I_1+I_4\right).
\end{equation}

The terms $I_1, I_2, I_3, I_4$ are sums over prime subsets, defined below.

\begin{equation}\label{eq:I_i}
\begin{split}
I_i=\sum_{\mathrm{S}_{>2,i}^{N,3}}\sum_{\mathrm{S}_{>2,i}^{N,1}} \sum_{\mathrm{S}_{0,i}^m} & \prod_{p \in \mathrm{S}_{>2,i}^{N,3}\cup \mathrm{S}_{>2,i}^{N,1}} \left[\frac{n_p^N-1}{2}\right] \prod_{p \in \mathrm{P}_{>2}^{N,3}\cup \mathrm{P}_{>2}^{N,1} \backslash \mathrm{S}_{>2,i}^{N,3} \cup \mathrm{S}_{>2,i}^{N,1}} \left[\frac{n_p^N+2}{2}\right] \\  
&\quad \prod_{p \in \mathrm{P}_1^m} \frac{n_p^m+1}{2}
\prod_{p\in \mathrm{S}_{0,i}^m} \frac{n_p^m}{2} \prod_{p \in \mathrm{P}_{0,i}^m \backslash \mathrm{S}_{0,i}^m} \frac{n_p^m+2}{2} \cdot \mathrm{K}_i^j.    
\end{split}    
\end{equation}
The above symbols are explained as follows:
\begin{enumerate}
\item $\mathrm{S}_{>2,i}^{N,3}$ runs through all the subsets of $\mathrm{P}_{>2}^{N,3}$ such that $|\mathrm{S}_{>2,i}^{N,3}|$ is odd for $i=1,4$ and $|\mathrm{S}_{>2,2}^{N,3}|$ is even. (There is no condition on $|S_{>2,3}^{N,3}|$.) In particular, both $\mathrm{S}_{>2,1}^{N,3}$ and $\mathrm{S}_{>2,4}^{N,3}$ are non-empty, while $\mathrm{S}_{>2,2}^{N,3}$ and $\mathrm{S}_{>2,3}^{N,3}$ could be empty.
\item $\mathrm{S}_{>2,i}^{N,1}$ runs through all the subsets of $\mathrm{P}_{>2}^{N,1}$ such that $f(\mathrm{P}_i; \mathrm{P}_1^m)=1$,  where 
$$
\mathrm{P}_i=-\prod_{p\in \mathrm{S}_{>2,i}^{N,3}}p\prod_{q\in \mathrm{S}_{>2,i}^{N,1}}q
$$
for $i=1,2,4$, and
$$
\mathrm{P}_3=-2\cdot \prod_{p\in \mathrm{S}_{>2,3}^{N,3}}p \prod_{q\in \mathrm{S}_{>2,3}^{N,1}}q.
$$
If the set $\mathrm{S}_{>2,2}^{N,3}$, $\mathrm{S}_{>2,3}^{N,3}$ or $\mathrm{S}_{>2,i}^{N,1}$ for $1\leq i\leq 4$ is empty, then we define the corresponding product  above to be $1$. 
\item $\mathrm{P}^m_{0,i}$ denotes the unique maximal subset of $\mathrm{P}_0^m$ such that  $f(\mathrm{P}_i; \mathrm{P}^m_{0,i})=1$.
\item $\mathrm{S}_{0,i}^m$ runs through all subsets of $\mathrm{P}_{0,i}^m$.
\item Assume that $|\mathrm{P}_1^m|>0$. Then 
\begin{align*}
&\mathrm{K}_1^0=2^{|\mathrm{P}_1^m|-1}\Big(1+\Big[\frac{k}{2}\Big]\Big),& &\mathrm{K}_2^0=2^{|\mathrm{P}_1^m|-1}\mathrm{max}\Big\{0,\Big[\frac{k-4}{2}\Big]\Big\},&\\
&\mathrm{K}_3^0=2^{|\mathrm{P}_1^m|-1}\mathrm{max}\Big\{0,\Big[\frac{k-7}{2}\Big]\Big\},& &\mathrm{K}_4^0=0,  & \\
&\mathrm{K}_1^1=2^{|\mathrm{P}_1^m|-1}\Big(\Big[\frac{k+1}{2}\Big]+\frac{(\frac{\mathrm{P}_1}{2})+1}{2}\Big),&  
&\mathrm{K}_2^1=2^{|\mathrm{P}_1^m|-1}\mathrm{max}\Big\{ 0, \Big[\frac{k-3}{2}\Big]\Big\},& \\ 
&\mathrm{K}_3^1=2^{|\mathrm{P}_1^m|-1}\mathrm{max}\Big\{0,\Big[\frac{k-6}{2} \Big]\Big\},& 
&\mathrm{K}_4^1=2^{|\mathrm{P}_1^m|-1} \frac{(\frac{\mathrm{P}_4}{2})+1}{2}.&
\end{align*}
For even $j\geq 2$, 
$$
\mathrm{K}_1^j=2^{|\mathrm{P}_1^m|-1}\Big(1+ j\frac{(\frac{\mathrm{P}_1}{2})+1}{4}\Big) \quad \text{and}  \quad   \mathrm{K}_4^j=2^{|\mathrm{P}_1^m|-1} j\frac{(\frac{\mathrm{P}_4}{2})+1}{4}.
$$
For odd $j\geq 3$,
$$ 	
\mathrm{K}_1^j=2^{|\mathrm{P}_1^m|-1}(1+j)\frac{(\frac{\mathrm{P}_1}{2})+1}{4} \quad \text{and} \quad  \mathrm{K}_4^j=2^{|\mathrm{P}_1^m|-1} (1+j)\frac{(\frac{\mathrm{P}_4}{2})+1}{4}.
$$
\item Assume that $|\mathrm{P}_1^m|=0$. Then we double $\mathrm{K}_i^j$ and further require $|\mathrm{S}_{0,i}^m|$ to be odd for $i=1,2,3$ and $|\mathrm{S}_{0,4}^m|$ to be even. 
\item For each of the five products in \eqref{eq:I_i}, we define it to be $1$ if the corresponding set is empty, and define it to be $0$ if the corresponding set cannot be selected to satisfy the above conditions.  
\item $[x]$ denotes the integer part of $x$ and $\big( \frac{n}{m}\big)$ denotes the Kronecker symbol as before. 
\end{enumerate}
\end{theorem}

\begin{proof}
We first formulate all the $p$-parts of $J_{1,2^jm}(2^kN)$ by Lemma \ref{lem:irreducible_D2^N_LP^N} and Lemma \ref{lem:CLP^N}.
\begin{enumerate}
\item  Let $p\mid N$. The only possible non-trivial $p$-part of $J_{1,2^jm}(2^kN)$ is $(\CC L_{p^{2k+1}}^{\epsilon_p'}(a'_p))^{\Gamma_0(n_p^N)}$ or $(\CC L_{p^{2k}}^+(a'_p))^{\Gamma_0(n_p^N)}$, whose dimension is $\mathrm{max}\big\{ 0,\big[\frac{n_p^N-1}{2}\big]\big\}$ or $ \big[\frac{n_p^N+2}{2}\big]$ respectively, where $k$ is determined by the relation that $2k, 2k+1\in \{n_p^N, n_p^N-1\}$, $\epsilon_p'=-$ for $p=3\,\bmod 4$ and $\epsilon_p'=+$ for $p=1\,\bmod 4$.

\item Let $p\mid m$. If $-a_pa_p'$ is not a square modulo $p$, then the only possible non-trivial $p$-part of $J_{1,2^j m}(2^k N)$ is $(\CC L_{p^{n_p^m}}^{+}(a_p)\otimes \CC L_{p^{n_p^m}}^{+}(a_p'))^{\SL_2(\ZZ)}$ with $n_p^m$ even, whose dimension is $1$. If $-a_pa_p'$ is a square modulo $p$, then the only possible non-trivial $p$-part of $J_{1,2^j m}(2^k N)$ is $(\CC L_{p^{n_p^m}}^{\epsilon}(a_p)\otimes \CC L_{p^{n_p^m}}^{\epsilon}(a_p'))^{\SL_2(\ZZ)}$ with $n_p^m$ odd and $\epsilon\in\{\pm \}$, or $(\CC L_{p^{n_p^m}}^{+}(a_p)\otimes \CC L_{p^{n_p^m}}^{+}(a_p'))^{\SL_2(\ZZ)}$ with $n_p^m$ even, or $(\CC L_{p^{n_p^m}}^{-}(a_p)\otimes \CC L_{p^{n_p^m}}^{-}(a_p'))^{\SL_2(\ZZ)}$ with $n_p^m$ even, whose dimension is respectively $(n_p^m+1)/2$, or  $(n_p^m+2)/2$, or $n_p^m/2$. 

\item The only possible non-trivial $2$-parts of $J_{1,m}(2^kN)$ and $J_{1,2m}(2^kN)$ are formulated in Table \ref{tab:m} and Table \ref{tab:2m} by Lemmas \ref{lem:CD1_a1a2_square}--\ref{lem:CD2CD2N}, respectively.
    
\begin{table}[ht]
\caption{The only possible non-trivial $2$-parts of $J_{1,m}(2^kN)$}\label{tab:m}
\renewcommand\arraystretch{1.5}
\noindent\[
\begin{array}{|c|c|c|}
\hline \text{2-part} & \text{Requirement} & \text{Dimension} \\
\hline (\CC D_1^+(a_2)\otimes \CC D_1^+(a_2'))^{\SL_2(\ZZ)} & a_2 a_2'=3 \bmod 4 & 1 \\
\hline (\CC D_1^+(a_2)\otimes \CC D_1^+(a_2'))^{\Gamma_0(2)} & a_2 a_2'=3 \bmod 4 & 1 \\
\hline (\CC D_1^+(a_2)\otimes \CC D_{2^{2n}}^+(a_2'))^{\Gamma_0(2^{2n+2})} & a_2 a_2'=3 \bmod 4 & 2+n \\
\hline (\CC D_1^+(a_2)\otimes \CC D_{2^{2n}}^-(a_2'))^{\Gamma_0(2^{2n+2})} & a_2 a_2'=1 \bmod 4 & \max \{0,n-1\} \\
\hline (\CC D_1^+(a_2)\otimes \CC D_{2^{2n+1}}^{+}(a_2'))^{\Gamma_0(2^{2n+3})} & a_2 a_2'=3 \bmod 4 & \max \{0,n-2\}\\
(\CC D_1^+(a_2)\otimes \CC D_{2^{2n+1}}^{-}(a_2'))^{\Gamma_0(2^{2n+3})} & a_2 a_2'=1 \bmod 4 & \max \{0,n-2\}\\\hline
\end{array} 
\]
\end{table}

\begin{table}[ht]
\caption{The only possible non-trivial $2$-parts of $J_{1,2m}(2^kN)$}\label{tab:2m}
\renewcommand\arraystretch{1.5}
\noindent\[
\begin{array}{|c|c|c|}
\hline \text{2-part} & \text{Requirement} & \text{Dimension} \\
\hline (\CC D_2^-(a_2)\otimes \CC D_{2^{2n+1}}^-(a_2'))^{\Gamma_0(2^{2n+3})} & a_2 a_2'=7 \bmod 8 & 1 \\
\hline (\CC D_2^+(a_2)\otimes \CC D_2^+(a_2'))^{\SL_2(\ZZ)} & a_2 a_2'=7 \bmod 8 & 1 \\
\hline (\CC D_2^+(a_2)\otimes \CC D_2^+(a_2'))^{\Gamma_0(2)} & a_2 a_2'=7 \bmod 8 & 2 \\
\hline (\CC D_2^+(a_2)\otimes \CC D_2^+(a_2'))^{\Gamma_0(2)} & a_2 a_2'=3 \bmod 8 & 1 \\
\hline (\CC D_2^+(a_2)\otimes \CC D_2^+(a_2'))^{\Gamma_0(4)} & a_2 a_2'=7 \bmod 8 & 2 \\
\hline (\CC D_2^+(a_2)\otimes \CC D_2^+(a_2'))^{\Gamma_0(4)} & a_2 a_2'=3 \bmod 8 & 1 \\
\hline (\CC D_2^+(a_2)\otimes \CC D_{2^{2n+1}}^+(a_2'))^{\Gamma_0(2^{2n+3})} & a_2 a_2'=7 \bmod 8 &  3+n \\
\hline (\CC D_2^+(a_2)\otimes \CC D_{2^{2n+1}}^+(a_2'))^{\Gamma_0(2^{2n+3})} & a_2 a_2'=3 \bmod 8 &  2+n \\
\hline (\CC D_2^+(a_2)\otimes \CC D_{2^{2n+1}}^-(a_2'))^{\Gamma_0(2^{2n+3})} & a_2 a_2'=1 \bmod 4 &  n \\
\hline (\CC D_2^+(a_2)\otimes \CC D_{2^{2n}}^{-}(a_2'))^{\Gamma_0(2^{2n+2})} & a_2 a_2'=1 \bmod 4 & \max \{0,n-2\}\\
(\CC D_2^+(a_2)\otimes \CC D_{2^{2n}}^{+}(a_2'))^{\Gamma_0(2^{2n+2})} & a_2 a_2'=3 \bmod 4 & \max \{0,n-2\}\\
\hline
\end{array} 
\]
\end{table}
\end{enumerate}

We now suppose that $V$ is a non-trivial constitute of $J_{1,m}(2^kN)$.  We view $\mathrm{S}_{0,i}^m$ as the subset of primes $p$ in $\mathrm{P}_0^m$ such that $\epsilon_p=-$ in (2) above, $\mathrm{S}_{>2,i}^{N,1}$ as the subset of primes $p$ in $\mathrm{P}_{>2}^{N,1}$ whose part is of form $L_{p^{2k+1}}^{+}(a_p')$ in (1) above, and $\mathrm{S}_{>2,i}^{N,3}$ as the subset of primes $p$ in $\mathrm{P}_{>2}^{N,3}$ whose part is of form $L_{p^{2k+1}}^{-}(a_p')$ in (1) above. We then discuss by cases.
\begin{enumerate}
\item[(a)] Assume that the $2$-part of $V$ is $(\CC D_1^+(a_2)\otimes \CC D_{2^{2n}}^+(a_2'))^{\Gamma_0(2^{2n+2})}$ with $n=[k/2-1]$, whose dimension is $[k/2]+1$ by Table \ref{tab:m}. In this case, we set $V_1=V$ and $i=1$. By Proposition \ref{prop:main-iso}, the number of “$-$" among all signs $\epsilon_p$ of $\CC L_{p^{n^{m}_p}}^{\epsilon_p}(a_p)$ for $p\mid m$ is odd. It follows that the number of “$-$" among all signs $\epsilon'_p$ of $\CC L_{p^{k_p}}^{\epsilon_p'}(a'_p)$ for $p\mid N$ is also odd. Therefore, $|\mathrm{S}_{>2,1}^{N,3}|$ has to be odd.  For any $p\mid m$, similar to \eqref{eq:a_pa_p'}, we confirm that $\big( \frac{-a_pa_p'}{p}\big)=1$ if and only if $\left( \frac{\mathrm{P}_1}{p} \right)=1$, where $\mathrm{P}_1$ is defined in (2) of the theorem. For the $2$-part, $a_2a_2'=3 \mod 4$ is guaranteed by the condition that $|\mathrm{S}_{>2,1}^{N,3}|$ is odd. One is free to choose the sign of the $p$-part of $V_1$ for any $p\in \mathrm{P}^m_1$, and the only constrain is that the number of “$-$" among all signs of the $p$-parts of $V_1$ for $p\in \mathrm{P}^m_1$ is odd if $|\mathrm{S}^m_{0,1}|$ is even, and is even if $|\mathrm{S}^m_{0,1}|$ is odd. This contributes to the additional factor $2^{|\mathrm{P}_1^m|-1}$ in $\mathrm{K}_1^0$ if $|\mathrm{P}_1^m|>0$. Besides, we have to require that $|\mathrm{S}^m_{0,1}|$ is odd if $|\mathrm{P}_1^m|=0$. Obviously, the factor $2^{|\mathrm{P}_2^N|}$ in the dimension formula follows from Proposition \ref{prop:dimension_induction}. We then prove that $\dim V_1=2^{|\mathrm{P}_2^N|} I_1$. 
\item[(b)] Assume that the $2$-part of $V$ is $(\CC D_1^+(a_2)\otimes \CC D_{2^{2n}}^-(a_2'))^{\Gamma_0(2^{2n+2})}$ with $n=[k/2-1]$, whose dimension is $\mathrm{max}\{ 0, [(k-4)/2]\}$ by Table \ref{tab:m}. In this case, we set $V_2=V$ and $i=2$. Again, the number of “$-$" among all signs $\epsilon_p$ of $\CC L_{p^{n^{m}_p}}^{\epsilon_p}(a_p)$ for $p\mid m$ is odd. It follows that the number of “$-$" among all signs $\epsilon'_p$ of $\CC L_{p^{k_p}}^{\epsilon_p'}(a'_p)$ for $p\mid N$ is even. Therefore, $|\mathrm{S}_{>2,2}^{N,3}|$ has to be even.  Similarly, for any $p\mid m$, $\big( \frac{-a_pa_p'}{p}\big)=1$ if and only if $\left( \frac{\mathrm{P}_2}{p} \right)=1$, where $\mathrm{P}_2$ is defined in (2) of the theorem. For the $2$-part, $a_2a_2'=1 \mod 4$ is guaranteed by the condition that $|\mathrm{S}_{>2,2}^{N,3}|$ is even. As in (a), we then prove that $\dim V_2=2^{|\mathrm{P}_2^N|} I_2$. 
\item[(c)] Assume that the $2$-part of $V$ is $(\CC D_1^+(a_2)\otimes \CC D_{2^{2n+1}}^{\epsilon_2'}(a_2'))^{\Gamma_0(2^{2n+2})}$ with $\epsilon_2'\in\{+,-\}$ and $n=[(k-3)/2]$, whose dimension is $\mathrm{max}\{ 0, [(k-7)/2]\}$ by Table \ref{tab:m}. In this case, we set $V_3=V$ and $i=3$. There is no restriction for $|\mathrm{S}_{>2,3}^{N,3}|$ since $\epsilon\in\{ \pm \}$. For any $p\mid m$, $\big( \frac{-a_pa_p'}{p}\big)=1$ if and only if $\left( \frac{\mathrm{P}_3}{p} \right)=1$, where $\mathrm{P}_3$ is defined in (2) of the theorem. For the $2$-part, $a_2a_2'=1 \mod 4$ if and only if $|\mathrm{S}_{>2,3}^{N,3}|$ is even, and  $a_2a_2'=3 \mod 4$ if and only if $|\mathrm{S}_{>2,3}^{N,3}|$ is odd. As in (a), we prove that $\dim V_3=2^{|\mathrm{P}_2^N|} I_3$. 
\end{enumerate}
Combining (a), (b) and (c) above, we derive the dimension formula of $J_{1,m}(2^k N)$. 

\vspace{3mm}

We now establish the dimension formula for $J_{1,2m}(2^kN)$. We suppose that $V$ is a non-trivial constitute of $J_{1,2m}(2^kN)$ and use the same notation as above. 
\begin{enumerate}
\item[(a)] Assume that the $2$-part of $V$ is $(\CC D_2^+(a_2)\otimes \CC D_{2^{2n+1}}^+(a_2'))^{\Gamma_0(2^{2n+3})}$ with $n=[(k-3)/2]$. By Table \ref{tab:2m}, its dimension is $[(k+3)/2]$ if $a_2a_2'=7 \bmod 8$, and $[(k+1)/2]$ if $a_2a_2'=3 \bmod 8$. In this case, we set $V_1=V$ and $i=1$. The number of “$-$" among all signs $\epsilon_p$ of $\CC L_{p^{n^{m}_p}}^{\epsilon_p}(a_p)$ for $p\mid m$ is odd. It follows that the number of “$-$" among all signs $\epsilon'_p$ of $\CC L_{p^{k_p}}^{\epsilon_p'}(a'_p)$ for $p\mid N$ is also odd. Therefore, $|\mathrm{S}_{>2,1}^{N,3}|$ has to be odd.  For any $p\mid m$, $\big( \frac{-a_pa_p'}{p}\big)=1$ if and only if $\left( \frac{\mathrm{P}_1}{p} \right)=1$, where $\mathrm{P}_1$ is defined in (2) of the theorem. Since $|\mathrm{S}_{>2,1}^{N,3}|$ is odd, $\mathrm{P}_1=1 \bmod 4$. For the $2$-part, $a_2a_2'=7 \bmod 8$ if and only if $\big( \frac{\mathrm{P}_1}{2} \big)=1$, and $a_2a_2'=3 \bmod 8$ if and only if $\big(\frac{\mathrm{P}_1}{2} \big)=-1$. Therefore, the dimension of the $2$-part of $V_1$ can be written as $[(k+1)/2]+\big(\big(\frac{\mathrm{P}_1}{2} \big)+1\big)/2$, which contributes to $\mathrm{K}_1^1$. By a similar argument as in the previous proof,  it is easy to show that $\dim V_1=2^{|\mathrm{P}_2^N|} I_1$. 

\item[(b)] Assume that the $2$-part of $V$ is $(\CC D_2^+(a_2)\otimes \CC D_{2^{2n+1}}^+(a_2'))^{\Gamma_0(2^{2n+3})}$ with $n=[(k-3)/2]$, whose dimension is $\mathrm{max}\{0, [(k-3)/2]\}$. In this case, we set $V_2=V$ and $i=2$. Similar to (b) in the previous proof, we confirm that $\dim V_2=2^{|\mathrm{P}_2^N|} I_2$. 

\item[(c)] Assume that the $2$-part of $V$ is $(\CC D_2^+(a_2)\otimes \CC D_{2^{2n}}^{\pm}(a_2'))^{\Gamma_0(2^{2n+2})}$ with $n=[(k-2)/2]$, whose dimension is $\mathrm{max}\{0, [(k-6)/2]\}$. In this case, we set $V_3=V$ and $i=3$. Similar to (c) in the previous proof, we verify that $\dim V_3=2^{|\mathrm{P}_2^N|} I_3$.

\item[(d)] Assume that the $2$-part of $V$ is $(\CC D_2^-(a_2)\otimes \CC D_{2^{2n+1}}^{-}(a_2'))^{\Gamma_0(2^{2n+3})}$ with $n=[(k-3)/2]$, whose dimension is $1$ if $a_2a_2'=7\bmod 8$, and $0$ otherwise.  In this case, we set $V_4=V$ and $i=4$. The number of “$-$" among all signs $\epsilon_p$ of $\CC L_{p^{n^{m}_p}}^{\epsilon_p}(a_p)$ for $p\mid m$ is even. Thus we have to require $|\mathrm{S}_{0,4}^m|$ to be even if $|\mathrm{P}_{1}^m|=0$.  It also follows that the number of “$-$" among all signs $\epsilon'_p$ of $\CC L_{p^{k_p}}^{\epsilon_p'}(a'_p)$ for $p\mid N$ is odd. Therefore, $|\mathrm{S}_{>2,4}^{N,3}|$ has to be odd. For any $p\mid m$, $\big( \frac{-a_pa_p'}{p}\big)=1$ if and only if $\big( \frac{\mathrm{P}_4}{p} \big)=1$, where $\mathrm{P}_4$ is defined in (2) of the theorem. Since $|\mathrm{S}_{>2,4}^{N,3}|$ is odd, $\mathrm{P}_4=1 \bmod 4$. For the $2$-part, $a_2a_2'=7 \bmod 8$ if and only if $\left( \frac{\mathrm{P}_4}{2} \right)=1$. Therefore, the dimension of the $2$-part of $V_4$ can be written as $\big(\big(\frac{\mathrm{P}_4}{2} \big)+1\big)/2$, which contributes to $\mathrm{K}_4^1$.  By these facts, it is enough to deduce that $\dim V_4=2^{|\mathrm{P}_2^N|} I_4$. 
\end{enumerate}

Combining (a), (b), (c) and (d) above, we derive the dimension formula of $J_{1,2m}(2^k N)$. 

\vspace{3mm}

We finally establish the dimension formula for $J_{1,2^j m}(N)$ with $j\geq 2$. Again, we suppose that $V$ is a non-trivial constitute of $J_{1,2^j m}(N)$ and use the same notation as above.  The $2$-parts of $J_{1,2^j m}(N)$ are determined by Lemma \ref{lem:irreducible_D2^N_LP^N}. There are two cases. 
\begin{enumerate}
\item[(i)] Assume that the $2$-part of $V$ is $(\CC D_{2^j}^+(a_2)\otimes \CC D_{2^j}^+(a_2'))^{\SL_2(\ZZ)}$. By Lemma \ref{lem:irreducible_D2^N_LP^N}, its dimension is $[j/2]+1$ if $a_2a_2'=7 \bmod 8$, and $1$ if $a_2a_2'=3 \bmod 8$ and $j$ is even, and $0$ otherwise. In this case, we set $V_1=V$ and $i=1$. The number of “$-$" among all signs $\epsilon_p$ of $\CC L_{p^{n^{m}_p}}^{\epsilon_p}(a_p)$ for $p\mid m$ is odd. It follows that the number of “$-$" among all signs $\epsilon'_p$ of $\CC L_{p^{k_p}}^{\epsilon_p'}(a'_p)$ for $p\mid N$ is also odd. Therefore, $|\mathrm{S}_{>2,1}^{N,3}|$ has to be odd.  For any $p\mid m$, $\big( \frac{-a_pa_p'}{p}\big)=1$ if and only if $\left( \frac{\mathrm{P}_1}{p} \right)=1$, where $\mathrm{P}_1$ is defined in (2) of the theorem. Since $|\mathrm{S}_{>2,1}^{N,3}|$ is odd, $\mathrm{P}_1=1 \bmod 4$. For the $2$-part, $a_2a_2'=7 \bmod 8$ if and only if $\left( \frac{\mathrm{P}_1}{2} \right)=1$, and $a_2a_2'=3 \bmod 8$ if and only if $\left(\frac{\mathrm{P}_1}{2} \right)=-1$. Therefore, the dimension of the $2$-part of $V_1$ can be written as $1+j\big(\big( \frac{\mathrm{P}_1}{2}\big)+1\big)/4$ if $j$ is even, and $(1+j)\big(\big( \frac{\mathrm{P}_1}{2}\big)+1\big)/4$ if $j$ is odd, which contributes to $\mathrm{K}_1^j$. These facts are enough to verify that $\dim V_1=2^{|\mathrm{P}_2^N|} I_1$. 

\item[(ii)] Assume that the $2$-part of $V$ is $(\CC D_{2^j}^-(a_2)\otimes \CC D_{2^j}^-(a_2'))^{\SL_2(\ZZ)}$. By Lemma \ref{lem:irreducible_D2^N_LP^N}, its dimension is $[(j+1)/2]$ if $a_2a_2'=7 \mod 8$, and $0$ otherwise. In this case, we set $V_4=V$ and $i=4$. Similarly, $|\mathrm{S}_{>2,4}^{N,3}|$ is odd, which yields that $\mathrm{P}_4=1 \bmod 4$. For the $2$-part, $a_2a_2'=7 \bmod 8$ if and only if $\left( \frac{\mathrm{P}_4}{2} \right)=1$. Therefore, the dimension of the $2$-part of $V_4$ can be written as $j\big(\big( \frac{\mathrm{P}_4}{2}\big)+1\big)/4$ if $j$ is even, and $(1+j)\big(\big( \frac{\mathrm{P}_4}{2}\big)+1\big)/4$ if $j$ is odd, which contributes to $\mathrm{K}_4^j$. These facts are sufficient to deduce that $\dim V_4=2^{|\mathrm{P}_2^N|} I_4$. 
\end{enumerate}

We thus prove the dimension formula for $J_{1,2^j m}(N)$ by (i) and (ii) above. 
\end{proof}

As an application, we consider one particular case. Let $N$ be a positive odd integer and $k$ be a non-negative integer. For $J_{1,2}(2^k N)$, $I_1=I_2=I_3=0$. Thus $\dim J_{1,2}(2^k N)$ is given by
\begin{equation}
2^{|\mathrm{P}_2^N|}\sum_{\mathrm{S}_{>2,4}^{N,3}}\sum_{\mathrm{S}_{>2,4}^{N,1}}  \prod_{p \in \mathrm{S}_{>2,4}^{N,3}\cup \mathrm{S}_{>2,4}^{N,1}} \left[\frac{n_p^N-1}{2}\right] \prod_{p \in \big(\mathrm{P}_{>2}^{N,3}\cup \mathrm{P}_{>2}^{N,1}\big) \backslash \big(\mathrm{S}_{>2,4}^{N,3} \cup \mathrm{S}_{>2,4}^{N,1}\big)} \left[\frac{n_p^N+2}{2}\right] \cdot \frac{\big( \frac{\mathrm{P}_4}{2}\big)+1}{2}. 
\end{equation}
Clearly, $\dim J_{1,2}(2^k N)$ is independent of $k$. It is not difficult to recover Proposition \ref{prop:index-2} from the dimension formula. In particular, for any primes $p=3\bmod 4$ and $q=1\bmod 4$ and any non-negative integers $a$, $b$ and $k$, we have
\begin{equation}
\dim J_{1,2}(2^k p^a q^b) = \max\Big\{0, \Big[ \frac{a-1}{2} \Big]\Big\}  \Big( \max\Big\{0, \Big[ \frac{b-1}{2} \Big]\Big\} \frac{\big( \frac{-pq}{2}\big)+1}{2} + \Big[ \frac{b+2}{2} \Big] \frac{\big( \frac{-p}{2}\big)+1}{2} \Big). 
\end{equation}

Our second theorem gives a formula of $\dim J_{1,m}(N)$ for squarefree $m$.

\begin{theorem}\label{th:m_squarefree}
Let $m$ be a positive odd squarefree integer and $N, N'$ be positive odd integers with $\mathrm{gcd}(m,N')=1$. Suppose that every prime divisor of $N$ divides $m$. We set
\begin{align*}
\mathrm{P}^{m}:=&\big\{\text{prime}\; p: \; p\mid m\big\}, \\
\mathrm{P}_0^{m}:=&\big\{\text{prime}\; p:\; p\mid m,\; n_p^N=0 \big\}, \\
\mathrm{P}_{>0}^{m}:=&\big\{\text{prime}\; p:\; p\mid m,\; n_p^N>0 \big\}, \\
\mathrm{P}_2^{N'}:=&\big\{\text{prime}\; p: \; n_p^{N'}=2\big\}, \\
\mathrm{P}_{>2}^{N',1}:=&\big\{\text{prime}\; p: \; n_p^{N'} \geq 3,\; p=1\bmod 4\big\}, \\
\mathrm{P}_{>2}^{N',3}:=&\big\{\text{prime}\; p: \; n_p^{N'} \geq 3,\; p=3\bmod 4\big\}, \\
\mathrm{P}_{>1}^{N,1}:=&\big\{\text{prime}\; p: \; n_p^{N} \geq 2,\; p=1\bmod 4\big\}, \\
\mathrm{P}_{>1}^{N,3}:=&\big\{\text{prime}\; p: \; n_p^{N} \geq 2,\; p=3\bmod 4\big\}.     
\end{align*}
For $j\in\{0,1\}$ and any non-negative integer $k$, 
\begin{equation}
\dim J_{1,2^jm}(2^kNN') = 2^{|\mathrm{P}_2^{N'}|}\left(  I_1+I_2+I_3 +I_4\right).
\end{equation}
For $1\leq i\leq 4$, the number $I_i$ is defined by
\begin{equation}\label{eq:I-square-free}
\begin{split}
I_i=&\sum_{\mathrm{S}_{>2,i}^{N',1}}\sum_{\mathrm{S}_{>2,i}^{N',3}} \sum_{\mathrm{S}_{>1,i}^{N,1}} \sum_{\mathrm{S}_{>1,i}^{N,3}} \sum_{\mathrm{S}^m_i} \prod_{p \in \mathrm{S}_{>2,i}^{N',3}\cup \mathrm{S}_{>2,i}^{N',1}} \left[\frac{n_p^{N'}-1}{2}\right] \prod_{p \in \big(\mathrm{P}_{>2}^{N',3}\cup \mathrm{P}_{>2}^{N',1}\big) \backslash \big(\mathrm{S}_{>2,i}^{N',3}\cup \mathrm{S}_{>2,i}^{N',1}\big)} \left[\frac{n_p^{N'}+2}{2}\right]\\ 
& \qquad \qquad \prod_{p\in \mathrm{S}_{>1,i}^{N,1}\cup \mathrm{S}_{>1,i}^{N,3}} \left[\frac{n_p^N}{2}\right]\prod_{p\in \big(\mathrm{P}_{>0}^m\backslash \mathrm{S}^m_i\big)\backslash \big(\mathrm{S}_{>1,i}^{N,3}\cup \mathrm{S}_{>1,i}^{N,1}\big)}\left(\left[\frac{n_p^N+1}{2}\right]+\frac{(\frac{\mathrm{P}_i}{p})+1}{2}  \right)\cdot \mathrm{K}_i^j.   
\end{split}    
\end{equation}
The above symbols are explained as follows:
\begin{enumerate}
\item  $\mathrm{S}_{>2,i}^{N',1}$, $\mathrm{S}_{>2,i}^{N',3}$ and $\mathrm{S}_{>1,i}^{N,1}$ run through all the subsets of $\mathrm{P}_{>2}^{N',1}$, $\mathrm{P}_{>2}^{N',3}$ and $\mathrm{P}_{>1}^{N,1}$, respectively.

\item $\mathrm{S}_{>1,i}^{N,3}$ runs through all the subsets of $\mathrm{P}_{>1}^{N,3}$ such that $f(\mathrm{P}_i;\mathrm{P}_0^m)=1$, where
$$
\mathrm{P}_i=-\prod_{p\in \mathrm{S}_{>2,i}^{N',1}\cup \mathrm{S}_{>2,i}^{N',3}}p\cdot \prod_{q\in \mathrm{S}_{>1,i}^{N,1}\cup  \mathrm{S}_{>1,i}^{N,3}}q 
$$
for $i\in\{1,2,4\}$, and
$$
\mathrm{P}_3=-2\cdot \prod_{p\in \mathrm{S}_{>2,i}^{N',1}\cup \mathrm{S}_{>2,i}^{N',3}}p\cdot \prod_{q\in \mathrm{S}_{>1,i}^{N,1}\cup  \mathrm{S}_{>1,i}^{N,3}}q.
$$
In the above, we define the product to be $1$ if the associated set is empty. 
We further require that $|\mathrm{S}_{>2,i}^{N',3}|+|\mathrm{S}_{>1,i}^{N,3}|$ is odd for $i\in\{1,4\}$, and $|\mathrm{S}_{>2,2}^{N',3}|+|\mathrm{S}_{>1,2}^{N,3}|$ is even.

\item $\mathrm{S}^m_i$ runs through all the subsets of $\mathrm{P}_{>0}^m\backslash \big(\mathrm{S}_{>1,i}^{N,1}\cup \mathrm{S}_{>1,i}^{N,3}\big)$ such that $f(\mathrm{P}_i;\mathrm{S}^m_i)=1$. 

\item Assume that $|\mathrm{P}_0^m|>0$. Then 
\begin{align*}
&\mathrm{K}_1^0=2^{|\mathrm{P}_0^m|-1}\Big(1+\Big[\frac{k}{2}\Big]\Big),& &\mathrm{K}_2^0=2^{|\mathrm{P}_0^m|-1}\mathrm{max}\Big\{0,\Big[\frac{k-4}{2}\Big]\Big\},&\\
&\mathrm{K}_3^0=2^{|\mathrm{P}_0^m|-1}\mathrm{max}\Big\{0,\Big[\frac{k-7}{2}\Big]\Big\},& &\mathrm{K}_4^0=0,  & \\
&\mathrm{K}_1^1=2^{|\mathrm{P}_0^m|-1}\Big(\Big[\frac{k+1}{2}\Big]+\frac{(\frac{\mathrm{P}_1}{2})+1}{2}\Big),&  
&\mathrm{K}_2^1=2^{|\mathrm{P}_0^m|-1}\mathrm{max}\Big\{ 0, \Big[\frac{k-3}{2}\Big]\Big\},& \\ 
&\mathrm{K}_3^1=2^{|\mathrm{P}_0^m|-1}\mathrm{max}\Big\{0,\Big[\frac{k-6}{2} \Big]\Big\},& 
&\mathrm{K}_4^1=2^{|\mathrm{P}_0^m|-1} \frac{(\frac{\mathrm{P}_4}{2})+1}{2}.&
\end{align*}

\item Assume that $|\mathrm{P}_0^m|=0$. Then we double $\mathrm{K}_i^j$ and further require $|\mathrm{S}^m_i|$ to be odd for $i\in\{1,2,3\}$ and $|\mathrm{S}^m_4|$ to be even.

\item For each of the four products in \eqref{eq:I-square-free}, we define it to be $1$ if the corresponding set is empty, and define it to be $0$ if the corresponding set cannot be selected to satisfy the above conditions. 

\item $[x]$ denotes the integer part of $x$ and $\big( \frac{n}{m}\big)$ denotes the Kronecker symbol as before. 
\end{enumerate}
\end{theorem}

\begin{proof}
The proof is essentially the same as that of Theorem \ref{th:m_N_COPRIME}, expect that the $p$-part has a larger range. We formulate the only possible non-trivial $p$-parts in Table \ref{tab:m-square-fee} below. 

\begin{table}[ht]
\caption{The only possible non-trivial $p$-parts of $J_{1,2^j m}(2^kNN')$}\label{tab:m-square-fee}
\renewcommand\arraystretch{1.5}
\noindent\[
\begin{array}{|c|c|c|c|c|}
\hline p & n & \text{$p$-part} & \text{Condition} & \dim \\
\hline \mathrm{S}^{N',3}_{>2,i} & \big[\frac{n_p^{N'}-1}{2}\big] & (\CC L_{p^{2n+1}}^-(a_p'))^{\Gamma_0(p^{2n+1})} & - & n\\
\hline \mathrm{S}^{N',1}_{>2,i} & \big[\frac{n_p^{N'}-1}{2}\big] & (\CC L_{p^{2n+1}}^+(a_p'))^{\Gamma_0(p^{2n+1})} & - & n\\
\hline \big(\mathrm{P}_{>2}^{N',3}\cup \mathrm{P}_{>2}^{N',1}\big) \backslash \big(\mathrm{S}_{>2,i}^{N',3}\cup \mathrm{S}_{>2,i}^{N',1}\big) & \big[\frac{n_p^{N'}}{2}\big] & (\CC L_{p^{2n}}^+(a_p'))^{\Gamma_0(p^{2n})} & - & 1+n \\
\hline \mathrm{S}_i^m & \big[\frac{n_p^{N}-1}{2}\big] & (\CC L_p^-(a_p)\otimes\CC L_{p^{2n+1}}^-(a_p'))^{\Gamma_0(p^{2n+1})} & \big( \frac{-a_pa_p'}{p} \big)=1 & 1 \\
\hline \mathrm{P}_0^m & - & (\CC L_p^+(a_p)\otimes\CC L_p^+(a_p'))^{\SL_2(\ZZ)} & \big( \frac{-a_pa_p'}{p} \big)=1 & 1 \\
\mathrm{P}_0^m & - & (\CC L_p^-(a_p)\otimes\CC L_p^-(a_p'))^{\SL_2(\ZZ)} & \big( \frac{-a_pa_p'}{p} \big)=1 & 1 \\
\hline \mathrm{S}_{>1,i}^{N,3} & \big[\frac{n_p^{N}}{2}\big] & (\CC L_p^+(a_p)\otimes\CC L_{p^{2n}}^-(a_p'))^{\Gamma_0(p^{2n})} & - & n \\
\hline \mathrm{S}_{>1,i}^{N,1} & \big[\frac{n_p^{N}}{2}\big] & (\CC L_p^+(a_p)\otimes\CC L_{p^{2n}}^+(a_p'))^{\Gamma_0(p^{2n})} & - & n \\
\hline  \big(\mathrm{P}_{>0}^m\backslash \mathrm{S}^m_i\big)\backslash \big(\mathrm{S}_{>1,i}^{N,3}\cup \mathrm{S}_{>1,i}^{N,1}\big) & \big[\frac{n_p^{N}-1}{2}\big] &  (\CC L_p^+(a_p)\otimes\CC L_{p^{2n+1}}^+(a_p'))^{\Gamma_0(p^{2n+1})} & \big( \frac{-a_pa_p'}{p} \big)=1 & 2+n \\
\hline \big(\mathrm{P}_{>0}^m\backslash \mathrm{S}^m_i\big)\backslash \big(\mathrm{S}_{>1,i}^{N,3}\cup \mathrm{S}_{>1,i}^{N,1}\big) & \big[\frac{n_p^{N}-1}{2}\big] & (\CC L_p^+(a_p)\otimes\CC L_{p^{2n+1}}^+(a_p'))^{\Gamma_0(p^{2n+1})} & \big( \frac{-a_pa_p'}{p} \big)\neq 1 & 1+n \\
\hline 
\end{array} 
\]
\end{table}
We only establish the formula of $\dim J_{1,2m}(2^k NN')$, since the other case is similar. Suppose that $V$ is a non-trivial constitute of $J_{1,2m}(2^kNN')$. The $2$-part of $V$ is the same as the case of $J_{1,2m}(2^k N)$ in Theorem \ref{th:m_N_COPRIME}. 

Assume that the $2$-part of $V$ is $(\CC D_2^+(a_2)\otimes \CC D_{2^{2n+1}}^+(a_2'))^{\Gamma_0(2^{2n+3})}$ with $n=[(k-3)/2]$. By Table \ref{tab:2m}, its dimension is $[(k+3)/2]$ if $a_2a_2'=7 \bmod 8$, and $[(k+1)/2]$ if $a_2a_2'=3 \bmod 8$. In this case, we set $V_1=V$ and $i=1$. The number of “$-$" among all signs $\epsilon_p$ of $\CC L_{p}^{\epsilon_p}(a_p)$ for $p\mid m$ is odd, which yields that $|\mathrm{S}_1^m|$ is odd if $|\mathrm{P}_0^m|$ is empty. One is free to choose the sign of the $p$-part of $V_1$ for any $p\in \mathrm{P}^m_0$, and the only constrain is that the number of “$-$" among all signs of the $p$-parts of $V_1$ for $p\in \mathrm{P}^m_0$ is odd if $|\mathrm{S}^m_{1}|$ is even, and is even if $|\mathrm{S}^m_{1}|$ is odd. This contributes to the additional factor $2^{|\mathrm{P}_0^m|-1}$ in $\mathrm{K}_1^0$ if $|\mathrm{P}_0^m|>0$. The number of “$-$" among all signs $\epsilon'_p$ of $\CC L_{p^{k_p}}^{\epsilon_p'}(a'_p)$ for $p\mid mN'$ is even. Thus $|\mathrm{S}_{>2,1}^{N',3}|+|\mathrm{S}_{>1,1}^{N,3}|$ is odd. For any $p\in \mathrm{S}_1^m$, $\big( \frac{-a_pa_p'}{p} \big)=1$ if and only if $\big( \frac{\mathrm{P}_1}{p} \big)=1$, which leads to the restriction on $\mathrm{S}_1^m$. For any $p\in\mathrm{P}_0^m$, $\big( \frac{-a_pa_p'}{p} \big)=1$ if and only if $\big( \frac{\mathrm{P}_1}{p} \big)=1$, which yields the restriction on $\mathrm{S}_{>1,1}^{N,3}$. For any $p\in \big(\mathrm{P}_{>0}^m\backslash \mathrm{S}^m_i\big)\backslash \big(\mathrm{S}_{>1,i}^{N,3}\cup \mathrm{S}_{>1,i}^{N,1}\big)$, $\big( \frac{-a_pa_p'}{p} \big)=1$ if and only if $\big( \frac{\mathrm{P}_1}{p} \big)=1$, which is compatible with the last product in \eqref{eq:I-square-free}. Note that $\mathrm{P}_1=1\bmod 4$. For the $2$-part, $a_2a_2'=7 \bmod 8$ if and only if $\left( \frac{\mathrm{P}_1}{2} \right)=1$, and $a_2a_2'=3 \bmod 8$ if and only if $\left(\frac{\mathrm{P}_1}{2} \right)=-1$. This contributes to $\mathrm{K}^1_1$. We then prove that $\dim V_1 = 2^{|\mathrm{P}_2^{N'}|} I_1$. 

Similarly, when the $2$-part of $V$ is of type (b), (c) or (d) in the proof of Theorem \ref{th:m_N_COPRIME}, by a similar argument, we show that $\dim V_i = 2^{|\mathrm{P}_2^{N'}|} I_i$ for $i=2,3,4$ respectively. We then deduce the dimension formula of $J_{1,2m}(2^k NN')$. 
\end{proof}

As an application, we consider one particular case. Let $m=p$ be an odd prime and $j=0$. In this case, $\mathrm{S}_{>1,i}^{N,3}\cup S_{>1,i}^{N,1}=\emptyset$ for $1\leq i \leq 3$, since $|\mathrm{S}^m_i|$ is odd if $|\mathrm{P}_0^m|=0$. Thus $I_i$ can be simplified as 
$$ 
\sum_{\mathrm{S}_{>2,i}^{N',1}}\sum_{\mathrm{S}_{>2,i}^{N',3}} \prod_{q \in \mathrm{S}_{>2,i}^{N',3}\cup \mathrm{S}_{>2,i}^{N',1}} \left[\frac{n_q^{N'}-1}{2}\right] \prod_{q \in \big(\mathrm{P}_{>2}^{N',3}\cup \mathrm{P}_{>2}^{N',1}\big) \backslash \big(\mathrm{S}_{>2,i}^{N',3}\cup  \mathrm{S}_{>2,i}^{N',1}\big)}\left[\frac{n_q^{N'}+2}{2}\right] \cdot \mathrm{K}^0_i.
$$
Moreover, it is easy to check that $\dim J_{1,p}(2^kp^aN')$ is independent of the non-negative integer $a$. When $N'=1$, we obtain
\begin{equation}\label{eq:dim-index-p}
\dim J_{1,p}(2^k) = \frac{\big( \frac{-1}{p} \big)+1}{2} \mathrm{max}\Big\{0, \Big[ \frac{k-4}{2} \Big]\Big\}  + \frac{\big( \frac{-2}{p} \big)+1}{2} \mathrm{max}\Big\{0, \Big[ \frac{k-7}{2} \Big]\Big\}.
\end{equation}

Similarly, we derive that 
\begin{equation}
\begin{split}
\dim J_{1,2p}(2^kp^a) =& \frac{\big( \frac{-1}{p} \big)+1}{2} \mathrm{max}\Big\{0, \Big[ \frac{k-3}{2} \Big]\Big\}  + \frac{\big( \frac{-2}{p} \big)+1}{2} \mathrm{max}\Big\{0, \Big[ \frac{k-6}{2} \Big]\Big\} \\
&+ \frac{\big(\big( \frac{-p}{2} \big)+1\big)\big(1-\big(\frac{-1}{p} \big)\big)}{4} \Big[ \frac{a}{2} \Big]
\end{split}     
\end{equation}
for any non-negative integers $k$ and $a$. In particular, $\dim J_{1,6}(2^k 3^a)=\max\{0, [k/2]-3\}$, which yields \cite[Lemma 3.13]{CDH18}, that is, $J_{1,6}(36)=0$. It is also easy to recover \cite[Lemma 3.14]{CDH18}, that is, $J_{1,30}(36)=0$, by the above dimension formula. 

In Table \ref{tab:jacobi_dimensions_final_extended}, we formulate the dimensions computed for $\dim J_{1,m}(N)$ for a range of indices no greater than $50$ and some explicit levels. These values are direct outputs of the formulas developed in Theorem \ref{th:m_N_COPRIME} and Theorem \ref{th:m_squarefree}.

This paper also provides methods for calculating the dimension in some other cases. For example, Lemma \ref{lem:J_1,P2(P2)} proves that for any prime number $p=3\bmod 4$, $\dim J_{1,p^2}(p^2)=1$. This case is not covered in Theorem \ref{th:m_N_COPRIME} and Theorem \ref{th:m_squarefree}, but can be obtained using the constructive method in this paper.

\begin{table}[p]
\centering
\footnotesize
\caption{Dimension of the space of Jacobi forms $J_{1, m}(N)$ of weight one for specific levels and indices $m \le 50$. A hyphen (-) indicates that the given parameters do not satisfy the conditions of Theorems \ref{th:m_N_COPRIME} and \ref{th:m_squarefree}.}
\label{tab:jacobi_dimensions_final_extended}
\resizebox{\textwidth}{!}{%
\begin{tabular}{|c|c|c|c|c|c|c|c|c|}
\hline
\textbf{Index $m$} & \textbf{$N=3^3$} & \textbf{$N=7^3$} & \textbf{$N=11^3$} & \textbf{$N=2^9$} & \textbf{$N=2^{10}$} & \textbf{$N=2^6 \cdot 3^3$} & \textbf{$N=2^6 \cdot 7^3$} & \textbf{$N=2^6 \cdot 3^3 \cdot 5^3 \cdot 7^3$} \\
\hline
1 & 0 & 0 & 0 & 0 & 0 & 0 & 0 & 0 \\
2 & 0 & 1 & 0 & 0 & 0 & 0 & 1 & 6 \\
3 & 0 & 0 & 1 & 1 & 1 & 0 & 0 & 6 \\
4 & 0 & 1 & 0 & - & - & - & - & - \\
5 & 0 & 0 & 1 & 2 & 3 & 2 & 2 & 5 \\
6 & 0 & 0 & 0 & 1 & 2 & 0 & 0 & 11 \\
7 & 1 & 0 & 0 & 0 & 0 & 4 & 0 & 10 \\
8 & 0 & 2 & 0 & - & - & - & - & - \\
9 & - & 0 & 1 & 1 & 1 & - & 0 & - \\
10 & 0 & 0 & 0 & 3 & 3 & 2 & 2 & 11 \\
11 & 0 & 1 & 0 & 1 & 1 & 0 & 4 & 27 \\
12 & - & 0 & 1 & - & - & - & - & - \\
13 & 1 & 0 & 0 & 2 & 3 & 6 & 2 & 33 \\
14 & 0 & 1 & 0 & 0 & 0 & 3 & 1 & 14 \\
15 & 0 & 0 & 2 & 0 & 0 & 2 & 0 & 11 \\
16 & 0 & 2 & 0 & - & - & - & - & - \\
17 & 0 & 0 & 0 & 3 & 4 & 2 & 2 & 26 \\
18 & - & 1 & 0 & 1 & 2 & - & 1 & - \\
19 & 1 & 0 & 0 & 1 & 1 & 4 & 0 & 27 \\
20 & 0 & 0 & 1 & - & - & - & - & - \\
21 & 1 & 0 & 0 & 0 & 0 & 4 & 0 & 18 \\
22 & 0 & 2 & 0 & 1 & 2 & 0 & 5 & 29 \\
23 & 0 & 1 & 1 & 0 & 0 & 0 & 4 & 30 \\
24 & - & 0 & 0 & - & - & - & - & - \\
25 & 0 & 0 & 1 & 2 & 3 & 2 & 2 & - \\
26 & 0 & 0 & 0 & 3 & 3 & 5 & 2 & 27 \\
27 & - & 0 & 2 & 2 & 2 & - & 0 & - \\
28 & 1 & - & 0 & - & - & - & - & - \\
29 & 0 & 1 & 0 & 2 & 3 & 2 & 6 & 36 \\
30 & 0 & 0 & 0 & 0 & 0 & 2 & 0 & 16 \\
31 & 1 & 0 & 1 & 0 & 0 & 4 & 0 & 27 \\
32 & 0 & 3 & 0 & - & - & - & - & - \\
33 & 0 & 0 & 1 & 2 & 2 & 0 & 0 & 35 \\
34 & 0 & 0 & 0 & 4 & 5 & 2 & 2 & 26 \\
35 & 0 & 0 & 0 & 0 & 0 & 0 & 2 & 15 \\
36 & - & 1 & 1 & - & - & - & - & - \\
37 & 1 & 1 & 1 & 2 & 3 & 6 & 6 & 42 \\
38 & 0 & 0 & 0 & 1 & 2 & 3 & 0 & 25 \\
39 & 1 & 0 & 0 & 0 & 0 & 6 & 0 & 41 \\
40 & 0 & 0 & 0 & - & - & - & - & - \\
41 & 0 & 0 & 0 & 3 & 4 & 2 & 2 & 15 \\
42 & 0 & 0 & 0 & 0 & 0 & 3 & 0 & 21 \\
43 & 1 & 1 & 0 & 1 & 1 & 4 & 4 & 37 \\
44 & 0 & 3 & - & - & - & - & - & - \\
45 & - & 0 & 3 & 2 & 3 & - & 2 & - \\
46 & 0 & 2 & 0 & 0 & 0 & 0 & 5 & 36 \\
47 & 0 & 0 & 1 & 0 & 0 & 0 & 0 & 21 \\
48 & - & 0 & 1 & - & - & - & - & - \\
49 & 1 & - & 0 & 0 & 0 & 4 & - & - \\
50 & 0 & 1 & 0 & 3 & 3 & 2 & 3 & - \\
\hline
\end{tabular}
}
\end{table}

\newpage

\bigskip
\noindent
\textbf{Acknowledgements} 
H. Wang thanks Shaoyun Yi for valuable discussions. The authors sincerely thank the referee for careful reading and valuable comments.

\bibliographystyle{plainnat}
\bibliofont
\bibliography{refs}

\end{document}